\documentclass[12pt,twoside]{article}
\def\<#1>{\mathinner{\langle#1\rangle}}
\usepackage{amsmath,amsfonts,amsthm,amssymb,mathrsfs}
\usepackage{comment}
\usepackage{enumitem}
\usepackage{fancyhdr}
\usepackage{tikz}
\usepackage{titling}

\makeatletter
\renewcommand{\dddot}[1]{%
  {\mathop{\kern\z@#1}\limits^{\vbox to-1.4\ex@{\kern-\tw@\ex@
   \hbox{\normalfont ...}\vss}}}}
\usepackage[
backend=biber,
style=alphabetic, 
]{biblatex}
\usepackage{subcaption}
\usepackage[top=2cm,bottom=2cm,left=2cm,right=2cm,asymmetric]{geometry}

\usepackage{times}
\usepackage[normalem]{ulem}
\usepackage{color}
\theoremstyle{plain}
\newtheorem{theorem}{Theorem}[section]
\newtheorem{lemma}[theorem]{Lemma}
\newtheorem{lem}[theorem]{Lemma}
\newtheorem{prop}[theorem]{Proposition}
\newtheorem{cor}[theorem]{Corollary}
\theoremstyle{definition}

\newtheorem{defn}[theorem]{Definition}
\theoremstyle{remark}
\newtheorem*{remark}{Remark}
\newtheorem*{rmk}{Remark}

\newcommand{\vphi}{\varphi}

\newcommand{\pl}{\partial}

\newcommand{\na}{\nabla}
\newcommand{\lt}{\left}
\newcommand{\rt}{\right}
\newcommand{\rw}{\rightarrow}
\newcommand{\R}{\mathbb{R}}
\newcommand{\cu}{\check{u}}
\newcommand{\cV}{\check{V}}
\newcommand{\mr}{\mathbf{r}}

\renewcommand{\tilde}{\widetilde}

\newcommand{\tr}{\mbox{tr}}

\numberwithin{equation}{section}
\title{Quasi-Spherical Metrics and the Static Minkowski Inequality}
\author{Brian Harvie and Ye-Kai Wang}
\date{}
\begin{document}
\pagestyle{fancy}
\fancyhf{} 
\renewcommand{\headrulewidth}{0pt}
\fancyhf[EHC]{B. Harvie \& Y.K. Wang}
\fancyhf[OHC]{Quasi-Spherical Metrics and the Static Minkowski Inequality}
\fancyhf[FC]{\thepage}

\maketitle
\vspace{-2cm}
\begin{abstract}
We prove that equality in the Minkowski inequality for asymptotically flat static manifolds from \cite{Mc} is achieved only by slices of Schwarzschild space. In order to show this, we prove that a static vacuum metric on $\mathbb{S}^{n-1} \times (r_{0},\infty)$ of the form $g= u^{2}(\theta,r)dr^{2} + r^{2}g_{S^{n-1}}$ belongs to the Schwarzschild family-- that is, we establish uniqueness of \textit{quasi-spherical} static metrics. Additionally, we strengthen the static Minkowski inequality to hold in all dimensions and under relaxed asymptotic assumptions.

Using the complete rigidity statement, we further develop the approach to static uniqueness problems in general relativity from \cite{HW}. Through this approach, we obtain strengthened versions of the higher-dimensional black hole uniqueness theorem from \cite{AM} and the photon surface and static metric extension uniqueness theorems from \cite{HW}. Finally, as a notable by-product of our analysis, we establish regularity of weak inverse mean curvature flow in asymptotically flat manifolds-- that is, a weak IMCF eventually becomes smooth in any asymptotically flat background, as was suggested by Huisken-Ilmanen in \cite{HI2}.

\end{abstract}

\tableofcontents

\section{Introduction}
A Riemannian manifold $(M^{n},g)$ is \textit{static} if it admits a positive solution $V \in C^{\infty}_{+}(M^{n})$ to the system of equations \footnote{Solutions are also taken to smoothly extend to $\partial M$ if non-empty, so in particular $V|_{\partial M} \geq 0$.}
\begin{eqnarray}
    \text{Hess}_{g} V &=& V \text{Ric}_{g}, \label{static} \\
    \Delta_{g} V &=& 0. \nonumber
\end{eqnarray}
 The study of static manifolds is strongly motivated by general relativity-- one may check that $V \in C^{\infty}_{+}(M)$ solves the system \eqref{static} on $(M^{n},g)$ if and only if the Lorentzian metric $\widehat{g}=-V^{2}(x)d\tau^{2} + g$ on $M^{n} \times \mathbb{R}$ is Ricci-flat and hence solves the Einstein equations in vacuum. The most familiar example of a static manifold is Schwarzschild space, a rotationally symmetric Riemannian manifold depending on a parameter $m \in \mathbb{R}$ and given by 

\begin{eqnarray} \label{schwarzschild}
    M^{n}&=& \mathbb{S}^{n-1} \times (r_{m},\infty), \hspace{0.25cm} g_{m}= \left(1-\frac{2m}{r^{n-2}}\right)^{-1} dr^{2} + r^{2}g_{\mathbb{S}^{n-1}}, \hspace{0.25cm} V_{m}(r)= \sqrt{1 - \frac{2m}{r}},
\end{eqnarray}    
 where $r_{m}^{n-2}=\max\{2m, 0\}$. This corresponds to Schwarzschild spacetime one dimension higher.

In this paper, we say that a static triple $(M^{n},g,V)$ is \textit{asymptotically flat} if there exists a compact set $K \subset M$ and a diffeomorphism $F: \mathbb{R}^{n} \setminus B_{1}(0) \rightarrow M^{n} \setminus K$ such that in the coordinates $(x^{1},\dots,x^{n})$ the pull-back metric $g_{ij}$ and potential $V$ satisfy

    \begin{eqnarray}
     g_{ij}(x) &=& \delta_{ij} + o_{2} (1) \hspace{2cm} i,j \in \{ 1, \dots, n \} , \label{af} \\
       V(x) &=& 1 + o_{2}(1), \label{V_decay}
    \end{eqnarray}
as $r \rightarrow \infty$, where  $r= \sqrt{(x^{1})^{2} + \dots + (x^{n})^{2}}$ and $\delta_{ij}$ is the Euclidean metric on $\mathbb{R}^{n} \setminus B_{1}(0)$. We will refer to $M \setminus K$ as the {\it exterior region}. Here, the decay assumption on the metric $g_{ij}$ is weaker than the more standard decay rate of $o_{2}(r^{\frac{n-2}{2}})$ and is insufficient to define the ADM mass of $(M^{n},g)$. However, in the spirit of Smarr \cite{Smarr}, we will refer to the ``mass" of a static triple $(M^{n},g,V)$ with boundary $\partial M$ as the integral quantity

\begin{equation} \label{mass}
    m = \frac{1}{(n-2)w_{n-1}} \int_{\partial M} \frac{\partial V}{\partial \nu} d\sigma.
\end{equation}
We are interested in a Minkowski-type inequality for asymptotically flat static manifolds from \cite{Mc}. We state a refined version of the main theorem in that paper here.

\begin{theorem}[The Static Minkowski Inequality] \label{Mccormick}
Let $(M^{n},g,V)$ be asymptotically flat and let $\Sigma^{n-1} \subset \partial M$ be an outer-minimizing boundary component of $M$. Assume that

\begin{itemize}
\item $n > 7$ and $\partial M = \Sigma$ is connected, OR 
\item $3 \leq n \leq 7$, and $\partial M \setminus \Sigma$, if non-empty, consists of closed minimal hypersurfaces.
\end{itemize}
Then we have the Minkowski-type inequality  
\begin{equation} \label{Minkowski_static}
   \frac{1}{(n-1)w_{n-1}} \int_{\Sigma} VH d\sigma + 2m \geq \left( \frac{|\Sigma|}{w_{n-1}} \right)^{\frac{n-2}{n-1}},
\end{equation}
on the surface $\Sigma^{n-1}$.
\end{theorem}
Brendle-Hung-Wang first proved the analogous inequality for star-shaped hypersurfaces in anti-de-Sitter Schwarzschild space in their breakthrough paper \cite{BHW}. Their proof utilizes smooth solutions of inverse mean curvature flow (IMCF), and equality holds only on slices $\Sigma = \{ r= r_{0} \}$ of ADS-Schwarzschild. Later, Wei \cite{Wei} proved inequality \eqref{Minkowski_static} for outer-minimizing hypersurfaces in Schwarzschild space using weak solutions of IMCF, with equality once again being achieved only by slices. In \cite{Mc}, McCormick applied Wei's analysis to obtain \eqref{Minkowski_static} for asymptotitcally flat static $(M^{n},g,V)$ of order $q > \frac{n-2}{2}$ and dimension $n < 8$. Theorem \ref{Mccormick} removes the conditions on $n$ (for a connected boundary) and on $q$ from McCormick's paper. Both of these improvements follow from a straight-forward generalization of the analysis of Wei \cite{Wei}, and because of this we will leave the proof of Theorem \ref{Mccormick} for the appendix.

The main focus of this paper is to characterize equality in inequality \eqref{Minkowski_static}, which follows our earlier work on this problem in \cite{HW}. Equality holds in \eqref{Minkowski_static} when $\Sigma= \{ r = r_{0} \}$ is a slice in Schwarzschild space of any mass $m$, but whether this is the only situation where equality holds is in fact a highly non-trivial question. In this paper, we will show that if $\Sigma \subset \partial M$ achieves equality in \eqref{Minkowski_static}, then outside of a compact set $K$ the manifold $(M^{n},g)$ has the form 
\begin{equation} \label{quasi-spherical}
    M^{n} \setminus K \cong \mathbb{S}^{n-1} \times (r_{0},\infty), \hspace{1cm} g= u^{2}(\theta,r) dr^{2} + r^{2}g_{S^{n-1}}, \hspace{1cm} u \in C^{\infty}_{+}(\mathbb{S}^{n-1} \times (r_{0},\infty)).
\end{equation} 
If the metric coefficient $u$ in \eqref{quasi-spherical} is independent of $\theta$, then the static equations immediately imply that $(M^{n},g,V)$ is a rotationally symmetric piece of Schwarzschild space \eqref{schwarzschild}. However, there is no a-priori reason that $u$ in \eqref{quasi-spherical} must be rotationally symmetric, at least without imposing some additional inner boundary condition. Metrics of the form \eqref{quasi-spherical} in fact have a venerable history in mathematical relativity, originally having been studied in seminal work of Bartnik \cite{B} and later Shi-Tam \cite{ST}. In the terminology of \cite{B}, metrics of the form \eqref{quasi-spherical} are \textit{quasi-spherical metrics} with shear vector $\beta_{A}=0$. Bartnik and Shi-Tam constructed many quasi-spherical metrics that are both asymptotically flat and scalar flat. Our main theorem states that the metric of Schwarzschild space is the only member of this family which is static.

\begin{theorem}[Uniqueness of Quasi-Spherical Static Metrics] \label{main}
Consider a Riemannian manifold
\begin{equation*}
(M^{n},g) \cong (\mathbb{S}^{n-1} \times (r_{0},\infty), u^{2}(\theta,r)dr^{2} + r^{2} g_{S^{n-1}}), \hspace{2cm} u \in C^{\infty}(\mathbb{S}^{n-1} \times (r_{0},\infty)).
\end{equation*}
Suppose that $(M^{n},g)$ admits a static potential $V \in C^{\infty}_{+}(M^{n})$. Then
\begin{eqnarray*}
u(\theta,r) &=& \frac{1}{\sqrt{1 - \frac{2m}{r^{n-2}}}}, \\
V(\theta,r) &=& \sqrt {1 - \frac{2m}{r^{n-2}}}, 
\end{eqnarray*}
for some $m \in \mathbb{R}$. That is, $(M^{n},g,V)$ is isometric to Schwarzschild space $(\mathbb{S}^{n-1} \times (r_{0},\infty), g_{m}, V_{m})$ from \eqref{schwarzschild} for some (not necessarily positive) mass $m$ and $r_{0} \geq r_{m}$.
\end{theorem}
\begin{remark}
The system of PDEs that the the static equations \eqref{static} impose on the metric coefficient $u$ and potential $V$ also admits solutions with $V$ not strictly positive. For example, one may choose $u=1$ so that $(M^{n},g) \cong (\mathbb{R}^{n+1} \setminus B_{r_{0}}(0), \delta)$ and $V=x^{i}$ to be a coordinate function. 
\end{remark}
Notice in the above theorem that no asymptotics are assumed for the metric and potential. Removing these conditions is not necessary to establish rigidity in Theorem \ref{Mccormick}. However, the stronger uniqueness statement for quasi-spherical metrics is interesting from a PDE perspective\footnote{Anderson \cite{Anderson00} (in 4-dimension) used collapsing theory and Bing-Long Chen \cite{Chen19} (in all dimensions) used gradient estimate to show that static spacetimes must have quadratic curvature decay. Reiris also studies extensively the asymptotics of static equations in dimension 3, see \cite{Reiris18} for example.}, as it holds without interior or asymptotic boundary conditions.  With Theorem \ref{main} in hand, we may fully characterize equality in inequality \eqref{Minkowski_static}. 

\begin{theorem}[Equality in the Static Minkowski Inequality] \label{equality case}
Equality holds in \eqref{Minkowski_static} if and only if $\Sigma$ is a slice $\{ r=r_* \}$ of Schwarzschild space \eqref{schwarzschild}.
\end{theorem}
\subsection{Applications of the main theorems}

As a consequence of Theorem \ref{equality case}, an asymptotically flat static $(M^{n},g,V)$ with connected, outer-minimizing boundary $\partial M=\Sigma^{n-1}$ is isometric to a piece of Schwarzschild space if and only if $\Sigma^{n-1}$ saturates inequality \eqref{Minkowski_static}. This rigidity criterion for Schwarzschild is robust in that it holds in all dimensions with mild asymptotics, is independent of the positive mass theorem, and is valid regardless of the specific inner boundary condition. We will show under several different geometric boundary conditions that the left-hand side of \eqref{Minkowski_static} can be controlled above in terms of boundary geometry. From this, we achieve saturation and hence uniqueness of Schwarzschild space in several contexts.

Our first application is to the problem of static black hole uniqueness. This problem asks if an asymptotically flat $(M^{n},g,V)$ with $V|_{\partial M} \equiv 0$ is isometric to Schwarzschild space, which physically means that Schwarzschild spacetime is the only (standard) static vacuum spacetime containing a black hole. Israel \cite{Israel} originally answered this question affirmatively for a connected boundary in $3$ dimensions, and Bunting-Massood ul Alam \cite{BMA} later showed this without a priori assuming $\partial M$ is connected. 

In higher dimensions, the known proofs of black hole uniqueness require additional assumptions to asymptotic flatness, see \cite{HI12} for a comprehensive review. A recent higher-dimensional result due to Agostiniani-Mazzieri \cite{AM} involves the Einstein-Hilbert energy of $\partial M=\Sigma$. The \textit{normalized Einstein-Hilbert energy} of an $(n-1)$-dimensional closed Riemannian manifold $(\Sigma^{n-1},\sigma)$ is the functional

\begin{equation} \label{einstein_hilbert}
   E(\Sigma,\sigma) = \left( |\Sigma| \right)^{\frac{3-n}{n-1}} \int_{\Sigma} R_{\sigma} d\sigma.
\end{equation}
Due to the landmark works of Schoen \cite{Schoen} and Aubin \cite{Aub} on the Yamabe problem, we know that every conformal class of metrics on $\Sigma$ contains a metric $\sigma$ satisfying $E(\Sigma,\sigma) \leq E(\mathbb{S}^{n-1},g_{S^{n-1}})$. This suggests a dichotomy between metrics with ``small" and ``large" Einstein-Hilbert energy. Indeed, \cite{AM} derives a lower bound on the boundary Einstein-Hilbert energy and a consequent uniqueness theorem for asymptotically flat static black holes with decay order $o_{2}(r^ \frac{2-n}{2})$. Here, we recover the same lower bound and the corresponding uniqueness theorem under the asymptotics \eqref{af}-\eqref{V_decay}. Using Proposition 5.1 of \cite{BM23}, we obtain that the horizon is automatically outer-minimizing, and so this assumption may be dropped in our theorem.

\begin{theorem}[Geometry and Uniqueness of Higher-Dimensional Static Vacuum Black Holes] \label{black_hole_uniqueness} 
Let $(M^{n},g,V)$ be asymptotically flat  with connected boundary $\partial M= \Sigma$ arising as a time slice of a Killing horizon-- that is, $V|_{\Sigma} \equiv 0$. Then
\begin{equation} \label{bh_bounds}
\left( \frac{|\Sigma|}{w_{n-1}} \right)^{\frac{n-2}{n-1}} \leq 2m \leq \frac{E(\Sigma,\sigma)}{E(\mathbb{S}^{n-1},g_{S^{n-1}})} \left( \frac{|\Sigma|}{w_{n-1}} \right)^{\frac{n-2}{n-1}},
\end{equation}
where $\sigma=g|_{\Sigma}$. In particular, $E(\Sigma,\sigma) \geq E(\mathbb{S}^{n-1},g_{S^{n-1}})$ with equality if and only if $(M^{n},g,V)$ is isometric to Schwarzschild space $(\mathbb{S}^{n-1} \times (r_{m},\infty), g_{m},V_{m})$ for some mass $m >0$.
\end{theorem}


Next, we mention some uniqueness theorems from \cite{HW} that may be upgraded because of the strengthened version of the Minkowski inequality. We first discuss the problem of photon surface uniqueness. A \textit{photon surface} is a timelike and totally umbilical hypersurface in a Lorentzian manifold. The condition of umbilicity is equivalent to the hypersurface being \textit{null totally geodesic}, meaning that every null geodesic initially tangent to the hypersurface remains tangent to it. For example, the surface $\{ r = (nm)^{\frac{1}{n-2}} \}$ in Schwarzschild spacetime traps null geodesics that are initially tangent to it. 

In static spacetimes it is physically natural to consider photon surfaces where the potential function $V$ is constant over each time slice, dubbed an \textit{equipotential photon surface} by Cederbaum-Galloway \cite{CG}. It $3+1$ dimensions, it is known that the Schwarzschild spacetime is the only asymptotically flat static vacuum spacetime containing an equipotential photon surface-- see \cite{Cederbaum}, \cite{BCC24}, \cite{Rau}, \cite{photon_sphere_positive_mass}, \cite{geometry_photon_surfaces},  \cite{CG}, and \cite{CCF24}. For the higher-dimensional case, we showed in Theorem 1.13 of \cite{HW} that outer-minimizing equipotential photon surfaces in dimension $3 \leq n \leq 7$ also respect the lower bound $E(\Sigma,\sigma) \geq E(\mathbb{S}^{n-1},g_{S^{n-1}})$ on \eqref{einstein_hilbert}, with equality implying uniqueness. Here, we generalize this previous result by extending to all dimensions and masses.

\begin{theorem}[Geometry and Uniqueness of Higher-Dimensional Equipotential Photon Surfaces] \label{photon_surface}
Let $(M^{n},g,V)$ be asymptotically flat with connected boundary $\partial M = \Sigma$ arising as a time slice of an equipotential photon surface. Assume that
\begin{itemize}
\item $m \geq 0$ and $\Sigma$ is outer-minimizing in $(M^{n},g)$, OR,
\item $m < 0$ and $\Sigma$ is outer-minimizing in $(M^{n},g_{-})$ for the conformal metric $g_{-}=V^{\frac{4}{n-2}} g$.
\end{itemize}
Then
\begin{equation} \label{photon_surface_bounds}
\left( \frac{|\Sigma|}{w_{n-1}} \right)^{\frac{n-2}{n-1}} \leq \frac{1}{(n-1)w_{n-1}} \int_{\Sigma} VH d\sigma + 2m \leq \frac{E(\Sigma,\sigma)}{E(\mathbb{S}^{n-1},g_{S^{n-1}})} \left( \frac{|\Sigma|}{w_{n-1}} \right)^{\frac{n-2}{n-1}}.
\end{equation}
In particular, $E(\Sigma,\sigma) \geq E(\mathbb{S}^{n-1},g_{S^{n-1}})$ with equality if and only if $(M^{n},g,V)$ is isometric to Schwarzschild space $(\mathbb{S}^{n-1} \times (r_{0},\infty),g_{m},V_{m})$ for some $r_{0} > r_{m}$.     
\end{theorem}
\begin{remark}
When $n=3$, the normalized Einstein-Hilbert energy reduces to $E(\Sigma,\sigma)= 2 \int_{\Sigma} K d\sigma = 4\pi \chi(\Sigma) \leq E(\mathbb{S}^{2},g_{S^{2}})$ as a consequence of the Gauss-Bonnet Thoerem. Therefore, we obtain the full photon surface and black hole uniqueness theorems for asymptotically flat static manifolds in $3$ dimensions as a corollaries of Theorems \ref{black_hole_uniqueness} and \ref{photon_surface}.
\end{remark}

\begin{remark}
Around the time of release of this paper, Cederbaum-Cogo-Leandro-Dos Santos \cite{CCLS24} derived similar black hole and equipotential photon surface uniqueness theorems based on a divergence identity approach which generalizes the one of Robinson \cite{R78}. If we are correctly interpreting their paper, the main differences appear to be that the authors in \cite{CCLS24} assume a first-order expansion near infinity for the static potential but do not make an outer-minimizing assumption when $\Sigma$ is a cross section of an equipotential photon surface. The bounds on mass in \cite{CCLS24} also appear to be slightly different.
\end{remark}
Finally, we address the uniqueness of static metric extensions, a problem which is motivated by R. Bartnik's proposal for a quasi-local mass-- see \cite{Bar03} for a review. Consider an asymptotically flat manifold $(M^{n},g,V)$ such that the induced metric $\sigma$ and mean curvature $H$ of the boundary $\Sigma$ correspond to that of a slice $\{ r = r_{0}\}$ of Schwarzschild space \eqref{schwarzschild} with mass $m \geq 0$. In \cite{HW}, we showed that $(M^{n},g,V) \cong (\mathbb{S}^{n-1} \times (r_{0}, \infty), g_{m},V_{m})$ in dimensions less than $8$ under the additional assumption that the first eigenvalue of the Jacobi operator of $\Sigma$ is bounded below by the corresponding first eigenvalue of $\{r = r_{0} \} \subset (\mathbb{S}^{n-1} \times (r_{m},\infty), g_{m})$. We called this condition ``Schwarzschild stability" in \cite{HW}. We remark that this stability condition agrees with an earlier result on this problem by Miao \cite{boundary_effects}. To conclude, we show static metric extension uniqueness for the larger class of static manifolds considered here.
\begin{theorem}[Uniqueness of Schwarzschild-Stable Static Metric Extensions] \label{static_extension}
Let $(M^{n},g,V)$ be asymptotically flat  with connected, outer-minimizing boundary $\partial M=\Sigma \cong \mathbb{S}^{n-1}$. Suppose that induced metric and mean curvature of $\Sigma$ satisfy

\begin{eqnarray}
    \sigma &=& r_{0}^{2} g_{S^{n-1}},  \\
    H_{\Sigma} &=& H_{0}, \hspace{1cm} 0 < H_{0} \leq (n-1) r_{0}^{-1}.
\end{eqnarray}
Suppose further that $\Sigma \subset (M^{n},g)$ is Schwarzschild stable in the sense of \cite{HW}, Definition 1.7. Then $(M^{n},g,V)$ is isometric to a piece of Schwarzschild space $(\mathbb{S}^{n-1} \times (r_{0},\infty), g_{m},V_{m})$ from \eqref{schwarzschild} for some $m \geq 0$.
\end{theorem}

\subsection{Methods and structure of the paper}
Sections 2 and 3 are dedicated to constructing a diffeomorphism $F: \mathbb{S}^{n-1} \times (r_{0},\infty) \rightarrow (M^{n} \setminus K, g)$ such that the pull-back metric is given by \eqref{quasi-spherical} in the equality case via smooth IMCF. In Section 2, we prove that weak IMCF is eventually smooth in an arbitrary asymptotically flat background by extending the analysis for star-shaped IMCFs in $\mathbb{R}^{n}$ from \cite{HI2}. To explain the idea of this smoothing result, recall that from Corollary 2.2 and Remark 2.8(a) of \cite{HI2},
IMCF exists as long as the mean curvature stays positive and bounded. 

In Euclidean space, the evolution of mean curvature is given by
\begin{align*}
\pl_t H = \frac{1}{H^2}\Delta H - 2 \frac{|\na H|^2}{H^3} - \frac{|A|^2}{H}.
\end{align*}
So $H$ remains bounded but may go to zero in finite time. The insight of \cite{imcf_starshaped} and \cite{Ur} is to consider star-shaped hypersurfaces, namely the support function $\langle X,\nu \rangle$ is positive everywhere. From the evolution equation 
\begin{align}\label{support}
\pl_t \langle X,\nu \rangle = \frac{1}{H^2}\Delta \langle X,\nu \rangle + \frac{|A|^2}{H^2} \langle X,\nu \rangle,
\end{align} 
one sees that star-shapedness is preserved and $H \langle X,\nu \rangle$ is bounded from below via maximum principle. Hence IMCF starting from a star-shaped hypersurface exists for all time. 

Brendle-Hung-Wang \cite{BHW} study IMCF in anti-de Sitter-Schwarzschild manifold with positive mass. They realize that star-shapedness is again preserved with the position vector $X$ replaced by the (radial) conformal Killing vector; moreover the term $-\frac{Ric(\nu,\nu)}{H}$ in the evolution equation
\begin{align*}
\pl_t H = \frac{1}{H^2}\Delta H - 2 \frac{|\na H|^2}{H^3} - \frac{|A|^2}{H} - \frac{Ric(\nu,\nu)}{H} 
\end{align*}
is positive if the mass is positive. Therefore, $H \langle X,\nu \rangle$ is bounded from below as in Euclidean space and the long-time existence of IMCF starting from star-shaped hypersurfaces follows.

Our idea of handling the Ricci curvature term with wrong sign is motivated by the derivation of the monotonicity formula in \cite{BHW}. Anti-de Sitter-Schwarzschild manifold is in fact a static manifold (with negative cosmological constant): there exists a positive function $f$ satisfying 
\[ \Delta_g f g - D^2 f + f Ric =0. \]
Using the relation
\begin{align*}
\Delta_g f = \Delta f + D^2 f (\nu,\nu) + H \frac{\pl f}{\pl \nu}, 
\end{align*}
Brendle-Hung-Wang rearrange 
\begin{align*}
\pl_t (fH) &= \frac{\pl f}{\pl\nu} + f \lt( -\Delta \frac{1}{H} - \frac{|A|^2}{H} - \frac{Ric(\nu,\nu)}{H} \rt) \\
&= 2 \frac{\pl f}{\pl\nu} + \frac{\Delta f}{H} + f \lt( -\Delta\frac{1}{H} - \frac{|A|^2}{H} \rt)
\end{align*}
so that the curvature term is eliminated. The lesson is that a weight function $f$ with $D^2 f -(\Delta_g f)g$ sufficiently positive can be used to handle the curvature term. It turns out $f = r^{-l}$ works for the asymptotically flat manifolds under consideration. The vector field $X = r \frac{\pl}{\pl r}$ is only approximately conformal Killing and the error results in a small amount of $|A|$ so that we end up using $f \langle X,\nu \rangle^2 H$ to gain enough positivity. Following \cite{HI2}, we estimate the sup-norm of $u = f^{-1} \langle X,\nu \rangle^{-2} H^{-1}$ to show that $H$ stays positive and the regularity follows. 

The smooth flow surfaces $\Sigma_{t}$ of IMCF are totally umbilical whenever $\Sigma$ achieves equality in \eqref{Minkowski_static}. In Section 3, we show that the blow-down lemma from \cite{HI} also implies that these $\Sigma_{t}$ are each intrinsically round. Furthermore, the foliation by smooth IMCF outside of a compact set gives us an optimal gauge to express the ambient metric in, and after a change of variable the metric takes the form \eqref{quasi-spherical}.

Next, we give the idea of the proof of Theorem \ref{main} in 3-dimension. The special form of the metric makes it feasible to study the static equations using spherical harmonics (Definition \ref{spherical harmonics}). {\it Assume} $u$ and $V$ admit a power series expansion in $\frac{1}{r}$:
\begin{align*}
u &= 1 + \frac{m}{r} + \sum_{k=2}^\infty \frac{u^{(-k)}}{r^k}\\
V &= 1 - \frac{m}{r} + \sum_{k=2}^\infty \frac{V^{(-k)}}{r^{k}}
\end{align*}
where $m \in \R$ and $u^{(-k)}, V^{(-k)}$ are functions on $S^2$. Let $\na$ and $\Delta$ denote the Levi-Civita connection and Laplacian with respect to $g_{S^2}$. The zero scalar curvature equation of $g$ depends only on $u$. Expanding it order by order, we see that $u^{(-2)}, u^{(-3)}$ are in $\mathcal{H}_{\ell\le 1}$ and 
\begin{align*}
(\Delta + 6) u^{(-4)} + 2 u^{(-2)}\Delta u^{(-2)} -3 (u^{-2)})^2 = f_1
\end{align*}
with $f_1 \in \mathcal{H}_{\ell\le 1}$. Hence $2u^{(-2)}\Delta u^{(-2)} - 3 (u^{(-2)})^2$ has no $\ell=2$ mode and it follows that $u^{(-2)}$ has no $\ell=1$ mode. Using the $rr$-component of the static equation, we solve for $V^{(-2)} = -\frac{1}{3} u^{(-2)}$. Consequently, $u$ and $V$ are rotationally symmetric up to the order $r^{-2}$.

Next, make the inductive hypothesis that $u$ and $V$ are rotationally symmetric up to the order of $r^{-k+1}$. From the previous step, $k \ge 3$. The $ra$-component of the static equation implies
\begin{align*}
-(k+1) V^{(-k)} =  u^{(-k)} + c_1
\end{align*}
for some constant $c_1$. The $ab$-component of the static equation implies
\begin{align*}
\na_a\na_b (V^{(-k)} + u^{(-k)}) = \lt( k V^{(-k)} - k u^{(-k)} + 2 u^{(-k)} + c_2 \rt) (g_{S^2})_{ab}
\end{align*}
for some constant $c_2$. As a result, $V^{(-k)} + u^{(-k)}$ is in $\mathcal{H}_{\ell\le 1}$ and
\begin{align*}
-V^{(-k)} - u^{(-k)} = k V^{(-k)} - k u^{(-k)} + 2 u^{(-k)} + c_3 
\end{align*}
for some constant $c_3$. Putting these together, we get $(2-k) u^{(-k)} = c_1 - c_3$ and hence $u^{(-k)}$ and $V^{(-k)}$ are both constant. By induction, $u$ and $V$ are rotationally symmetric to any order.

It turns out that the solution to the zero scalar curvature equation has an $r^{-4}\log r$ term so that a nonconstant $u^{(-2)}$ is allowed. We resort to the static equations to show that $u^{(-2)}$ and $V^{(-2)}$ are each constant functions. Finally, since we do not know that $u$ and $V$ are analytic at infinity in polar coordinates, we will instead turn the formal computation in the second step into an estimate, making full use of the static equations to show that the non-rotationally symmetric part of $u$ and $V$ must vanish.

The rest of the paper is organized as follows. In Section \ref{section:static equations in polar coordinates}, we derive the static equations in polar coordinates and recall basic definitions and properties of spherical harmonics. In Section \ref{section:asymptotics of the lapse}, we study the asymptotic behavior for the solutions of zero scalar curvature equation based on the work of Bartnik \cite{B} in dimension $n=3$ and Shi-Tam \cite{ST} in dimensions $n \ge 4$. In Section \ref{section:rotational symmetry at the order}, we show that the static potential $V$ has similar asymptotic behavior as $u$ and use it to show that $u$ and $V$ are rotationally symmetric to the order $r^{1-n}$. In Section \ref{section:proof of main theorem}, we prove full rotational symmetry. This leads to the proof of Theorems \ref{main} and \ref{equality case}.

We prove the applications of the main rigidity theorem in Section \ref{section:applications}. All of these rely on an additional Minkowski inequality for the level sets of the potential function originally derived in \cite{HW}. Since this inequality comes as a consequence of \eqref{Minkowski_static}, we extend the level-set Minkowski inequality to the larger class of static manifolds. In Theorem \ref{general_rig}, we combine these inequalities to prove \eqref{photon_surface_bounds} for any $(M^{n},g,V)$ with CMC, equipotential, and totally umbilical boundaries. Theorems \ref{photon_surface} and \ref{static_extension} then follow as corollaries. For the black hole uniqueness theorem, we apply the level set Minkowski inequality to the near-horizon level sets to obtain the upper bound in \eqref{bh_bounds}.

We include the proof of the upgraded version of the static Minkowski inequality in the appendix. A proper weak solution to IMCF always exists on an $(M^{n},g)$ with connected, smooth boundary $\partial M= \Sigma$ and the asymptotics \eqref{af}-- this was shown in Section 3 of \cite{HI}-- and jumps in the weak flow are constructed at discrete times for multiple boundary components in low dimensions. The monotonicity argument for the static Minkowski functional $Q(t)$ under proper weak IMCF from \cite{Wei} and \cite{Mc} is in fact local-- in particular, it does not depend on ADM mass-- so it directly translates to weakly\footnote{Here, ``weakly asymptotically flat" means $g-\delta = o_1(1)$ and $ V = 1 + o_1(1)$ instead of $o_2(1)$.} asymptotically flat $(M^{n},g,V)$. Finally, only the blow-down lemma (Lemma 7.1 in \cite{HI}) which assumes the asymptotics \eqref{af} is required to evaluate the limit of $Q(t)$. 

\subsection*{Acknowledgements}
Ye-Kai Wang acknowledges support from Wei-Lun Ko and NSTC grant 112-2115-M-A49-009-MY2. Brian Harvie thanks Columbia University and the NCTS Mathematics Division for continued support.


\section{Regularity of inverse mean curvature flow in asymptotically flat manifolds}
The proof of Theorem \ref{Mccormick} is based on inverse mean curvature flow (IMCF). Given a Riemannian manifold $(M^{n},g)$ and a smooth, closed manifold $\Sigma^{n-1}$, a family of embeddings $F: \Sigma^{n-1} \times [0,T) \rightarrow (M^{n},g)$ solve IMCF if
\begin{equation} \label{IMCF_flow}
    \frac{\partial}{\partial t} F (p,t) = \frac{1}{H} \nu(p,t), \hspace{1cm} (p,t) \in \Sigma^{n-1} \times [0,T),
\end{equation}
where $\nu$ and $H> 0$ are the outer unit normal and mean curvature of the surface $\Sigma_{t}=F_{t}(\Sigma) \subset (M^{n},g)$. For most geometric applications, it is more useful to work with weak solutions of \eqref{IMCF_flow}. A (proper) weak IMCF on a Riemannian manifold $(M^{n},g)$ corresponds to a function $u \in C^{0,1}_{\text{loc}} (M^{n})$, with the weak flow surfaces given by the level sets $\Sigma_{t}=\partial \{ u < t\}$ of $u$. We will postpone a more thorough discussion of weak IMCF until the appendix. However, we require an additional asymptotic property of weak IMCF in asymptotically flat manifolds to understand rigidity in \eqref{Minkowski_static}, which we prove in this section.

\begin{theorem}[Regularity of Weak IMCF in Asymptotically Flat Manifolds] \label{imcf_smoothing}
Let $(M^{n},g)$ be a Riemannian manifold which is asymptotically flat; more precisely, there exists a function $\epsilon(r)$ with $\lim_{r\rw\infty} \epsilon(r) = 0$ such that
\begin{align}
|g_{ij} - \delta_{ij}| + r|\pl_k g_{ij}| + r^2 |\pl_k\pl_l g_{ij}| \le \epsilon(r). 
\end{align}
Then for any proper weak solution $u \in C^{0,1}_{\text{loc}} (M^{n})$ of inverse mean curvature flow on $(M^{n},g)$, there exists a time $T_0 < +\infty$ such that $\Sigma_{t}= \partial \{ u < t \}$ is smooth, i.e. solves classical IMCF \eqref{IMCF_flow}, for $t \in (T_{0},\infty)$. Moreover, the mean curvature for $\Sigma_t, t > T_0$ satisfies the lower bound
\begin{align}
H \ge c \min ( 1, (t - T_0)^{\frac{1}{2}} ) e^{-\frac{t}{n-1}}
\end{align}
where the constant $c$ depends only on $n,\epsilon(r),$ and $T_0$.
\end{theorem}
\begin{remark}
This result was not neccessary in the prequel paper \cite{HW} since the additional boundary conditions in that paper are sufficient to ensure that the classical IMCF of $\Sigma_{0}$ does not develop finite-time singularities. However, Theorem \ref{imcf_smoothing} is needed for a complete rigidity statement, since in certain situations, e.g. a horizon boundary, it is not obvious a priori that the weak solution solves \eqref{IMCF_flow} at any time.
\end{remark}

Theorem \ref{imcf_smoothing} is independent of the static equations and may therefore be more broadly useful in mathematical relativity. In fact, according to the top of page 435 in \cite{HI2}, Huisken and Ilmanen at one point intended to prove this theorem in asymptotically flat manifolds and determine the center of mass of the resulting foliation by IMCF as an application. The key to the proof is a lower bound of mean curvature for smooth solutions of \eqref{IMCF_flow} with far-outlying, nearly-spherical initial data in an asymptotically flat end.
 
Let $\{ \Sigma_t \}_{0 \le t < \infty}$ be a weak solution of IMCF with smooth $\Sigma_0$ in an asymptotically flat manifold. Our starting point is the Blowdown Lemma \cite[Blowdown Lemma 7.1, Equation (7.5)]{HI} that says there exists a time $T$ such that for $t \ge T$, $\Sigma_t $ lies in the exterior region and is $C^1$-close to the coordinate sphere $\{ r = \mbox{const.}$. More precisely, the following $C^0$ and $C^1$ bounds hold for $\Sigma_t, t \ge T_0$:
\begin{align}\label{C0C1bound}
r_0 e^{\frac{t}{n-1}} \le r \le (r_0 + r_1(t)) e^{\frac{t}{n-1}}, \quad \langle \pl_r ,\nu \rangle > \eta(t)
\end{align}
where $\lim_{t \rw \infty} r_1(t) \rw 0$ and $\lim_{t\rw\infty} \eta(t) = 1$.

Following \cite{HI2}, we will first prove an apriori estimate on the lower bound of the mean curvature for smooth flow and then pass it to the weak flow by approximation.

\begin{theorem}\label{main_imcf}
Suppose $\{ \Sigma_t\}_{t \ge 0}$ is a smooth solution of IMCF lying in the asymptotic end and satisfying \eqref{C0C1bound}. Then there exists a time $T_1$ that depends only on $n,\epsilon(r)$ such that 
\[ H \ge c \min \{ 1, (t - T_1)^{\frac{1}{2}} \} e^{-\frac{t}{n-1}} \] 
for $t > T_1$ where the constant $c$ depends only on $n,\epsilon(r),$ and $T_1$. 
\end{theorem}

\begin{lem}
Let $l > 0$. For $f = r^{-l}$ and $X = r \frac{\pl}{\pl r}$, we have \begin{align}\label{D2f}
D_iD_j f -(\Delta_g f) g_{ij}  = l \lt[ (l+2) r^{-l-4} x^ix^j + (n-3-l) r^{-l-2}\delta_{ij} \rt] + o(r^{-l-2})
\end{align}
and the covariant derivative
\begin{align}\label{DX}
D X = g + \varepsilon
\end{align}
with 2-tensor $\varepsilon = o_1(1)$. 
\end{lem}
\begin{proof} Note that the Christoffel symbols of $g$ satisfy $\bar\Gamma_{ij}^k = o_1(r^{-1})$. We compute
\begin{align*}
\pl_i f = -l r^{-l-2} x^i = O(r^{-l-1})
\end{align*}
and
\begin{align*}
D_iD_j f &= \pl_i\pl_j f - \bar\Gamma_{ij}^k \pl_k f = l(l+2) r^{-l-4} x^ix^j - l r^{-l-2}\delta_{ij} + o(r^{-l-2})\\
\Delta_g f &= l(l+2-n) r^{-l-2} +o(r^{-l-2}). 
\end{align*}
to get \eqref{D2f}. Next, we compute with $X = x^i \frac{\pl}{\pl x^i}$
\begin{align*}
(DX)_{ij} = g_{jk} \lt( \pl_i x^k + \bar\Gamma_{il}^k x^l \rt) = g_{ij} + o_1(1).
\end{align*}
 \end{proof}

\begin{lem}\label{differential inequality}
Let $u = H^{-1} f^{-1} \langle X,\nu \rangle^{-2}$. Then there exists a large time $T_2$ and small $l$, depending only on $\epsilon(r)$ and $n$, such that 
\begin{align}
\pl_t u &< \frac{1}{H^2}\Delta u - \frac{2}{H^2} u^{-1} |\na u|^2 - \frac{2}{H^3} \na^a H \na_a u \notag\\
&= \na^a \lt( \frac{1}{H^2} \na_a u \rt) - \frac{2}{H^2} u^{-1}|\na u|^2
\end{align}
on $\Sigma_t, t \ge T_2$. 
\end{lem}
 
\begin{proof}
Note that the curvatures of $g$ satisfies $R_{ijkl} = o(r^{-2})$. We compute
\begin{align*}
\pl_t \langle X,\nu \rangle = \frac{1}{H} \lt( 1 + \varepsilon(\nu,\nu) \rt) - \langle X, \na (\frac{1}{H}) \rangle,
\end{align*}
\begin{align}\label{gradXnu}
\pl_b \langle X,\nu \rangle = \varepsilon(\pl_b, \nu) + h_b^c \langle X,\pl_c \rangle, 
\end{align}
and by Codazzi equation
\begin{align*}
\Delta \langle X,\nu \rangle &= \sigma^{ab} \lt[ \pl_a \lt( \varepsilon(\pl_b, \nu) + h_b^c \langle X,\pl_c \rangle \rt) - \Gamma_{ab}^c \lt( \varepsilon(\pl_c,\nu) + h_c^d \langle X, \pl_d \rangle \rt)\rt]\\
&= \sigma^{ab} D_{\pl_a} \varepsilon (\pl_b,\nu) - H \varepsilon(\nu,\nu) + h^{bc} \varepsilon(\pl_b,\pl_c) \\
&\quad + \langle X,\na H \rangle + R(X^T, \nu) + H + h^{ac} \varepsilon(\pl_a, \pl_c) - |A|^2 \langle X,\nu \rangle.
\end{align*}
Here $X^T$ denotes the tangential component of $X$ along $\Sigma_t$.  As a result, the evolution equation of $\langle X,\nu \rangle$ is 
\begin{align*}
\pl_t \langle X,\nu \rangle = \frac{1}{H^2} \lt( \Delta \langle X,\nu \rangle + |A|^2 \langle X,\nu \rangle - R(X^T,\nu) + A * \varepsilon + o(r^{-1}) \rt).
\end{align*}
Using the relation of Laplacians on a hypersurface, we get 
\begin{align*}
\pl_t f = \frac{2}{H} \frac{\pl f}{\pl\nu} + \frac{1}{H^2} \lt( \Delta f - \Delta_g f + D^2 f(\nu,\nu) \rt).
\end{align*}

We are ready to compute the evolution equation of $H f \langle X,\nu \rangle^2$.
\begin{align*}
&\pl_t \lt( H f \langle X,\nu \rangle^2 \rt) \\
&= \lt( \frac{1}{H^2}\Delta H - \frac{2}{H^3} |\na H|^2 - \frac{|A|^2}{H} - \frac{R(\nu,\nu)}{H} \rt) f \langle X,\nu\rangle^2 \\
&\quad + \frac{1}{H^2} \lt(\Delta f +  2H \frac{\pl f}{\pl\nu} + D^2 f(\nu,\nu) - \Delta_g f \rt) \\
&\quad + 2 H f \langle X,\nu \rangle \cdot \frac{1}{H^2} \lt( \Delta \langle X,\nu \rangle + |A|^2 \langle X,\nu \rangle - R(X^T,\nu) + A * \varepsilon + o(r^{-1}) \rt)\\
&= \frac{1}{H^2} \lt[ \Delta (H f \langle X,\nu \rangle^2) - 2 \na H \cdot \na f \langle X,\nu \rangle^2 - 2 H \na f \cdot \na \langle X,\nu \rangle^2 - 2 f \na H \cdot \na\langle X,\nu \rangle^2 \rt] \\
&\quad - \frac{2}{H^3} |\na H|^2 \cdot f \langle X,\nu \rangle^2 - \frac{f}{H} |\na \langle X,\nu \rangle|^2  \\
&\quad + \frac{\langle X,\nu \rangle}{H} \Big[  f |A|^2 \langle X,\nu \rangle - f R(\nu,\nu) \langle X,\nu \rangle + 2 H \frac{\pl f}{\pl\nu} \langle X,\nu \rangle \\
&\qquad\qquad\qquad\qquad + (D^2 f(\nu,\nu) - \Delta_g f)\langle X,\nu \rangle - 2f R(X^T,\nu) + 2f A * \varepsilon + f \cdot o(r^{-1}) \Big].
\end{align*}
Note that \eqref{C0C1bound} implies 
\begin{align*}
|X^T| = \langle X,\nu \rangle \cdot o(1)
\end{align*} Hence by \eqref{gradXnu} and Cauchy-Schwarz, we have \begin{align*}
|\na\langle X,\nu \rangle|^2  \le 2 |A^2| |X^T|^2 + o(1) = |A|^2 \langle X,\nu \rangle^2 \cdot o(1) + o(1).
\end{align*}
Moreover, note that $R(\nu,\nu) = o(r^{-2})$ and $R(X^T,\nu) = o(r^{-1})$. We thus obtain
\begin{align*}
\pl_t \lt( Hf\langle X,\nu \rangle^2 \rt) &= \frac{1}{H^2}\Delta (Hf\langle X,\nu \rangle^2) - \frac{2}{H^3} \na H \cdot \na (Hf\langle X,\nu \rangle^2) + \frac{\langle X,\nu\rangle}{H} \cdot I
\end{align*}
where
\begin{align*}
I &= -4 \na f \cdot \na\langle X,\nu \rangle + f |A|^2 \langle X,\nu \rangle (1 - o(1)) + 2H \frac{\pl f}{\pl\nu} \langle X,\nu \rangle \\
&\quad +(D^2 f(\nu,\nu) - \Delta_g f)\langle X,\nu \rangle + 2f A * \varepsilon + f \cdot o(r^{-1})
\end{align*}

We claim that $I > 0$ if $t$ is sufficiently large and $l$ is sufficiently small. First of all, we have 
\begin{align*}
|\na f| &= |Df| \cdot o(1) = f \cdot o(r^{-1}) \\
\frac{\pl f}{\pl\nu} &= - |Df| (1 + o(1)) = - lf r^{-1} (1 + o(1))
\end{align*}
as a consequence of \eqref{C0C1bound}.
By \eqref{gradXnu} and Cauchy-Schwarz, we have
\begin{align*}
- \na f \cdot \na \langle X,\nu \rangle &= - A(X^T,\na f) - \varepsilon (\na f,\nu) \ge  - |A|\langle X,\nu \rangle f \cdot o(r^{-1}) - f \cdot o(r^{-1}) \\
&\ge -\frac{1}{16} f |A|^2 \langle X,\nu \rangle  - f \cdot o(r^{-1}).
\end{align*}
By Cauchy-Schwarz, we have
\begin{align*}
2H \frac{\pl f}{\pl\nu} \langle X,\nu \rangle &\ge - \frac{H^2}{4(n-1)} f \langle X,\nu \rangle - 4(n-1) \lt( \frac{\pl f}{\pl\nu}\rt)^2 \frac{\langle X,\nu \rangle}{f} \\
&\ge - \frac{1}{4} f |A|^2 \langle X,\nu \rangle - 4(n-1)l^2 \frac{f}{r^2} \langle X,\nu \rangle (1 + o(1))
\end{align*}
and
\begin{align*}
2f A * \varepsilon \ge - \frac{1}{4}f |A|^2 \langle X,\nu \rangle + f \langle X,\nu \rangle^{-1} \cdot o(1)
\end{align*}
Putting these together, we get
\begin{align*}
I \ge f|A|^2 \langle X,\nu \rangle \lt( \frac{1}{4} - o(1) \rt) + \lt( D^2 f(\nu,\nu) - \Delta_g f - 4(n-1) \frac{l^2}{r^2} f - f \cdot o(r^{-2}) \rt)\langle X,\nu \rangle.
\end{align*}
By \eqref{D2f}, $D^2 f(\nu,\nu) - \Delta_g f = l(n-1) r^{-l-2} + o(r^{-l-2})$ and its follows that the second term on the right-hand side is positive if $l$ is small enough and $r$ is large enough (equivalently $t$ large enough). This completes the proof of the claim and we arrive at
\begin{align*}
\pl_t \lt( Hf\langle X,\nu \rangle^2 \rt) &> \frac{1}{H^2}\Delta (Hf\langle X,\nu \rangle^2) - \frac{2}{H^3} \na H \cdot \na (Hf\langle X,\nu \rangle^2).
\end{align*}

For $u = \lt( Hf\langle X,\nu \rangle^2\rt)^{-1}$, we have
\begin{align*}
\pl_t u &< - (Hf \langle X,\nu \rangle^2)^{-2} \lt( \frac{1}{H^2}\Delta (Hf\langle X,\nu \rangle^2) - \frac{2}{H^3} \na H \cdot \na (Hf\langle X,\nu \rangle^2) \rt) \\
&= \frac{1}{H^2}\Delta u - \frac{2}{H^2} u^{-1} |\na u|^2 - \frac{2}{H^3} \na H \cdot \na u
\end{align*}
\end{proof}

Lemma \ref{differential inequality} takes the same form of \cite[Lemma 1.4]{HI2} while $u = (H \langle X,\nu \rangle)^{-1}$ there. 

We need the Sobolev inequality on hypersurfaces.  It takes the same form on Euclidean spaces and on the exterior region of asymptotically flat manifolds under consideration. Note that \cite{HI2} works in $\R^{n+1}$ while we work in $n$-dimensional manifolds.

\begin{prop}\cite[Theorem 2.1]{HoSp}
There is a large constant $\mathcal{R}$ depending on $\epsilon(r)$ such that if
\begin{align*}
\inf_{N^{n-1}} r \ge \mathcal{R}, 
\end{align*}
then there is a constant $c(n)$ depending only on $n \ge 3$ such that
\begin{align*}
\lt( \int_{N^{n-1}} f^{\frac{n-1}{n-2}} \,d\mu\rt)^{\frac{n-2}{n-1}} \le c(n) \int_{N^{n-1}} |\na f| + |H||f| \,d\mu
\end{align*}
for any $f \in C^{0,1}_c(N^{n-1})$.
\end{prop}

We choose $l$ and a time $T_1$ such that the differential inequality \eqref{differential inequality} and the Sobolev inequality holds for $t \ge T_1$. The rest of our analysis follows closely the proof of Theorem 1.5 and Theorem 1.1 in \cite{HI2}; we thus omit arguments that are identical to those in \cite{HI2}. We make a translation in time $\Sigma_t' = \Sigma_{t + T_1}$ so that the comparison to \cite{HI2} is more transparent. It suffices to work under the crude bound
\begin{align}
R_1 e^{\frac{t}{n-1}} \le r, \langle X,\nu\rangle \le R_2 e^{\frac{t}{n-1}}
\end{align}
instead of \eqref{C0C1bound}.
 
The first step is an $L^p$-estimate.
\begin{lem}\label{Lp}
For all $p > 2$,
\begin{align*}
\| u\|_{L^p(\Sigma_t')} < c(n) (\frac{p}{n-1}+1 - \frac{pl}{n-1})^{\frac{1}{2}}  R_1^{-2+l} |\Sigma_0'|^{\frac{p+n-1}{p(n-1)}} e^{\frac{2t}{p}} \lt( e^{(\frac{2}{n-1}+\frac{2}{p} - \frac{2l}{n-1})t} - 1 \rt)^{-\frac{1}{2}}.
\end{align*}
\end{lem}
\begin{proof} In the proof we write $\mr = R_1 e^{\frac{t}{n-1}}$ and $\int$ for $\int_{\Sigma_t'}$. We compute for $p \ge 2$
\begin{align*}
\frac{d}{dt} \int u^p d\mu &< -p(p+1) \int \frac{1}{H^2} u^{p-2} |\na u|^2 d\mu + \int u^p d\mu \\
&= -p(p+1) \int f^2 \langle X,\nu \rangle^4	 u^{p} |\na u|^2 d\mu + \int u^p d\mu \\
&< - \mr^{4-2l} \int |\na g|^2 d\mu + \int u^p d\mu
\end{align*}
where $g = u^{\frac{p}{2}+1}$.

In applying the Sobolev inequality, we have to distinguish the cases $n=3$ and $n >3$. Note that $u^p = g^q$ with $q = \frac{2p}{p+2}$; $1< q < 2$ when $p > 2$. 

I. Case $n=3$. Setting $f = g^q$, the Sobolev inequality becomes 
\[ \lt( \int_{N} g^{2q} d\mu \rt)^{\frac{1}{q}} \le cq^2 |N|^{\frac{1}{q}} \int_{N} |\na g|^2 + H^2 g^2 d\mu \] and from H\"older inequality we derive
\begin{align*}
\frac{d}{dt} \int g^q d\mu < -c^{-1}q^{-2} |\Sigma_t'|^{-\frac{2}{q}} \mr^{4-2l} \lt( \int g^q \rt)^{\frac{2}{q}} + \mr^{4-2l} \int H^2 g^2 d\mu + \int u^p d\mu.
\end{align*}
In view of $\mr^{4-2l} \int H^2 g^2 d\mu \le \int g^q d\mu$ and $|\Sigma_t'| = |\Sigma_0'| e^t$, we get
\begin{align*}
\frac{d}{dt} \int g^q d\mu < -c^{-1} q^{-2} |\Sigma_0'|^{-\frac{2}{q}} R_1^{4-2l} e^{(1-\frac{2}{p} - l)t} \lt( \int g^q d\mu \rt)^{1+\frac{2}{p}} + 2 \int g^q d\mu.
\end{align*} 
Setting $\vphi = e^{-2t} \int g^q d\mu$ and recalling $1 < q < 2$, this is equivalent to
\begin{align*}
\frac{d}{dt} \vphi < -c^{-1} R_1^{4-2l} |\Sigma_0'|^{-\frac{p+2}{p}} e^{(1+\frac{2}{p} - l)t} \vphi^{\frac{p+2}{p}}.
\end{align*}

II. Case $n>3$. Using the Sobolev inequality in the form
\[ \lt( \int_{N} g^{\frac{2(n-1)}{n-3}} d\mu \rt)^{\frac{n-3}{n-1}} \le c(n) \int_{N} |\na g|^2 + H^2 g^2 d\mu,\] we obtain
\begin{align*}
\frac{d}{dt} \int g^q d\mu &< - c(n)^{-1} \mr^{4-2l} \lt(\int g^{\frac{2(n-1)}{n-3}} \rt)^{\frac{n-3}{n-1}} + \mr^{4-2l} \int H^2 g^2 d\mu + \int u^p d\mu \\
&\le - c(n)^{-1} \mr^{4-2l} |\Sigma_t'|^{-\frac{2(p+n-1)}{p(n-1)}} \lt( \int g^q d\mu \rt)^{\frac{p+2}{p}} +2 \int g^q d\mu \\
&= -c(n)^{-1} R_1^{4-2l} |\Sigma_0'|^{-\frac{2(p+n-1)}{p(n-1)}} e^{\lt( \frac{2}{n-1} - \frac{2}{p} - \frac{2}{n-1}l\rt)t } \lt( \int g^q d\mu \rt)^{\frac{p+2}{p}} + 2 \int g^q d\mu.
\end{align*}

Setting $\vphi = e^{-2t} \int g^q d\mu$, we see that for $n \ge 3$,
\begin{align*}
\frac{d}{dt}\vphi < -c(n)^{-1} R_1^{4-2l} |\Sigma_0'|^{-\frac{2(p+n-1)}{p(n-1)}} e^{\lt( \frac{2}{n-1} + \frac{2}{p} -\frac{2l}{n-1} \rt)t } \vphi^{\frac{p+2}{p}}.
\end{align*}
After simplifications
\begin{align*}
\frac{d}{dt}\vphi^{-\frac{2}{p}} &>  \frac{2}{p}c(n)^{-1} R_1^{4-2l} |\Sigma_0'|^{-\frac{2(p+n-1)}{p(n-1)}} e^{(\frac{2}{n-1}+\frac{2}{p} - \frac{2l}{n-1})t},\\
\vphi^{-\frac{2}{p}} &> \frac{2}{p}c(n)^{-1} R_1^{4-2l} |\Sigma_0'|^{-\frac{2(p+n-1)}{p(n-1)}} \lt( \frac{2}{n-1} + \frac{2}{p} - \frac{2l}{n-1} \rt)^{-1} \lt( e^{(\frac{2}{n-1}+\frac{2}{p} - \frac{2l}{n-1})t} - 1 \rt),
\end{align*}
we arrive at
\begin{align*}
\| u\|_{L^p} < c(n) (\frac{p}{n-1}+1 - \frac{pl}{n-1})^{\frac{1}{2}}  R_1^{-2+l} |\Sigma_0'|^{\frac{p+n-1}{p(n-1)}} e^{\frac{2t}{p}} \lt( e^{(\frac{2}{n-1}+\frac{2}{p} - \frac{2l}{n-1})t} - 1 \rt)^{-\frac{1}{2}}.
\end{align*}
\end{proof}

Let $t_0 > 0$ be arbitrary but fixed and set \begin{align*}
v = (t-t_0)^\beta u
\end{align*}
where $\beta = \frac{1}{4}$. Let $v_k = \max(v-k,0)$ for $k \ge 0$ and let $A(k) = \{ p \in \Sigma_t' | v(p,t) > k \}$.

We have
\begin{align*}
&\frac{d}{dt} \int_{\Sigma_t'} v_k^2 d\mu + 2k^2 (t-t_0)^{-2\beta} \mr^{4-2l} \int_{\Sigma_t'} |\na v_k|^2 + H^2 v_k^2 d\mu \\
&< 2\beta (t-t_0)^{-1} \int_{A(k)} v^2 d\mu + 3 \int_{\Sigma_t'} v_k^2 d\mu 
\end{align*} 
where $\mr = R_1 e^{\frac{t}{n-1}}$. Again we have to distinguish the cases $n=3$ and $n \ge 4$.

I. Case $n=3$. The Sobolev inequality for some fixed $1<q < \infty$ yields
\begin{align*}
&\frac{d}{dt} \int_{\Sigma_t'} v_k^2 d\mu + 2 k^2  c^{-1}q^{-2} (t-t_0)^{-2\beta} \mr^{4-2l} |\Sigma_t'|^{-\frac{1}{q}} \lt( \int_{\Sigma_t} v_k^{2q} d\mu \rt)^\frac{1}{q} \\
&< 2\beta (t-t_0)^{-1} \int_{A(k)} v^2 d\mu + 3 \int_{\Sigma_t'} v_k^2 d\mu. 
\end{align*}
Using H\"{o}lder inequality and $|\Sigma_t'| = e^t |\Sigma_0'|$, the term $3 \int v_k^2 d\mu$ can be absorbed to the left and we obtain
\begin{align*}
\frac{d}{dt} \int_{\Sigma_t'} v_k^2 d\mu + k^2 c^{-1} q^{-2}R_1^{4-2l} (t-t_0)^{-2\beta} e^{(2-l)t} |\Sigma_t'|^{-\frac{1}{q}} \lt( \int_{\Sigma_t'} v_k^{2q} d\mu \rt)^\frac{1}{q} < 2\beta (t-t_0)^{-1} \int_{A(k)} v^2 d\mu
\end{align*}
in the interval $[t_0,t_1]$ provided $k \ge  k_0 >0$ and 
\begin{align}\label{k0_1}
k_0^2 \ge 3c q^2 (t_1 - t_0)^{2\beta} R_1^{-4+2l} |\Sigma_0'| e^{(-1+l)t_0}.
\end{align}

II. Case $n \ge 4$. Set $q = \frac{n-1}{n-3}$ and we get the estimate
\begin{align*}
\frac{d}{dt} \int_{\Sigma_t'} v_k^2 d\mu + k^2 c(n)^{-1}  (t-t_0)^{-2\beta} R_1^{4-2l} e^{\frac{4-2l}{n-1}t } \lt( \int_{\Sigma_t} v_k^{2q} d\mu \rt)^\frac{1}{q} < 2\beta (t-t_0)^{-1} \int_{A(k)} v^2 d\mu
\end{align*}
in the interval $[t_0,t_1]$ provided $k \ge k_0 >0$ and 
\begin{align}\label{k0_2}
k_0^2 \ge 3 c(n) (t_1-t_0)^{2\beta} R_1^{-4+2l} |\Sigma_0'|^{\frac{2}{n-1}} e^{(-1+l)\frac{2}{n-1} t_0}.
\end{align}

In summary, setting 
\begin{align*}
B(k) := \begin{cases}
c^{-1} q^{-2} k^2 R_1^{4-2l} |\Sigma_{t_1}'|^{-\frac{1}{q}} e^{(2-l)t_0}, & n=3,\\
c(n)^{-1} k^2 R_1^{4-2l} e^{\frac{4-2l}{n-1} t_0}, &n\ge 4,
\end{cases}
\end{align*} we have the estimate
\begin{align*}
\frac{d}{dt} \int_{\Sigma_t} v_k^2 d\mu + B(k) (t - t_0)^{-2\beta}  \lt( \int_{\Sigma_t} v_k^{2q} d\mu \rt)^\frac{1}{q} < 2\beta (t-t_0)^{-1} \int_{A(k)} v^2 d\mu.
\end{align*}
in the interval $[t_0,t_1]$ for all $n \ge 3$ provided $k \ge k_0$. Fast forwarding, we arrive at the analogue of \cite[(1.14)]{HI2}:
\begin{align*}
&\int_{t_0}^{t_1} \int_{A(k)} v_k^2 d\sigma\\
&< c(n) B(k)^{-\frac{1}{q_0}} \| A(k)\|^{1 - \frac{1}{q_0}} \int_{t_0}^{t_1} (t-t_0)^{-1} |A(k)|^{1-\frac{1}{r}} \lt( \int_{\Sigma_t'} v^{2r} d\mu \rt)^{\frac{1}{r}} dt
\end{align*}
for some $r>1$ to be chosen. Here $d\sigma := (t-t_0)^{-2\beta} d\mu dt$, $\| A(k) \| := \int_{t_0}^{t_1} \int_{A(k)} d\sigma$, and 
\begin{align*}
1 < q < \infty, \quad q_0 = 2 - \frac{1}{q}&, \quad \mbox{if } n=3,\\
q = \frac{n-1}{n-3}, \quad q_0 = 2 - \frac{1}{q} = \frac{n+1}{n-1}&, \quad \mbox{if } n \ge 4.
\end{align*}

Using the $L^p$-estimate in Lemma \ref{Lp} for $u = v (t-t_0)^{-\frac{1}{4}}$, we get
\begin{align*}
\lt( \int_{\Sigma_t'} v^{2r}\rt)^{\frac{1}{r}} &\le (t-t_0)^{\frac{1}{2}} \cdot c(n) \lt( \frac{2r}{n-1} + 1 - \frac{2rl}{n-1}\rt) R_1^{-4+2l} |\Sigma_0'|^{\frac{2r+n-1}{r(n-1)}} e^{\frac{2t}{r}} \lt( e^{(\frac{2}{n-1} + \frac{1}{r} - \frac{2l}{n-1})t} - 1 \rt)^{-1} \\
&\le (t-t_0)^{\frac{1}{2}} \cdot c(n) \lt( \frac{2r}{n-1} + 1 - \frac{2rl}{n-1}\rt) R_1^{-4+2l} |\Sigma_0'|^{\frac{2r+n-1}{r(n-1)}}  \cdot \\
& \hspace{5cm} e^{(\frac{1}{r} - \frac{2}{n-1} + \frac{2l}{n-1})t} \max \lt( 1, t^{-1} \lt( \frac{2}{n-1} + \frac{1}{r} - \frac{2l}{n-1} \rt)^{-1} \rt).
\end{align*}
Moreover, we have $\int_{t_0}^{t_1} (t-t_0)^{-\frac{1}{2}} |A(k)|^{1-\frac{1}{r}} dt \le 2^{\frac{1}{r}} \| A(k) \|^{1 - \frac{1}{r}} (t_1-t_0)^{\frac{1}{2r}}$.

Now choosing $r$ depending only on $q_0 = q_0(n)$ large enough so that $\gamma = 2 - \frac{1}{q_0} - \frac{1}{r} > 1$, we get the iteration inequality
\begin{align*}
|h-k|^2 \| A(h)\| &< \int_{t_0}^{t_1} \int_{A(k)} v_k^2 d\sigma\\
&< c(n,\epsilon) R_1^{-4+2l} B(k)^{-\frac{1}{q_0}} |\Sigma_0'|^{\frac{2r+n-1}{(n-1)r}} e^{(\frac{1}{r} - \frac{2}{n-1} + \frac{2l}{n-1}) t_1}\cdot \max( 1, t_0^{-1} ) (t_1 - t_0)^{\frac{1}{2r}} \| A(k)\|^{\gamma}  .  
\end{align*}
for $h>k \ge k_0 > 0$ where the dependence of $c$ on $\epsilon(r)$ comes from $l$. Stampacchia lemma (see \cite{Fre} for example) then yields
\begin{align*}
\|  A(k_0 + d) \| =0
\end{align*} with
\begin{align*}
d^2 &= c(n,\epsilon) R_1^{-4+2l}  B(k_0)^{-\frac{1}{q_0}}  |\Sigma_0'|^{\frac{2r+n-1}{(n-1)r}} e^{(\frac{1}{r} - \frac{2}{n-1} + \frac{2l}{n-1}) t_1} \cdot \max( 1, t_0^{-1})  (t_1 - t_0)^{\frac{1}{2r}} \| A(k_0)\|^{\gamma-1}\\
&\le c(n,\epsilon) R_1^{-4+2l} B(k_0)^{-\frac{1}{q_0}} |\Sigma_0'|^{1 - \frac{1}{q_0} + \frac{2}{n-1}}\cdot \max( 1, t_0^{-1})(t_1 - t_0)^{\frac{1}{2}(1 - \frac{1}{q_0})} e^{\lt( 1 - \frac{1}{q_0} - \frac{2}{n-1}+ \frac{2l}{n-1} \rt)t_1}
\end{align*}
where in the last inequality we used $\| A(k_0)\| \le 2 |\Sigma_0'| e^{t_1} (t_1 - t_0)^{\frac{1}{2}}$. 

The optimal bound is achieved when $k_0$ is chosen so that $d$ and $k_0$ are comparable. By definition, 
\begin{align*}
B(k_0) = c^{-1} k_0^2 R_1^{4-2l} e^{\frac{4-2l}{n-1} t_0} \lt( |\Sigma_0'| e^{t_1} \rt)^{-\frac{s}{q}}
\end{align*}
with the exponent $s=1$ for $n=3$ and $s=0$ for $n \ge 4$. Note that
\begin{align*}
1 - \frac{1}{q_0} - \frac{2}{(n-1)q_0} + \frac{s}{qq_0} =0.
\end{align*}
We distinguish small and large times:

I. Case $t_0 \le 1$. We choose $t_1 = 2 t_0$ and $k_0 = c(n) t_0^{-\frac{1}{4}} R_1^{-2+l} |\Sigma_0'|^{\frac{1}{n-1}} e^{-\frac{1}{n-1}t_0}$ where $c(n)$ is chosen such that inequalities \eqref{k0_1} and \eqref{k0_2} are satisfied. Since $t_0 \le 1$, the exponential function is bounded and we get
\begin{align*}
\frac{d^2}{k_0^2} &\le c(n,\epsilon) ( k_0^2 R_1^{4-2l} )^{-1 - \frac{1}{q_0}} |\Sigma_0'|^{\frac{s}{qq_0} + 1 - \frac{1}{q_0} + \frac{2}{n-1}} t_0^{-\frac{1}{2} (1 + \frac{1}{q_0})} \le c(n,\epsilon).    
\end{align*}

II. Case $t_0 \ge 1$. We choose $t_1 = t_0 + 1$ and $k_0 = c(n) R_1^{-2 + l} |\Sigma_0'|^{\frac{1}{n-1}} e^{\frac{1}{n-1}(-1+l) t_0}$. 
we have
\begin{align*}
\frac{d^2}{k_0^2} &\le c(n,\epsilon) (k_0^2 R_1^{4-2l})^{-1-\frac{1}{q_0}} |\Sigma_0'|^{\frac{s}{qq_0} + 1 - \frac{1}{q_0} + \frac{2}{n-1}} \\
&\hspace{5cm} \cdot \exp \lt( \lt( 1 - \frac{1}{q_0} - \frac{2}{n-1} + \frac{2l}{n-1} + \frac{-4+2l}{(n-1)q_0} + \frac{s}{qq_0} \rt) t_0 \rt)\\
&\le c(n,\epsilon).
\end{align*}

In view of the definition $v = (t-t_0)^{\frac{1}{4}} H^{-1} f^{-1} \langle X,\nu \rangle^{-2}$ and the choice $t_1 = \min (2t_0,t_0+1)$, this shows in both cases
\begin{align*}
\sup_{\Sigma_{t_1}'} H^{-1} f^{-2} \langle X,\nu \rangle^{-2} \le c(n,\epsilon) \max(1, t_1^{-\frac{1}{2}}) R_1^{-2 + l} |\Sigma_0'|^{\frac{1}{n-1}} e^{\frac{1}{n-1} (-1 + l) t_1}.
\end{align*} 
and \begin{align*}
\min_{\Sigma_t'} H \ge c(n,\epsilon) \min(1,t^{\frac{1}{2}}) e^{-\frac{t}{n-1}} \lt( \frac{R_1}{R_2}\rt)^{2-l} |\Sigma_0'|^{-\frac{1}{n-1}}.
\end{align*}
Recalling $\Sigma_t' = \Sigma_{T_1 + t}$ and using the bound \eqref{C0C1bound}, the proof of Theorem \ref{main_imcf} is completed.

\begin{proof}[Proof of Theorem \ref{imcf_smoothing}]
Given the apriori estimate for the mean curvature, the eventual smoothness of IMCF is proved in Theorem 2.7 of \cite{HI2} for Euclidean space, see also \cite[Theorem 1.2]{LW}. Their proof can be easily adapted to asymptotically flat manifolds. We outline it for the convenience of the readers. 

By the Weak Existence Theorem 3.1 of \cite{HI}, all level sets of a weak solution of IMCF have non-negative, bounded, measurable mean curvature. Let $T$ be given by the Blowdown Lemma. Using the theory of mean curvature flow, there exists a smooth solution of mean curvature flow $\hat F: \Sigma \times (0,\epsilon_0 )$ such that $\hat \Sigma_\epsilon = \hat F_\epsilon (\Sigma)$ converges to $\Sigma_T$ uniformly in $C^{1,\beta} \cap W^{2,p}$ for all $0 < \beta < 1, 1 \le p < \infty$ as $\epsilon \rw 0$. By strong maximum principle, $\hat\Sigma_\epsilon$ either has positive mean curvature everywhere or is a closed minimal hypersurface. Since we may assume all coordinate spheres in the exterior region have positive mean curvature, the latter cannot happen. Increase $T$ if necessary, Theorem \ref{main_imcf} implies that the solution $\hat\Sigma_{\epsilon,s}$ of IMCF starting from $\hat\Sigma_\epsilon$ have a uniform positive lower bound of mean curvature independent of $\epsilon$. By Corollary 2.3 and Remark 2.8(i) of \cite{HI2}, the positive lower and upper bounds of $H$ (inherited from the Weak Existence Theorem) guarantee that the solution exists and is smooth for all time $0 \le s< \infty$. Exactly as in the proof of Theorem 2.5 of \cite{HI2} on page 448-449, $\hat\Sigma_{\epsilon,s}$ converges as $\epsilon \rw 0$ to a smooth solution of IMCF starting from $\Sigma_T$. The uniqueness of IMCF then implies $\Sigma_t$ is smooth for $t> T$.
\end{proof}

\section{Quasi-spherical metrics and umbilical IMCF}
Using Theorem \ref{imcf_smoothing}, we show in this section that an asymptotically flat static manifold $(M^{n},g)$ is isometric to the quasi-spherical metric \eqref{quasi-spherical} outside of a compact set if $\Sigma$ achieves equality in inequality \eqref{Minkowski_static}. \eqref{Minkowski_static} arises from considering the monotone quantity

\begin{equation} \label{Q}
    Q(t) = |\Sigma_{t}|^{-\frac{n-2}{n-1}} \left( \int_{\Sigma_{t}} VH d\sigma + 2(n-1) w_{n-1} m \right)
\end{equation}
under the weak IMCF $\Sigma_{t}$ of $\Sigma$, where $m$ is defined in \eqref{mass}-- we include the proof of monotonicity before the smoothing time $T_{0}$ in the appendix. Here, we calculate the variation of \eqref{Q} for $t > T_{0}$. Given

\begin{eqnarray}
    \frac{\partial}{\partial t} H &=& -\Delta_{\Sigma_{t}} \frac{1}{H} - \left(|h|^{2} + \text{Ric}(\nu,\nu) \right) \frac{1}{H} \nonumber \\
                                  &=& \frac{1}{H^{2}} \Delta_{\Sigma_{t}} H - 2 \frac{|\nabla H|^{2}}{H^{3}} - \frac{1}{n-1} H - \left({|\mathring{h}|^{2}} + \text{Ric}(\nu,\nu) \right) \frac{1}{H} \label{H_evolution} \\
    \frac{\partial}{\partial t} V &=& \frac{2}{H} \frac{\partial V}{\partial \nu} - \frac{1}{H} \frac{\partial V}{\partial \nu} \nonumber \\
                                  &=& \frac{1}{H^{2}} \Delta_{\Sigma_{t}} V + \frac{V}{H^{2}} \text{Ric}(\nu,\nu) + \frac{2}{H} \frac{\partial V}{\partial \nu}. \nonumber
\end{eqnarray}
and the first variation of area, the evolution of $Q(t)$ under \eqref{IMCF_flow} is
\begin{eqnarray*}
   \frac{\partial}{\partial t} Q(t) &=&-\frac{n-2}{n-1} |\Sigma_{t}|^{-\frac{n-2}{n-1}} \left(\int_{\Sigma_{t}} VH d\sigma + 2(n-1) w_{n-1} m \right) \\
   & & + |\Sigma_{t}|^{-\frac{n-2}{n-1}} \left(\int_{\Sigma_{t}} \frac{\partial}{\partial t} (VH) + VH d\sigma\right ) \\
    &=&  -\frac{n-2}{n-1} |\Sigma_{t}|^{-\frac{n-2}{n-1}} \left(\int_{\Sigma_{t}} (VH) d\sigma + 2(n-1) w_{n-1} m \right) \\
    & & + |\Sigma_{t}|^{-\frac{n-2}{n-1}} \left(\int_{\Sigma_{t}} 2 \frac{\partial V}{\partial \nu} + \frac{n-2}{n-1} VH - \frac{V}{H} |\mathring{h}|^{2} d\sigma \right) \\
    &=& - |\Sigma_{t}|^{-\frac{n-2}{n-1}} \int_{\Sigma_{t}} \frac{V}{H} |\mathring{h}|^{2} d\sigma.
\end{eqnarray*}
Therefore, $Q(t) = Q(0)$ for $t \in [T_{0},\infty)$ implies that $\Sigma_{t} \subset (M^{n},g)$ is totally umbilical. Combining this with the blow-down lemma for IMCF from \cite{HI}, we may completely characterize the intrinsic geometry of $\Sigma_{t}$.


\begin{theorem}
Let $(M^{n},g)$ be a Riemannian manifold. Suppose there exists a diffeomorphism $F: \mathbb{R}^{n} \setminus B_{1}(0) \rightarrow M^{n} \setminus K$ such that

\begin{equation} \label{weakly_af}
g_{ij}=\delta_{ij} + o_{1}(1).
\end{equation}
Now suppose that $(M^{n},g)$ is foliated by a solution $F: \mathbb{S}^{n-1} \times (0,\infty) \rightarrow (M^{n},g)$ to IMCF \eqref{IMCF_flow} where each $\Sigma_{t}= F_{t}(\mathbb{S}^{n-1}) \subset (M^{n},g)$ is totally umbilical. Then the induced metric $\sigma(p,t)$ on the slice $\Sigma_{t}$ is

\begin{equation}
    \sigma(p,t)= r_{0}^{2} e^{\frac{2t}{n-1}} g_{S^{n-1}},
\end{equation}
where $g_{S^{n-1}}$ is the standard metric on the sphere and $r_{0} = \left( \frac{|\Sigma|}{w_{n-1}} \right)^{\frac{1}{n-1}}$.
\end{theorem}
\begin{remark}
We also showed this in the prequel, c.f. \cite{HW} Theorem 3.2. In that paper we assumed the standard notion of asymptotic flatness. The same proof holds under the decay condition \eqref{weakly_af}, but we include it here for the convenience of the reader.
\end{remark}

\begin{proof}
 Given $\Sigma_{t}$ umbilical, the first variation formula for the induced metric $\sigma(p,t)$ of $\Sigma_{t}$ becomes

\begin{eqnarray*}
    \frac{\partial}{\partial t} \sigma_{ab} &=& \frac{2}{H} h_{ab} = \frac{2}{H} \frac{H}{n-1} \sigma_{ab} = \frac{2}{n-1} \sigma_{ab}.
\end{eqnarray*}
Therefore, the induced metric evolves by scaling under \eqref{IMCF_flow} with

\begin{equation} \label{self_similar}
\sigma_{ab}(p,t) = e^{\frac{2t}{n-1}} \sigma_{ab} (p,0), \hspace{1cm} (p,0) \in \mathbb{S}^{n-1} \times (0,\infty).    
\end{equation}

Let $U= \mathbb{R}^{n} \setminus B_2(0)$ be the end of $(M^{n},g)$ with coordinates $(x^{1},\dots,x^{n})$ and the asymptotics \eqref{weakly_af}, and consider the blow-down embedding $\widetilde{F}_{t}$ and corresponding blow-down metric $g_{t}$

\begin{eqnarray}
   \widetilde{\mathbf{F}}_{t} &=& r(t)^{-1} \mathbf{F}_{t}: \mathbb{S}^{n-1} \rightarrow U \\
   g_{t}(\mathbf{x})&=& r(t)^{-2}g(r(t)\mathbf{x}), \hspace{2cm} x \in U. \label{blow_down2}
\end{eqnarray}
for embeddings $\mathbf{F}_{t}(p)= \left(F_{1}(p,t),F_{2}(p,t),\dots,F_{n}(p,t) \right) \in U$ satisfying \eqref{IMCF_flow} and for
\begin{equation} \label{r}
    r(t)= \left( \frac{|\Sigma_{0}|}{w_{n-1}} \right)^{\frac{1}{n-1}} e^{\frac{t}{n-1}}.
\end{equation}


By \eqref{self_similar}, $\widetilde{\mathbf{F}}_{t}: \left( \mathbb{S}^{n-1}, r_{0}^{-2} \sigma_{0} \right) \rightarrow (U,g_{t}(x))$ is an isometric embedding, where $\sigma_{0}$ abstractly denotes the induced metric of $\Sigma_{0}$. The Hessian of each component with respect to $r_{0}^{-2} \sigma_{0}$ may be decomposed as

\begin{equation} \label{hessian}
    \nabla^{2}_{r_{0}^{-2}\sigma_{0}} \widetilde{F}_{i} (p,t) = \nabla^{2}_{g_{t}} x_{i} - \nu_{t}(x_{i}) \widetilde{h}(p,t) = \Gamma^{i}_{jk}(t) - \frac{\nu_{t}(x_{i})}{n-1} \widetilde{H} r_{0}^{-2} \sigma(p,0). 
\end{equation}
Here, $\Gamma^{i}_{jk}(t)$ are the Christoffels symbol of the metric $g_{t}$ in coordinates $(x^{1},\dots,x^{n})$. Given that $g_{ij}(x) = \delta_{ij} + o_{1}(1)$, we have by scaling invariance of the Levi-Civita connection that

\begin{equation*}
    \Gamma^{i}_{jk}(t) = o(1), \hspace{1cm} i,j,k \in \{ 1, \dots, n \}. 
\end{equation*}
Meanwhile, an upper bound on the second term in \eqref{hessian} follows from $C^{1}$ convergence of $\widetilde{\Sigma}_{t}$ and the $L^{\infty}$ bound on $\widetilde{H}$, see Section 7 in \cite{HI}. Applying these to \eqref{hessian}, we find

\begin{equation*}
    |\nabla^{2}_{r_{0}^{-2} \sigma_{0}} \widetilde{F}_{i}(p,t) | \leq C, \hspace{1cm} i \in \{ 1, \dots, n \},
\end{equation*}
for a uniform constant $C$. This gives in local coordinates $(\theta_{1},\dots,\theta_{n-1}) \in \mathbb{S}^{n-1}$ that
\begin{equation*}
    |\frac{\partial}{\partial \theta_{i}} \widetilde{F}_{j}(p,t) - \frac{\partial}{\partial \theta_{i}} \widetilde{F}_{j}(q,t)| \leq C d_{r_{0}^{-2} \sigma_{0}} (p, q), \hspace{1cm} p,q \in \mathbb{S}^{n-1}. 
\end{equation*}
By the standard compact containment $C^{1,1}(\mathbb{S}^{n-1},r_{0}^{-2} \sigma_{0}) \subset C^{1,\alpha}(\mathbb{S}^{n-1}, r_{0}^{-2} \sigma_{0})$, $0 <\alpha <1$, c.f. \cite{Aub}, we may pass to a subsequence $t_{k} \rightarrow \infty$ so that for each $j \in \{ 1, \dots, n \}$ we have

\begin{equation} \label{limit_map}
    \widetilde{F}_{j} ( ;, t_{k}) \rightarrow F_{j} (;,\infty) \hspace{1cm} \text{ in    } C^{1,\alpha}(\mathbb{S}^{n-1}, r_{0}^{-2} \sigma_{0}).
\end{equation}
We consider the $C^{1}$ map $\mathbf{F}_{\infty}$ into $U$ defined by $\mathbf{F}_{\infty}(p)= (F_{1}(p,\infty), \dots, F_{n}(p,\infty))$. $\widetilde{\Sigma}_{t} \rightarrow \partial B_{1}(0)$ in $C^{1}$ as $t \rightarrow \infty$ by the blow-down lemma for IMCF (Lemma 7.1 in \cite{HI}), and so we have that $\lim_{k} |\widetilde{F}_{k}| = 1$ and $\mathbf{F}_{\infty}$ embeds $\mathbb{S}^{n-1}$ into $\partial B_{1}(0)$. Furthermore, since as $t \rightarrow \infty$ we have

\begin{eqnarray*}
    \frac{\partial}{\partial \theta_{j}} \widetilde{F}_{k}(p,t) &=& \frac{\partial}{\partial \theta_{j}} F_{k}(p,\infty) + o(1),  \\
    (g_{t})_{ij}(x) &=& \delta_{ij} + o_{1}(1),
\end{eqnarray*}
it follows that
\begin{equation*}
    \mathbf{F}_{\infty}: (\mathbb{S}^{n-1}, r_{0}^{-2} \sigma_{0}) \rightarrow \partial B_{1} (0) \subset (U,\delta)
\end{equation*}
is a $C^{1}$ isometric embedding. To complete the proof, we argue that $\sigma_{0}=r_{0}^{2} g_{S^{n-1}}$ (this is immediate if $\mathbf{X}_{\infty}$ is $C^{\infty}$). We have that the isometry group $\text{Isom}(\partial B_{1}(0), g_{S^{n-1}})$ is isomorphic to $\text{Isom}(\mathbb{S}^{n-1}, r_{0}^{-2} \sigma_{0})$ via conjugation with $F_{\infty}$, which implies that $\sigma_{0}$ is of constant curvature. The conclusion follows.

\end{proof}

\begin{cor} \label{cor}
Suppose that $\Sigma$ achieves equality in \eqref{Minkowski_static}. Let $u: (M^{n},g) \rightarrow \mathbb{R}$ be the proper weak IMCF of $\Sigma$, and let $K= \{ u \leq T_{0} \}$ for $T_{0}$ from Theorem \ref{imcf_smoothing}. Then there exists a diffeomorphism $F: \mathbb{S}^{n-1} \times (r_{0},\infty) \rightarrow (M^{n} \setminus K,g)$ such that the pull-back of $g$ equals \eqref{quasi-spherical}.
\end{cor}
\begin{proof}
Since $Q(t)=Q(0)=(n-1)w_{n-1}^{\frac{1}{n-1}}$ for all $t$, $(M^{n} \setminus K, g)$ is foliated by an umbilical solution $\Sigma_{t}$ of \eqref{IMCF}, we may express the metric as

\begin{equation*}
g= \frac{1}{H^{2}} dt^{2} + \sigma = \frac{1}{H^{2}} dt^{2} + r_{0}^{2}e^{\frac{2(t-T_{0})}{n-1}} g_{S^{n-1}}, \hspace{2cm} r_{0}= \left( \frac{|\Sigma_{T_{0}}|}{w_{n-1}} \right)^{\frac{1}{n-1}}.
\end{equation*}
We then make the transformation

\begin{eqnarray*}
r &=& r_{0} e^{\frac{t-t_{0}}{n-1}}, \\
u(r,\theta) &=& \frac{(n-1)r}{H}.
\end{eqnarray*}
\end{proof}

\section{Static equations in polar coordinates}\label{section:static equations in polar coordinates}
We now consider the static equations for a metric of the form \eqref{quasi-spherical}. For a coordinate system $\theta^a, a = 1, \cdots, n-1$ on $S^{n-1}$, we will refer to $(r,\theta^a)$ as polar coordinates. In our analysis of static equations for quasi-spherical metrics, Section \ref{section:static equations in polar coordinates} to Section \ref{section:proof of main theorem}, $\nabla$ and $\Delta$ would denote the Levi-Civita connection and Laplacian of the standard metric $\sigma = g_{S^{n-1}}$ on $S^{n-1}$.

Taking the trace of the first static equation in \eqref{static}, the second one is equivalent to 
\begin{align}
R(g) =0. \label{Rg=0}
\end{align}
We first derive the static equations in polar coordinates. 
\begin{prop} \label{derivation of static equations}
In polar coordinates, the zero scalar curvature equation $R(g)=0$ is given by 
\begin{align}\label{zero scalar curvature}
u^2 \Delta u - (n-1) r u_r+ \frac{(n-1)(n-2)}{2}(u-u^3) =0
\end{align}
and the Laplace equation $\Delta_g V =0$ is given by
\begin{align}
u^{-2} \lt( V_{rr} - u^{-1} u_r V_r - \frac{u}{r^2} \na^a u \na_a V \rt) + \frac{1}{r^2}\Delta V + \frac{n-1}{r} u^{-2} V_r =0.
\end{align} 
The $ra$-component and $ab$-component of the static equations are given by
\begin{align}\label{ra}
\pl_a V_r -u^{-1}u_a V_r - \frac{1}{r} \pl_a V &= \frac{n-2}{r} V u^{-1} \pl_a u
\end{align}
and
\begin{align}\label{ab}
\na_a \na_b V + r u^{-2}  V_r \sigma_{ab} &= V \lt[ -u^{-1} \na_a \na_b u + r u^{-3} u_r \sigma_{ab} + (n-2) (1 - u^{-2}) \sigma_{ab}\rt].
\end{align}
The combination of $rr$-component of the static equation and the zero scalar curvature equation gives
\begin{align}\label{rr}
V_{rr} - u^{-1} u_r V_r - \frac{u}{r^2} \na^a u \na_a V &= \frac{(n-1)(n-2)}{2r^2} (1-u^2) V.
\end{align}
\end{prop}
\begin{proof}
Let $\nu $ denote unit vector $u^{-1} \pl_r$ (the unit normal of the level sets of $r$). Setting $\beta=0$ in Bartnik \cite[(2.20)-(2.22)]{B}, we have the Riemann curvature tensors\footnote{Bartnik made a sign error in the formula of $\langle R(\pl_a,\pl_b)\pl_c, \nu \rangle$. }
\begin{align*}
\langle R(\pl_a,\pl_b)\pl_c, \pl_d \rangle &= r^2 (-1+ u^{-2}) \lt( \sigma_{ac}\sigma_{bd} - \sigma_{ad}\sigma_{bc} \rt) \\ 
\langle R(\pl_a,\pl_b)\pl_c, \nu \rangle &= ru^{-2} \lt( -u_b \sigma_{ac} + u_a \sigma_{bc} \rt) \\
\langle R(\pl_a,\nu)\pl_b,\nu \rangle &= u^{-1} \na_a\na_b u - u^{-3} r u_r \sigma_{ab}.
\end{align*}
The Ricci curvature $R(X,Y) = \sum_{i=1}^n \langle R(X,e_i)e_i, Y \rangle$ is given by
\begin{align*}
R(\pl_a,\pl_b) &= -u^{-1}\na_a\na_b u + r u^{-3}u_r \sigma_{ab} + (n-2)(1-u^{-2})\sigma_{ab}\\
R(\pl_a,\nu) &= \frac{n-2}{r}u^{-2} u_a\\
R(\nu,\nu) &= - \frac{1}{r^2} u^{-1} \Delta u +  \frac{n-1}{r} u^{-3}u_r.
\end{align*}
The scalar curvature is given by
\begin{align*}
R(g) = \frac{1}{r^2} \lt( 2 u^{-1} \Delta u + 2(n-1)r u^{-3} u_r + (n-1)(n-2)(1 - u^{-2}) \rt)
\end{align*}
and \eqref{zero scalar curvature} follows.

We proceed to compute the Hessian of $V$. We have
\begin{align*}
D_{\pl_a} \nu &= \frac{1}{r}u^{-1} \pl_a\\
D_{\pl_a} \pl_b &= \Gamma_{ab}^c \pl_c - u^{-1}r \sigma_{ab} \nu\\
D_{\pl_r} \pl_r &= u^{-1}u_r \pl_r - \frac{1}{r^2} u \na^a u \pl_a
\end{align*}
where $\Gamma_{ab}^c$ are the Christoffel symbols of $\sigma$. We immediately get \eqref{ra} and \eqref{ab}. \eqref{rr} follows by taking \eqref{zero scalar curvature} into account.
\end{proof}

In the rest of this section, we present definitions and elementary facts for spherical harmonics on $S^{n-1}$.
\begin{defn}\label{spherical harmonics}
The eigenvalues of the Laplace operator $\Delta$ on $S^{n-1}$, $\Delta f = - \lambda f$, are
\begin{align*}
\lambda = \ell (\ell + n -2), \ell \in \mathbb{N}_{\ge 0}.
\end{align*}

Let $\mathcal{H}_{k}$ ($\mathcal{H}_{\ell \le k}$ resp.) denote the eigenspace of the eigenvalue $k(k+n-2)$ (the direct sum of eigenspaces of eigenvalue $\le k(k+n-2)$ resp.).  Eigenfunctions in $\mathcal{H}_{k}$ are said to be of {\it mode $k$}. Note that $\mathcal{H}_k$ form an orthogonal decomposition of $L^2(S^{n-1}) = \oplus_{k=0}^\infty \mathcal{H}_k$. Let $f_{\ell = k}$ or $f_k$ denote the projection of $f = \sum_{k=0}^\infty f_k$ to $\mathcal{H}_k$.
\end{defn}

\begin{prop}\label{characterizing Hle1}
Let $f$ be a $C^2$ function on $S^{n-1}$ satisfying $\na_a\na_b f = g \sigma_{ab}$ for some function $g$. Then $f$ is smooth and lies in $\mathcal{H}_{\ell\le 1}$.
\end{prop}
\begin{proof}
Taking the trace, we get $g = \frac{\Delta f}{n-1}$. For any smooth function $\zeta$, multiplying the equation by $\na^a\na^b \zeta$ and integrating by parts yields $\int_{S^{n-1}} (\Delta + n-1)f \cdot \Delta\zeta  =0$. Indeed, using Ricci identity, we have $\int \na_a\na_b f \na^a \na^b \zeta =  \int  -\na^b\na_a\na_b f \na^a \zeta = \int [-\na_a \Delta f - (n-2)\na_a f] \na^a \zeta = \int [(\Delta f + (n-2)f)]\Delta \zeta$.

Therefore $(\Delta+n-1)f$ is weakly harmonic; elliptic regularity implies that it is smooth and the maximum principle says it is constant. It follows that $f$ is smooth and lies in $\mathcal{H}_{\ell\le 1}$.
\end{proof}
\begin{prop}\label{product of first eigenfunctions}
If $f,g \in \mathcal{H}_{\ell\le 1}$, then $(\na f \cdot \na g)_{\ell=2} = - (fg)_{\ell=2}$.
\end{prop}
\begin{proof}
Let $\vphi$ be an eigenfunction in $\mathcal{H}_{\ell=2}.$ Integration by parts yields 
\begin{align*}
\int_{S^{n-1}} |\na f|^2 \vphi = \int_{S^{n-1}} - f \Delta f \vphi - \frac{1}{2} (\na f^2) \cdot \na\vphi = \int_{S^{n-1}} (n-1)f^2 \vphi - n f^2 \vphi
\end{align*}
and the assertion follows by polarization.
\end{proof}

We recall the basis of $\mathcal{H}_{\ell \le 1}$ and $\mathcal{H}_{2}.$ Let $x^i, i = 1, \ldots, n$ be the coordinate functions of $\R^n$ and $X^i = x^i \big|_{S^{n-1}}$. We have that
\begin{enumerate}
\item $\{ X^i: i =1, \ldots, n \}$ together with constant function $1$ form an orthogonal basis of $\mathcal{H}_{\ell\le 1}$.
\item $\mathcal{H}_2$ is spanned by $\{ X^iX^j - \frac{1}{n} \delta^{ij}: 1 \le i \le j \le n\}$; $\dim \mathcal{H}_2 = \frac{n(n+1)}{2}-1$: for any $n \times n$ symmetric matrix $A$, $\sum_{i,j=1}^n A_{ij} (X^iX^j - \frac{1}{n}\delta^{ij}) =0$ if and only if $A$ is a multiple of the identity matrix. 
\end{enumerate}
As a result, we have
\begin{prop}\label{udot_square_l=2}
If $\lt( a_0 + \sum_{i=1}^n a_i X^i \rt)^2_{\ell=2} =0$, then $a_i=0$ for $i = 1, \ldots, n$.
\end{prop}
\begin{proof}
We compute
\[ (a_0 + \sum_{i=1}^n a_i X^i)^2 = (a_0^2 + \frac{1}{n} \sum_{i=1}^n a_i^2) + 2 a_0 \sum_{i=1}^n a_i X^i + \sum_{i,j=1}^n a_i a_j (X^iX^j - \frac{1}{n}\delta^{ij}). \]
 By point 2 above, $\lt( a_0 + \sum_{i=1}^n a_i X^i \rt)^2_{\ell=2} =0$ implies that $a_ia_j = c \delta_{ij}$ and thus $a_i=0$.
\end{proof}
Finally, we specify the big-O notation used in the rest of the paper.
\begin{defn}
Writing $f \in O_k (g(r))$ means for all $i,j \ge 0,$ $i+j \le k,$ and $r \ge r_0$, there are constants $C_{i,j}$ such that
\begin{align*}
|(r\pl_r)^i \na^j f| \le C_{i,j} g(r).
\end{align*}
Also define $O_\infty(g) = \cap_{k=1}^\infty O_k(g)$. $f \in o_k (g(r))$ is defined similarly with
\[ \lim_{r\rw\infty} \frac{|(r\pl_r)^i \na^j f|}{g(r)} =0.\]
\end{defn}

\section{The asymptotics of the metric coefficient $u$}\label{section:asymptotics of the lapse}
\subsection{3-dimensional case}
In Section 5 of \cite{B}, Bartnik analyzed the zero scalar curvature equation \eqref{zero scalar curvature} in 3-dimension.  
\begin{theorem}\cite[Theorem 5.1]{B}\label{Bartnik}
If the metric $g = u^2(r,\theta)dr^2 + r^2 g_{S^2}$ has zero scalar curvature, then there is a constant $m$ and $\dot u \in \mathcal{H}_{\ell=1}$ such that\footnote{$m$ is denoted by $m_0$ and $\dot u$ is denoted by $\vphi$ in \cite{B}.} 
\[ u = 1 + m r^{-1} + (\dot u+ \frac{3}{2}m^2) r^{-2} + \ddot u r^{-3} + O_{\infty}(r^{-4}\log r) \]
where $\ddot u \in \mathcal{H}_{\ell\le 1}$ is determined by $m$ and $\dot u$.

\end{theorem}

We obtain a further expansion following the treatment on page 69-70 of \cite{B}.
\begin{theorem}\label{u_3}
Under the assumptions of Theorem \ref{Bartnik}, there are $\hat u, \dddot{u} \in \mathcal{H}_{\ell\le 2}$ such that 
\begin{align*}
u = 1 + m r^{-1} + (\dot u+ \frac{3}{2}m^2) r^{-2} + \ddot{u} r^{-3} + \hat u r^{-4}\log r + \dddot{u} r^{-4} + O_\infty(r^{-5}\log r)
\end{align*}
where \begin{align}\label{hat u l=2} \hat u_{\ell=2} = - \frac{7}{2}(\dot u^2)_{\ell=2} 
\end{align} and $\hat{u}_{\ell\le 1}, \dddot u_{\ell \le 1}$ are both determined by $m$ and $\dot u$.
\end{theorem}
\begin{proof}
Plug $u = 1 + m r^{-1} + \dot u r^{-2} + \ddot u r^{-3} +f$ into \eqref{zero scalar curvature}, we get
\begin{align}\label{temp1}
\Delta f - 2r f_r - 2f + r^{-4} \lt( 2\dot u \Delta \dot u - 3\dot u^2 + \mathcal{H}_{\ell\le 1}\rt) = O_\infty(r^{-5}\log r)
\end{align}
where the term $\mathcal{H}_{\ell\le 1}$ involves only $m$ and $\dot u$.

Let $\vphi$ be an eigenfunction in $\mathcal{H}_{\ell=2}$ and $M(r) = \int_{S^2} f\vphi$. Since $\Delta \dot u = -2\dot u$ and $\int_{S^2} \Delta f \vphi = \int_{S^2} f \Delta\vphi = -6 \int_{S^2} f\vphi$, after multiplying \eqref{temp1} by $\vphi$, we get
\begin{align*}
2r M' + 8 M = r^{-4}  N + O(r^{-5}\log r)
\end{align*}
where $N = -\int_{S^2} 7\dot  u^2 \vphi$. Integrating this ODE from $r$ to $\infty$, we get
\begin{align*}
r^4 M(r) = c + \frac{N}{2} \log r + O(r^{-1}\log r)
\end{align*}
where $c$ is an arbitrary constant. This proves \eqref{hat u l=2} and the fact $\dddot u_{\ell=2}$ is arbitrary. 

Multiplying \eqref{temp1} by an eigenfunction $\psi$ of mode $\ell \ge 3$, we see that $N(r) = \int_{S^2} f\psi$ satisfies 
\[ -2r N' - (\ell (\ell+1) + 2) N = O(r^{-5}\log r). \] Integration of this ODE yields $u_{\ell \ge 3} = O(r^{-5}\log r)$. The assertion of $\hat u_{\ell\le 1}$ and $\dddot u_{\ell\le 1}$ follows similarly. We conclude that $\hat u, \dddot u \in \mathcal{H}_{\ell\le 2}$ and there is no other term before $r^{-5}\log r$ in the expansion of $u$.
\end{proof}
\subsection{Higher-dimensional case}
In \cite{ST}, Shi-Tam generalized Barnik's result to general dimensions and a wider class of metrics. A special case is  
\begin{theorem}\cite[Lemma 2.8]{ST}\label{ShiTam}
In dimension $n\ge 3$, if the metric $u^2 dr^2 + r^2 g_{S^{n-1}}$ has zero scalar curvature, then 
\[ u = 1 + mr^{2-n} + O_1(r^{1-n}) \]
where $m$ is a constant.
\end{theorem}

We first improve the differentiability of the remainder using Schauder estimates.
\begin{theorem}\label{ShiTam_improve}
Under the assumption of Theorem \ref{ShiTam}, $u$ satisfies
\begin{align*}
u = 1 + m r^{2-n} + O_\infty(r^{1-n}).
\end{align*}
\end{theorem}
\begin{proof}
Let $r = e^t$. Plugging $u = 1 + m r^{2-n} + f$ into \eqref{zero scalar curvature}, we see that the remainder $f$ satisfies the parabolic equation
\begin{align*}
(n-1) \pl_t f - u^2 \Delta f = -(n-1)(n-2)f -\frac{3}{2}(n-1)(n-2) m^2 r^{4-2n} + O_1(r^{3-2n}).
\end{align*}
Schauder estimates then imply $\| f \|_{C^{2+\alpha, 1+\alpha/2}} = O(r^{1-n})$ for any $0 <\alpha < 1$. 

Similarly, applying Schauder estimates to the parabolic equation satisfied by $\pl_t f$,  we get $\| \pl_t f\|_{C^{2+\alpha, 1 + \alpha/2}} = O(r^{1-n})$. In particular, $f \in O_2(r^{1-n})$. For the differentiation in space, we take a Killing vector field $X$ on $S^{n-1}$ to get
\begin{align*}
(n-1)\pl_t (Xf) - u^2 \Delta (Xf) = -(n-1)(n-2)Xf + X(u^2)\Delta f + O_1 (r^{3-2n})
\end{align*}
and hence $\| Xf\|_{C^{2+\alpha,1+\alpha/2}} = O(r^{1-n})$. Since the Killing vector fields span the tangent space of every point of $S^{n-1}$, we have $\| \na f\|_{C^{2+\alpha, 1 + \alpha/2}} = O(r^{1-n})$. Successively differentiating in time and space leads to $f = O_\infty(r^{1-n})$. 
\end{proof}

\begin{lem}\label{u_higher}
In dimension $n \ge 4$, if the metric $u^2 dr^2 + r^2 g_{S^{n-1}}$ has zero scalar curvature, then $u$ has the following asymptotic behavior \begin{align*}
u = 1 + m r^{2-n} + \dot u r^{1-n} + \frac{3}{2}m^2 r^{4-2n} + O_\infty( r^{-n-\frac{2}{n-1}} )
\end{align*}
where $\dot u \in \mathcal{H}_{\ell=1}$ is an arbitrary eigenfunction. Moreover, \begin{align*}
u_{\ell= 2} = \hat u r^{-n-\frac{2}{n-1}} + \dddot u r^{2-2n} + O_\infty(r^{2-2n-\frac{2}{n-1}})
\end{align*}
where $\hat u \in \mathcal{H}_{\ell=2}$ is an arbitrary eigenfunction and \begin{align*}
\dddot u = \frac{\frac{3}{2}(n-1)(n-2)+2(n-1)}{n(n-3)} (\dot u^2)_{\ell=2}.
\end{align*}
\end{lem}
\begin{proof}
Plugging $u = 1 + m r^{2-n} + f$ into the zero scalar curvature equation \eqref{zero scalar curvature}, we get
\begin{align*}
\Delta f + O(r^{3-2n}) - (n-1)r \lt( (2-n)mr^{1-n} + f_r \rt) \\
+ \frac{(n-1)(n-2)}{2} \lt( -2m r^{2-n} -2f - 3m^2 r^{4-2n} + O(r^{3-2n}) \rt) =0.
\end{align*}
Hence $F = f - \frac{3}{2}m^2 r^{4-2n} $ satisfies
\begin{align*}
\Delta F -(n-1)r F_r - (n-1)(n-2)F = O(r^{3-2n}).
\end{align*}

Let $M = \int_{S^{n-1}} F \vphi$ for an eigenfunction $\vphi \in \mathcal{H}_\ell$. We have $M = O(r^{1-n})$,  
\begin{align*}
-\ell (\ell + n-2) M - (n-1)r M_r - (n-1)(n-2) M = O(r^{3-2n}),
\end{align*}
and integration yields
\begin{align*}
M = c r^{-1-\alpha} + O(r^{3-2n}), \quad 1 + \alpha = n-2 + \frac{\ell(\ell+n-2)}{n-1}
\end{align*}
for some constant $c$. For $\ell=0$, we have $c=0$ since $M = O(r^{1-n})$. For $\ell=1$ we have $1+\alpha = n-1$ and hence the existence of $\dot u$. For $\ell \ge 2$, we have $1 + \alpha \ge n+ \frac{2}{n-1}$. This proves the first statement.

We plug $u = 1 + m r^{2-n} + \dot u r^{1-n} + \frac{3}{2}m^2 r^{4-2n} + \mathcal{F}$ into the zero scalar curvature equation to get
\begin{align*}
&\Delta\mathcal{F} -(n-1)r \mathcal{F}_r - (n-1)(n-2)\mathcal{F}\\
&= \lt( -2m\Delta\dot u + 3(n-1)(n-2)m\dot u \rt)r^{3-2n} + \lt( \frac{3}{2}(n-1)(n-2) \dot u^2 - 2\dot u \Delta\dot u \rt) r^{2-2n} \\
&\quad + \frac{9}{2}(n-1)(n-2) m^3 r^{6-3n} + O(r^{2-2n - \frac{2}{n-1}}).
\end{align*}
Let $\vphi \in \mathcal{H}_{\ell=2}$ be an eigenfunction and $M(r) = \int_{S^{n-1}} \mathcal{F} \vphi$. We get 
\begin{align*}
-(n-1) r M' - (n^2-n+2)M = \lt[ \frac{3}{2}(n-1)(n-2) + 2(n-1) \rt] r^{2-2n}\int_{S^{n-1}} \dot u^2 \vphi \\ + O(r^{2-2n - \frac{2}{n-1}}).
\end{align*}
Integrating the ODE, we obtain
\begin{align*}
\mathcal{F}_{\ell=2} = \hat u r^{-n - \frac{2}{n-1}} + \frac{\frac{3}{2}(n-1)(n-2)+2(n-1)}{n(n-3)}\lt( \dot u^2\rt)_{\ell=2} r^{2-2n} + O(r^{2-2n - \frac{2}{n-1}}) 
\end{align*}
where $\hat u$ is an arbitrary eigenfunction in $\mathcal{H}_{\ell=2}$. 
\end{proof}

\section{Rotational symmetry to the order $r^{1-n}$}\label{section:rotational symmetry at the order}
The goal of this section is to show that $u$ and $V$ are rotationally symmetric to the order of $r^{1-n}$. First of all, $(M^{n},g)$ of the form \eqref{quasi-spherical} with scalar curvature $R_{g}=0$ are asymptotically flat of order $\tau = n-2$ by the expansion of the metric coefficient $u$. The asymptotic behavior of a static potential $V$ is well-understood on such a background. In fact, according to Proposition B.4 of \cite{HMM}, if $V >0$ on $M^{n}$ then we have after scaling that
\begin{lem}
\begin{align*}
    V = 1 - m r^{2-n} + O_2(r^{1-n})
\end{align*}
where $m$ is the constant in the expansion of the metric coefficient $u$.
\end{lem}

Next, we derive further expansions of $V$. It suffices to begin with $V = 1 - mr^{2-n} + o_1(r^{2-n})$.
\begin{lemma}\label{Vdot}
We have
\begin{align*}
V = 1 - m r^{2-n} - \frac{n-2}{n}\dot u r^{1-n} -\frac{1}{2} m^2 r^{4-2n} + O_2(r^{-\nu}), \quad \nu = \begin{cases} 3, & n=3\\ n + \frac{2}{n-1}, & n \ge 4 \end{cases}
\end{align*}
\end{lemma}
\begin{proof}
Write $V = 1 - mr^{2-n} + f, f \in o_1(r^{2-n})$. We will use $u = 1 + mr^{2-n} + \dot u r^{1-n} + \frac{3}{2}m^2 r^{4-2n} + o_1(r^{1-n})$ at this stage; the true decay $O(r^{-\nu})$ will be used later. Examining the $rr$-component of the static equations \eqref{rr}, we have on the left-hand side
\begin{align*}
V_{rr} &= -(n-1)(n-2)m r^{-n} + f_{rr}\\
-u^{-1}u_r V_r &= - (1 + O_\infty(r^{2-n}))(-(n-2)m r^{1-n} + O_\infty(r^{-n}))((n-2)m r^{1-n} + o(r^{1-n})) \\ &= (n-2)^2m^2 r^{2-2n} + o(r^{2-2n})\\
-\frac{u}{r^2}\na^a u \na_a V &= - r^{-2} (1 + O_\infty(r^{2-n})) (\na^a \dot u r^{1-n} + o(r^{1-n})) \cdot o(r^{2-n}) \\ &= o(r^{1-2n})
\end{align*} and right-hand side is given by
\begin{align*}
\mbox{RHS} &= \frac{(n-1)(n-2)}{2} r^{-2} \lt( -2nr^{2-n} - 2\dot u r^{1-n} - 4m^2 r^{4-2n} + o_1(r^{1-n}) \rt) \cdot \\ & \hspace{10cm} (1 - m r^{2-n} + o_1(r^{2-n}) ) \\
&= -(n-1)(n-2)mr^{-n} - (n-1)(n-2)m^2 r^{2-2n} \\
&\quad - (n-1)(n-2)\dot u r^{-1-n} + o_1(r^{-1-n}).
\end{align*}
We thus obtain
\begin{align*}
f_{rr} = - (n-1)(n-2) \dot u r^{-1-n} - (n-2)(2n-3)m^2 r^{2-2n} + o(r^{-1-n}).
\end{align*}
Integration in $r$ yields
\begin{align*}
f_r = \frac{(n-1)(n-2)}{n} \dot u r^{-n} + (n-2) m^2 r^{3-2n} + o(r^{-n})
\end{align*}
and then
\begin{align*}
f = - \frac{n-2}{n}\dot u r^{1-n} - \frac{1}{2}m^2 r^{4-2n} + o(r^{1-n}).
\end{align*}
We thus conclude that $V = 1 - m r^{2-n} - \frac{n-2}{n}\dot u r^{1-n} - \frac{1}{2}m^2 r^{4-2n} + F$ with $F = o(r^{1-n})$.

By Theorem \ref{Bartnik} and Lemma \ref{u_higher}, we write $u = 1 + m r^{2-n} + \dot u r^{1-n} + \frac{3}{2} m^2 r^{4-2n} + O_\infty(r^{-\nu})$, which holds in all dimensions, to improve the decay of $F$. In the $rr$-component of the static equations, the improvement happens at two terms:
\begin{align*}
-u^{-1}u_r V_r &= \cdots + O(r^{1-2n}),\\
\mbox{RHS} &= \cdots +  O_1(r^{-2-\nu}).
\end{align*}
As $1-2n \le -2-\nu$, we obtain $F_{rr} = O(r^{-2-\nu})$ and hence $F_r = O(r^{-1-\nu}), F= O(r^{-\nu})$. 

Next we consider the $ra$-component of the static equation. We compute
\begin{align*}
-u^{-1}u_a V_r &= O(r^{2-2n})\\
-\frac{1}{r} \pl_a V &= \frac{n-2}{n}\pl_a \dot u r^{-n} - \frac{1}{r} \pl_a F\\
\frac{n-2}{r}Vu^{-1} u_a &= (n-2)\pl_a \dot u r^{-n} + O(r^{-1-\nu})
\end{align*}
to get
\begin{align*}
\pl_r \pl_a F - \frac{1}{r} \pl_a F = O(r^{-1-\nu}).
\end{align*}
Integration in $r$ yields $\pl_a F = O(r^{-\nu})$ and then $\pl_r \pl_a F = O(r^{-1-\nu})$.

Finally consider $ab$-component of the static equations. We look at terms in front of $\sigma_{ab}$  
\[ u^{-2} \lt( -r V_r + r Vu^{-1}u_r\ + (n-2) V(u^2-1) \rt) \]
and compute
\begin{align*}
-r V_r &= -(n-2)mr^{2-n} - \frac{(n-1)(n-2)}{n}\dot u r^{1-n} - (n-2)m^2 r^{4-2n} + O(r^{-\nu})\\
Vu^{-1} \cdot ru_r &= (1- 2m r^{2-n} + O(r^{1-n})) \cdot \\
&\qquad\qquad  \lt( - (n-2)m r^{2-n} - (n-1)\dot u r^{1-n} - 3(n-2)m^2 r^{4-2n} + O_\infty(r^{\nu}) \rt)\\
V (u^2-1) &= (1 - mr^{2-n} + O_1 (r^{1-n})) \lt( 2m r^{2-n} + 2\dot u r^{1-n} + m^2 r^{4-2n} + O_\infty(r^{\nu}) \rt)
\end{align*}
Since $r^{3-2n}$ decays faster than $r^{-\nu}$, we obtain
\begin{align*}
u^{-2} \lt( -r V_r + r Vu^{-1}u_r\ + (n-2) V(u^2-1) \rt) = -\frac{2}{n}\dot u r^{1-n} + O(r^{-\nu}).
\end{align*}
On the other hand, since $\dot u$ is an eigenfunction in $\mathcal{H}_{\ell=1}$, we have $\na_a\na_b \dot u = - \dot u \sigma_{ab}$ and hence
\begin{align*}
\na_a\na_b V + Vu^{-1} \na_a\na_b u = \frac{n-2}{n}\dot u \sigma_{ab} + \na_a\na_b F - \dot u \sigma_{ab} + O_\infty(r^{-\nu}).
\end{align*}
Therefore, $\na_a\na_b F = O(r^{-\nu})$ and the proof of $F = O_2(r^{-\nu})$ is completed.
\end{proof}
Next we separate the cases of 3-dimension and higher dimensions $n \ge 4$.
\begin{lem}\label{V_3}
In $3$-dimension, we have
\begin{align*}
V = 1 - mr^{-1} + (\dot V-\frac{1}{2}m^2) r^{-2} + \ddot V r^{-3} + \hat V r^{-4}\log r + \dddot V r^{-4} + O_2(r^{-5}\log r) 
\end{align*}
where
\begin{enumerate}
\item $\dot V = -\frac{1}{3} \dot u$ and $\ddot V \in\mathcal{H}_{\ell\le 1}$ depends linearly on $\dot V, \dot u, \ddot u, m^2$.
\item $\hat V = - \frac{1}{10} \hat u$ and 
$\dddot V$ is in $\mathcal{H}_{\ell\le 2}$.
\end{enumerate}
\end{lem}
\begin{proof}
Plugging $V = 1 - mr^{-1} + (\dot V-\frac{1}{2}m^2) r^{-2}+ \mathfrak{F}$ into the $rr$-component of the static equations and using $u = 1 + mr^{-1} + (\dot u + \frac{3}{2}m^2) r^{-2} + \ddot u r^{-3} + O(r^{-4}\log r)$, one sees that
$\mathfrak{F} = + \ddot V r^{-3} + f$ with $f = O_1(r^{-4}\log r)$ and $\ddot V$ depending linearly on $\dot V, \dot u, \ddot u, m^2$.

Now using Theorem \ref{u_3} and the $rr$-component of the static equation, we get
\begin{align*}
&f_{rr} -4 \dot u \dot V r^{-6} - \na^a \dot u \na_a \dot V  r^{-6} \\
&= - 2 \hat u r^{-6}\log r  + \lt[ -2 \dot u \dot V -2 \dddot u- \dot u^2 + \mathcal{E}\rt] r^{-6} + O(r^{-7}\log r)
\end{align*}
with $\mathcal{E} \in \mathcal{H}_{\ell\le 1}$. Integration in $r$ yields
\begin{align*}
f_r =  \hat u \lt( \frac{2}{5} r^{-5}\log r + \frac{2}{25} r^{-5} \rt) - \frac{1}{5} \lt( 2\dot u\dot V + \na^a \dot u \na_a \dot V - 2 \dddot u - \dot u^2 + \mathcal{E} \rt)r^{-5} + O(r^{-6}\log r)
\end{align*}
and
\begin{align*}
f = -\frac{1}{10} \hat u r^{-4}\log r + \lt( -\frac{9}{200}\hat u + \frac{1}{20} \lt( 2\dot u\dot V + \na^a \dot u \na_a \dot V - 2 \dddot u - \dot u^2 + \mathcal{E}\rt) \rt) r^{-4} + O(r^{-5}\log r).
\end{align*}

The decay of the derivatives of $V$ can be improved as in Lemma \ref{Vdot} and we omit the details.
\end{proof}
\begin{lem}\label{V_high}
In dimensions $n \ge 4$, we have
\begin{align*}
V = 1 - m r^{2-n} + \dot V r^{1-n} - \frac{1}{2} m^2 r^{4-2n} + O_2(r^{-n-\frac{2}{n-1}})
\end{align*}
with $\dot V = -\frac{n-2}{n}\dot u$. Moreover, 
\begin{align*}
V_{\ell= 2} = \hat V r^{-n-\frac{2}{n-1}} + \dddot V r^{2-2n} + o_2(r^{2-2n})
\end{align*}
where
\begin{align*}
\hat V = -\frac{(n-1)(n-2)}{\lt( 1+n+\frac{2}{n-1}\rt)\lt( n+ \frac{2}{n-1} \rt)} \hat u, \quad \dddot V = - \frac{(n-2)^2}{n(2n-1)(2n-2)}  (\dot u^2)_{\ell=2} - \frac{n-2}{2(2n-1)} \dddot u.
\end{align*}
\end{lem}
\begin{proof}
Consider the $\ell \ge 2$-modes of the $rr$-component of the static equations with $V = 1 - m r^{2-n} + \dot V r^{1-n} - \frac{1}{2} m^2 r^{4-2n} + f$ and $u = 1 + m r^{2-n} + \dot u r^{1-n} + \frac{3}{2}m^2 r^{4-2n} + O_\infty(r^{-n-\frac{2}{n-1}})$. The $\ell\ge 2$ modes of the left-hand side are given by
\begin{align*}
(f_{\ell\ge 2})_{rr} -(n-1)^2 (\dot u \dot V)_{\ell=2} r^{-2n} - (\na^a \dot u \na_a \dot V)_{\ell=2} r^{-2n} + O(r^{-2n-\frac{2}{n-1}})
\end{align*}
while the $\ell=2$ mode of the right-hand side is given by
\begin{align*}
- (n-1)(n-2) (\dot u \dot V)_{\ell=2} - (n-1)(n-2)\hat u r^{-2-n-\frac{2}{n-1}} - (n-1)(n-2) \dddot u r^{2-2n} \\+ O(r^{-2n - \frac{2}{n-1}}).
\end{align*}
By Proposition \ref{product of first eigenfunctions}, we get
\begin{align*}
( f_{\ell=2} )_{rr} = - (n-1)(n-2)\hat u r^{-2-n-\frac{2}{n-1}} + (n-2)(\dot u \dot V)_{\ell=2} r^{-2n} - (n-1)(n-2) \dddot u r^{2-2n} \\ + O(r^{-2n - \frac{2}{n-1}})
\end{align*}
and the assertion follows from integration.
\end{proof}

We are in position to prove that $\dot u$ is constant. In $3$-dimension, the leading order of the $ab$-component of the static equations \eqref{ab} is at $r^{-4}\log r$ and is given by 
\[ \na_a \na_b (\hat V + \hat u) = g_1 \sigma_{ab} \]
for some function $g_1$. Therefore, $\hat V + \hat u$ is in $\mathcal{H}_{\ell\le 1}$ by Proposition \ref{characterizing Hle1}. By Lemma \ref{V_3}, both $\hat u$ and $\hat V$ are in $\mathcal{H}_{\ell\le 1}$. Theorem \ref{u_3}, Proposition \ref{udot_square_l=2}, and the fact that $\dot u \in \mathcal{H}_{\ell=1}$ imply $\dot u_{\ell=1}=0$.

In higher dimensions, we consider the order $r^{2-2n}$. Let the coefficients of $V$ and $u$ at this order be $\tilde V$ and $\tilde u$. As $Vu^{-1} = 1 - 2mr^{2-n} + O(r^{1-n})$ and $u_{\ell\ge 2}$ starts $r^{-n-\frac{2}{n-1}}$ we see that at the order $r^{2-2n}$, $-Vu^{-1}\na_a\na_b u$ gives rise to $-\na_a\na_b \tilde u r^{2-2n}.$ The $ab$-component of the static equations thus implies 
\begin{align*}
\na_a\na_b \lt( \tilde V + \tilde u \rt) = g_3 \sigma_{ab}
\end{align*}
for some function $g_3$. Therefore $\tilde V + \tilde u$ is in $\mathcal{H}_{\ell\le 1}$ by Proposition \ref{characterizing Hle1} and in particular its $\ell=2$ mode vanishes:
\begin{align*}
\lt[ \lt( 1 - \frac{n-2}{2(2n-1)}\rt)\frac{ \frac{3}{2}(n-1)(n-2)+4n}{n^2-3n} - \frac{(n-2)^2}{n(2n-1)(2n-2)} \rt] (\dot u^2)_{\ell=2} =0.
\end{align*} 
As the coefficient is positive when $n \ge 4$, we get $(\dot u^2)_{\ell=2}=0$ and hence $\dot u_{\ell=1}=0$ again by Proposition \ref{udot_square_l=2}.

Having shown that $\dot u$ is constant in all dimensions, it follows that $\dot V $ is also constant in all dimensions by Lemma \ref{V_3} and Lemma \ref{V_high} respectively.
 
\section{Proof of Theorem \ref{main} and Theorem \ref{equality case}}\label{section:proof of main theorem}
\subsection{Rotational symmetry outside a compact set}
Given a function $f(r,\theta)$, define 
\begin{align*}
\check f(r,\theta) = f(r,\theta) - \lt( \frac{1}{|S^{n-1}|} \int_{S^{n-1}} f \rt)(r)
\end{align*}
to be its non-rotationally symmetric part. Consider $\check u$ and $\check V$. From the definition, we have $\pl_a \check u = \pl_a u$, $\pl_a \check V = \pl_a V$, and  \begin{align}\label{cucV no 0 mode}
\int_{S^{n-1}} \check u(r, \cdot) = \int_{S^{n-1}} \check V(r, \cdot)=0 \mbox{ for any } r.
\end{align} 
We will write 
\begin{align*}
u = u_0(r) + \check u, \quad V = V_0(r) + \check V
\end{align*}
The main result of this subsection is
\begin{lem}\label{vanish out a compact set}
There exists a constant $r_1 > r_0$ that depends only on $n,m,$ and $r_0$ such that $\cu = \cV =0$ on $[r_1,\infty) \times S^{n-1}$.
\end{lem}
\begin{proof}

We begin by deriving the equations satisfied by $\cu$ and $\cV$. Firstly, multiplying the $ra$-component of the static equation \eqref{ra} by $u^{-2}$ and rewrite it as
\begin{align*}
r \pl_a \lt( u^{-2} V_r \rt) + ru^{-3} V_r \pl_a \cu - \lt( u^{-2}-1\rt) \pl_a \cV - \pl_a \cV = (n-2)\lt( Vu^{-3}-1 \rt)\pl_a \cu + (n-2)\pl_a \cu.
\end{align*}
Hence, we have
\begin{align}
ru^{-2} V_r - \cV - (n-2) \cu = a(r) + F \label{ra_rot}
\end{align}
where $F$ and $a(r)$ are uniquely determined by 
\begin{align}\label{F}
\pl_a F = - r u^{-3} V_r \pl_a \cu + (u^{-2}-1)\pl_a\cV + (n-2)(Vu^{-3}-1)\pl_a\cu, \quad \int_{S^{n-1}} F(\cdot, r) =0 \mbox{ for any }r.
\end{align}
Secondly, rewriting the $ab$-component of the static equations \eqref{ab} as
\begin{align*}
\na_a \na_b \cV + \na_a\na_b \cu = (1 - Vu^{-1})\na_a\na_b \cu + \lt[ -ru^{-2} V_r + V ru^{-3}u_r + (n-2)V (1 - u^{-2}) \rt]\sigma_{ab}
\end{align*}
and taking the divergence, we get
\begin{align*}
\na_a (\Delta+n-2)(\cV+\cu) &= \na^b \lt[ (1-Vu^{-1})\na_a\na_b \cu \rt] \\ &\qquad + \pl_a \lt[ -ru^{-2} V_r + V ru^{-3}u_r + (n-2)V(1 - u^{-2}) \rt]
\end{align*}
and hence
\begin{align}\label{ab_rot}
(\Delta+n-2)(\cV+\cu) = b(r) - ru^{-2}V_r + Vru^{-3}u_r + (n-2)V(1 - u^{-2}) + G
\end{align}
where $G$ and $b(r)$ are uniquely determined by
\begin{align}\label{G}
\na_a G = \na^b \lt[ (1-Vu^{-1})\na_a\na_b \cu \rt], \quad \int_{S^2} G(\cdot, r)=0 \mbox{ for any }r.
\end{align}
Thirdly, adding the equation $\Delta_g V =0$ and the $rr$-component of the static equation yields
\begin{align}\label{rr_rot}
\Delta \cV + (n-1)r u^{-2} V_r + \frac{(n-1)(n-2)}{2}(u^{-2}-1)V =0.
\end{align}
Finally, rewrite the zero scalar curvature equation
\begin{align}\label{zeroscalar_rot}
\Delta \cu + (u^2-1)\Delta \cu - (n-1)r u_r + \frac{(n-1)(n-2)}{2} (u - u^3) =0.
\end{align}

Subtracting equations \eqref{rr_rot} and \eqref{zeroscalar_rot} from the equation \eqref{ab_rot} to eliminate $\Delta \cV$, $\Delta \cu$ and adding $(n-2)$ times equation \eqref{ra_rot} to further eliminate $\cV, ru^{-2}V_r$, we obtain
\begin{align*}
&-(n-2)(n-3)\cu  - \frac{(n-2)(n-3)}{2} (u^{-2}-1)V + (n-1)r u_r - \frac{(n-1)(n-2)}{2}(u - u^3) \\
&= (n-2)F + (n-2)a(r) + b(r) + Vr u^{-3} u_r + G + (u^2-1)\Delta\cu \end{align*} 
Extracting the leading term (the first term on the right-hand side of each line) \begin{align*}
V r u^{-3} u_r &= ru_r + (Vu^{-3}-1) ru_r,\\
ru_r &= r \cu_r + r(u - \cu)_r,\\
(u^{-2}-1)V &= (u^{-2}-1 + 2\cu - 2 \cu)V = -2 \cu + 2 (1-V) \cu + (u^{-2}-1+2\cu ) V,
\end{align*}
we get
\begin{align*}
(n-2) r \cu_r + (n-1)(n-2) \cu &= (n-2)a(r) + b(r) + (n-2)F + G \\
&\quad + (Vu^{-3}-1) r u_r  - (n-2)r (u - \cu)_r \\ & \quad + (u^2-1)\Delta \cu + \frac{(n-1)(n-2)}{2}(u - u^3 + 2\cu). 
\end{align*}
We now abuse the notion of modes: For a function $f(r,\theta)$, $f_{\ell=0}(r)$ is the function with $\int_{S^{n-1}} [f(r,\cdot) - f_{\ell=0}(r)] =0$ for any $r$ and $f_{\ell\ge 1}(r,\theta) = f(r,\theta) - f_{\ell=0}(r)$. In this terminology, $\cu, \cu_r, F, G$ all have no 0-mode and we obtain, projecting the above formula onto $\ell \ge 1$ modes, 
\begin{align}\label{ODE for cu}
r \cu_r + (n-1) \cu &= F+ \frac{G}{n-2} + I 
\end{align}
where
\begin{align*}
I = \lt[ \frac{1}{n-2}(Vu^{-3}-1)ru_r + \frac{1}{n-2} (u^2-1)\Delta \cu + \frac{n-1}{2}(u - u^3 + 2\cu) \rt]_{\ell\ge 1}.
\end{align*}
In the following, we write $C^{k,\alpha} = C^{k,\alpha}(S^{n-1})$ for the H\"{o}lder space on $S^{n-1}$. Therefore, $\| \check u\|_{C^{k,\alpha}}$ is a function of $r$. In the previous section we have shown that $u$ and $V$ are rotationally symmetric up to the order $r^{1-n}$. Moreover, the next order appears at $r^{-3}$ for $n=3$ and $r^{-n-\frac{2}{n-1}}$ for $n \ge 4$. With the estimates of remainders for $u$ (Theorem \ref{u_3} and Lemma \ref{u_higher}) and $V$ (Lemma \ref{V_3} and Lemma \ref{V_high}) we see there exist constants $r_1 \gg 1$ and $M$ such that\footnote{In fact we have $\| \check u\|_{C^{k,\alpha}} \le M_k r^{-n}$ for all $k \ge 3$ (similarly for $\| \check u_r\|_{C^{k,\alpha}}$) but we only need 3-derivatives below.}
\begin{align}
\| \cu\|_{C^{3,\alpha}} \le M r^{-n}\label{initial bound cu},\\
\| \cu_r \|_{C^{1,\alpha}} \le M r^{-n-1}, \label{initial bound cur} \\
\| \na\cV \|_{C^{0,\alpha}} \le M r^{-n}, \label{initial bound cV}\\
\| \na\cV_r \|_{C^{0,\alpha}} \le M r^{-n-1}\label{initial bound cVr}
\end{align}
for $r \ge r_1$. 

\begin{lem}\label{error estimate} The right-hand side of \eqref{ODE for cu} is an error term that decays faster than $\cu$ by $r^{2-n}$.
\begin{align*}
\lt\| F + \frac{G}{n-2} + I \rt\|_{C^{0,\alpha}} \le c M r^{-2n +2}.
\end{align*}
\end{lem}
\begin{proof}
In the proof $c$ will denote a constant that depends only on $m, n$ and changes from line to line.

Recalling that the definition of $\pl_a F$ \eqref{F} does not involve radial derivatives $\check u_r$ and $\check V_r$, we have by \eqref{initial bound cu} and \eqref{initial bound cV}
\begin{align*}
\| \na F\|_{C^{0,\alpha}} \le c M r^{-2n+2}.
\end{align*}
Since $F \in \mathcal{H}_{\ell\ge 1}$, $F$ vanishes somewhere on $S^{n-1}$ and its sup norm is thus bounded by that of its gradient times the diameter of $S^{n-1}$:
\begin{align*}
\| F \|_{C^{0,\alpha}} \le c M r^{-2n+2}.
\end{align*}
Similarly, using the bounds up to 3 derivatives \eqref{initial bound cu}, we get \begin{align*}
\| \na G\|_{C^{0,\alpha}} \le cM r^{-2n+2} , \quad \| G \|_{C^{0,\alpha}} \le c M r^{-2n+2}.
\end{align*} 
The last two terms of $I$ are
\begin{align}\label{last 2 term of I}
\begin{split}
(u^2-1)\Delta \cu &= \lt( \frac{2m}{r^{n-2}} + O(r^{1-n}) \rt) \Delta \cu\\
(u - u^3 + 2\cu)_{\ell\ge 1} &= \lt[ -3(u_0^2-1) \cu - 3 u_0 \cu^2 - \cu^3 \rt]_{\ell\ge 1}.
\end{split} 
\end{align}
For the first term of $I$, the binomial theorem implies
\begin{align}
u^{-3} = u_0^{-3} - 3 u_0^{-4} \cu + \mathfrak{Q}, \quad \| \mathfrak{Q}\|_{C^{0,\alpha}} \le c \| \cu\|_{C^{0,\alpha}}^2 
\end{align}
Hence,
\begin{align*}
(Vu^{-3}-1)ru_r = (V_0 u_0^{-3}-1) ru_r + \lt( u_0^{-3} \cV - 3 V_0 u_0^{-4} \cu + \mathfrak{Q} \rt)ru_r.
\end{align*}
As $(V_0u_0^{-3}-1) = -\frac{4m}{r^{n-2}} + O(r^{1-n})$ and $ru_r =  \frac{(2-n)m}{r^{n-2}} + O(r^{1-n})$, we have
\begin{align}\label{first term of I}
\begin{split}
\lt[ (V_0 u_0^{-3}-1) ru_r \rt]_{\ell\ge 1} &= \lt( -\frac{4m}{r^{n-2}} + O(r^{1-n})\rt) r \cu_r \\
\| \lt( u_0^{-3} \cV - 3 V_0 u_0^{-4} \cu + \mathfrak{Q} \rt)ru_r \|_{C^{0,\alpha}} &\le c r^{2-n} ( \| \cV\|_{C^{0,\alpha}} + \| \cu \|_{C^{0,\alpha}}).
\end{split}
\end{align}
Note that $\| \cV \|_{C^{0,\alpha}} \le cM r^{-n}$  by \eqref{initial bound cV} because $\cV$ takes both positive and negative values. Therefore, by \eqref{initial bound cu} and \eqref{initial bound cur}, we get
\begin{align*}
\| I\|_{C^{1,\alpha}} \le c r^{-2n+2} M.
\end{align*}
\end{proof}

Now we integrate \eqref{ODE for cu} \[ ( r^{n-1} \cu )_r = r^{n-2}(F+\frac{G}{n-2}+I)\] from $r$ to $\infty$. Note that {\underline{there is no constant  of integration from $\infty$}} by \eqref{initial bound cu} and we obtain the improved bounds
\begin{align}\label{cu C0alpha}
\| \cu \|_{C^{0,\alpha}} \le c M r^{-2n+2}, \quad\| \cu_r\|_{C^{0,\alpha}} \le c M r^{-2n+1} 
\end{align}
where the latter follows from $r\cu_r = -(n-1) \cu + F+ \frac{G}{n-2}+I$.

For higher derivative estimates, we differentiate \eqref{ODE for cu} by $\pl_a$ and use the bounds on 3 derivatives to get
\begin{align*}
\| \cu\|_{C^{1,\alpha}} \le c Mr^{-2n+2}, \quad \| \cu_r\|_{C^{1,\alpha}} \le c Mr^{-2n+1}.
\end{align*}
Next, differentiate the zero scalar curvature equation by a Killing vector $X$ on the sphere to get
\begin{align*}
u^2 \Delta (X\cu) = - 2uX(\cu) \Delta\cu + 2r (X\cu)_r - X\cu + 3u^2 X\cu.
\end{align*}
The $C^{0,\alpha}$ norm of the right-hand side is bounded by $cr^{-2n+2} M$ and hence Schauder estimates imply $\| Xu\|_{C^{2,\alpha}} \le c r^{-2n+2} M$.

For $\na V$ and $\na V_r$, we view the $ra$-component of the static equation
\begin{align*}
r ( \pl_a \cV)_r - \pl_a \cV = \lt(  r u^{-1} V_r + (n-2) V u^{-1} \rt) \pl_a \cu
\end{align*}
as an ODE of $\na V$. After integration, the improved decay of $\cu$ implies
\[ \| \na \cV \|_{C^{0,\alpha}} \le c M r^{-2n+2}, \quad \| \na \cV_r \|_{C^{0,\alpha}} \le c M r^{-2n+1}. \]

In conclusion, the initial bounds are improved to
\begin{align*}
\| r^{n} \cu\|_{C^{3,\alpha}}, \| r^{n+1} \cu_r\|_{C^{1,\alpha}}, \| r^{n} \na\cV\|_{C^{0,\alpha}}, \| r^{n+1} \na\cV_r\|_{C^{0,\alpha}} \le \frac{c}{(r_1)^{n-2}} M
\end{align*}
on $r \ge r_1$. Finally we observe that the error terms $F,G,I$ are of the form 
\begin{align}\label{reason can iterate}
F,G,I = O(r^{2-n})\cdot \lt( \cu, \na\cu, \na^2\cu, \na^3\cu, \cu_r \rt) +  O(r^{2-n}) \cdot \lt( \cV + \na\cV \rt)
\end{align}
from Lemma \ref{error estimate}. Therefore we can iterate the estimates with improved bounds on $\cu$ and $\cV$ to get
\begin{align*}
\| r^{n} \cu\|_{C^{3,\alpha}}, \| r^{n+1} \cu_r\|_{C^{1,\alpha}}, \| r^{n} \na\cV\|_{C^{0,\alpha}}, \| r^{n+1} \na\cV_r\|_{C^{0,\alpha}} \le \lt( \frac{c}{(r_1)^{n-2}}\rt)^k M
\end{align*}
for any $k \ge 0$ inductively. This implies $\cu=0$ and $\na \cV=0$ on $r \ge r_1$ if $c < (r_1)^{n-2}$. Since $\cV$ has no zero mode, we also have $\cV=0$ on $r \ge r_1$. This completes the proof of Lemma \ref{vanish out a compact set}.
\end{proof} 

\subsection{Proof of rotational symmetry and Theorem \ref{main}}
We now have $u = u(r)$ and $V = V(r)$ on $[r_1,\infty) \times S^{n-1}$ for some $r_1 > r_0$. We observe that the rotational symmetry of $V$ on the domain of definition $(r_0,\infty) \times S^{n-1}$ and the exact form asserted in Theorem \ref{main} follows from the rotational symmetry of $u$ (the metric part) on $(r_0,\infty) \times S^{n-1}$: 
\begin{prop}\label{rotationally symmetric u,V}
Let $g = u^2(\theta,r)dr^2 + r^2 g_{S^{n-1}}$ be a solution to the static equations \eqref{static} on some $(r_0,\infty) \times S^{n-1}$. Suppose $u$ is rotationally symmetric and $u  \rw 1, V = 1 + o_1(1)$. Then
\begin{align*}
u = \frac{1}{\sqrt{1 - 2m_0 r^{2-n}}}, \quad V = \sqrt{1 - 2m_0 r^{2-n}}
\end{align*}
for some constant $m_0 \in \R$.
\end{prop}
\begin{proof}
As $u$ is rotationally symmetric, the zero scalar curvature equation, which decouples from the static equations, becomes an ODE and integration yields
\[ u = \frac{1}{\sqrt{1 - 2m_0 r^{2-n}}} \]
for some constant $m_0 \in \R$. The $rr$-component of the static equations becomes
\begin{align*}
V_{rr} - u^{-1} u_r V_r = \frac{(n-1)(n-2)}{2r^2} (1-u^2)V
\end{align*}
which is viewed as an ODE in each direction $\theta$. Let $V_0 = \sqrt{1 - 2m_0 r^{2-n}}$ be a solution to the above ODE. Then the Wronskian
\begin{align*}
W =\begin{vmatrix}
V & V_0 \\
V_r & (V_0)_r
\end{vmatrix}
\end{align*}
satisfies $W_r - u^{-1}u_r W =0$ or equivalently $(u^{-1}W)_r =0$. The assumption $V = 1 + o_1(1)$ implies $\lim_{r\rw\infty}W (r,\theta)=0$ and hence $W$ vanishes identically. It follows that $V = V_0$.
\end{proof}

Now it remains to push the rotational symmetry of $u$ from $[r_1,\infty) \times S^{n-1}$ to $(r_0,\infty) \times S^{n-1}$. This is a problem of unique continuation and we supply three proofs.

\begin{proof}[Proof 1.]
Consider the set 
\begin{align*}
S = \{ r_2 \in (r_0,\infty) : \cu = \cV =0 \mbox{ on } [r_2,\infty) \times S^{n-1} \}
\end{align*} 
We have shown that $S$ is not empty. By continuity, $S$ is a closed set. We will show that $S$ is open to conclude the rotational symmetry. Fix $r_2 \in S$ and a small constant $\delta$ such that 
\begin{align}\label{another initial bound}
\| \cu\|_{C^{3,\alpha}} + \| \cu_r \|_{C^{1,\alpha}} + \| \na\cV \|_{C^{0,\alpha}} + \| \na \cV_r\|_{C^{0,\alpha}} \le 1
\end{align}
on $[ r_2 - \delta, r_2] \times S^{n-1}$. The argument in Lemma \ref{error estimate} now implies that, in analogous to \eqref{reason can iterate},
\begin{align*}
F,G,I = O(1) \cdot \lt( \cu, \na\cu, \na^2\cu, \na^3\cu, \cu_r \rt) +  O(1) \cdot \lt( \cV + \na\cV \rt)
\end{align*}
and hence
\begin{align*}
\lt| r \cu_r + (n-1)\cu \rt| \le c
\end{align*}
by \eqref{another initial bound}.

Now integration from $r_2$ yields
\begin{align*}
r^{n-1} |\cu|   \le c \lt( r_2^{n-1} - r^{n-1} \rt)
\end{align*}
on $[r_2 - \delta, r_2] \times S^{n-1}$
and then
\begin{align*}
|\cu| \le c \lt[ \lt( \frac{r_2}{r_2 - \delta}\rt)^{n-1} -1 \rt] \le c\delta
\end{align*}
on $[r_2 - \delta, r_2] \times S^{n-1}$.

Higher derivative estimates can be done as above and we arrive at the improved bounds
\begin{align*}
\| \cu\|_{C^{3,\alpha}} + \| \cu_r \|_{C^{1,\alpha}} + \| \na\cV \|_{C^{0,\alpha}} + \| \na \cV_r\|_{C^{0,\alpha}} \le c\delta
\end{align*}
and iteration yields
\begin{align*}
\| \cu\|_{C^{3,\alpha}} + \| \cu_r \|_{C^{1,\alpha}} + \| \na\cV \|_{C^{0,\alpha}} + \| \na \cV_r\|_{C^{0,\alpha}} \le ( c\delta )^k
\end{align*}
for any $k \ge 0$. Therefore $\cu = \cV =0$ on $[r_2 - \delta, r_2] \times S^{n-1}$ for $\delta < \frac{1}{c}$. This completes the proof of openness of $S$.
\end{proof}  

\begin{proof}[Proof 2.]
Consider the change of variable $r = e^{(n-1)t}$. Then the zero scalar curvature equation becomes
\begin{align}
u^2 \Delta u - \pl_t u + \frac{(n-1)(n-2)}{2}(u - u^3)=0.
\end{align}
Let 
\begin{align*}
u_0 = \frac{1}{\sqrt{1 - 2m e^{-(n-1)(n-2)t}}}
\end{align*}
be the rotationally symmetric solution to the above equation that coincides with $u$ on $[r_1,\infty) \times S^{n-1}$. We see that the difference $ v = u - u_0$
\begin{align*}
\pl_t v &= \pl_t u - \pl_t u_0 \\
&= u^2 \Delta u + \frac{(n-1)(n-2)}{2} (u - u^3 - u_0 + u_0^3)\\
&= u^2 \Delta v + \frac{(n-1)(n-2)}{2} \lt(1 -(u^2 + u_0 u + u_0^2) \rt) v 
\end{align*} is a solution to a {\it linear parabolic equation}. Since $v =0$ on $\{ r_1\} \times S^{n-1}$, $v=0$ on $(r_0,\infty) \times S^{n-1}$ by the following weak unique continuation theorem of Lees-Protter \cite[Theorem 1, page 373]{LP}.
\end{proof}
\begin{theorem}
Let $M$ be a closed Riemannian manifold with Levi-Civita connection $D$. Suppose $v$ that satisfies the differential inequality 
\[ (Lv)^2 \le c_1 v^2 + c_2 |D v|^2 \]
on $M \times [-\tau, 0]$ for some positive constants $c_1$ and $c_2$ where $L$ is a parabolic differential operator
\[ L = a^{ij}(x,t) D_i D_j - \frac{\pl}{\pl t}.\] If $v(\cdot, 0)=0$ then $v$ vanishes identically on $M \times [-\tau, 0]$.
\end{theorem}

\begin{proof}[Proof 3.] 
By M\"uller zum Hagen \cite{static_analytic}, static metrics are analytic. As a result, the metric $g$ is isometric to the Schwarzschild metric on $\R^n - B_{r_0}$.  
\begin{prop}\label{conformally flat}
The metric $g = u^2(r,\theta) dr^2 + r^2 g_{S^{n-1}}$ is conformally flat if and only if there exist two functions $a_0(r)$ and $a_i(r)$ such that $u = a_0(r) + \sum_{i=1}^n a_i(r) X^i$. \end{prop}
\begin{proof}
    Suppose $u = a_0(r) + a_i(r) X^i.$ Then 
\begin{align*}
g = (a_0 + a_i X^i)^2 \lt( dr^2 + r^2 \frac{g_{S^{n-1}}}{(a_0 + a_i X^i)^2} \rt).
\end{align*}
It is well-known that the metric $\frac{g_{S^{n-1}}}{(a_0 + a_i X^i)^2}$ has constant sectional curvature 1 and is isometric to $g_{S^{n-1}}$. This proves that $g$ is conformally flat.

Suppose $g$ is conformally flat. Then the Cotton tensor of $g$
\begin{align*}
C_{ijk} = D_k R_{ij} - D_j R_{ik} + \frac{1}{2(n-1)} \lt( D_j R g_{ik} - D_k R g_{ij} \rt)
\end{align*}
vanishes, see Corollary 7.31 and Proposition 7.32 of \cite{Lee} for example. 

Fix a level set $S$ of $r$. We will show that $u|_S \in \mathcal{H}_{\ell\le 1}$. Let $\sigma$ and $\na$ denote the induced metric and its Levi-Civita connection of $S$. Moreover let $\nu = u^{-1} \frac{\pl}{\pl r}$ and $H$ denote the the unit normal vector and the mean curvature of $S$. Note that $S$ is umbilical with second fundamental form $\frac{H}{n-1} \sigma$ and hence $D_X \nu = \frac{H}{n-1} X$ for any vector tangent to $S$. Let $R^\sigma = (n-1)(n-2) r^{-2}$ denote the constant scalar curvature of $\sigma$ on $S$. The Gauss and Codazzi equations of $S$ read
\begin{align*}
R^\sigma &= R - 2 R_{\nu\nu} + \frac{n-2}{n-1}H^2\\
R_{b\nu} &= - \frac{n-2}{n-1} \pl_b H.
\end{align*}

Let $\mathcal{R}_{ab}$ be the symmetric 2-tensor on $S$ obtained by restricting the Ricci tensor of $g$ to $S$. We compute
\begin{align*}
D_b R_{\nu\nu} &= \pl_b R_{\nu\nu} - 2 \frac{H}{n-1} R_{b\nu} = \pl_b \lt( \frac{R}{2} + \frac{n-2}{2(n-1)} H^2 \rt) + 2 \frac{n-2}{(n-1)^2} H \pl_b H \\
D_\nu R_{\nu b} &= \frac{1}{2} \pl_b R - \sigma^{ac} D_a R_{cb} = \frac{1}{2} \pl_b R - \sigma^{ac}\lt( \na_a \mathcal{R}_{cb} + \frac{H}{n-1} R_{\nu b} \sigma_{ac} + \frac{H}{n-1} R_{\nu c} \sigma_{bc} \rt) \\
&= \frac{1}{2} \pl_b R - \na^a \mathcal{R}_{ab} + \frac{n (n-2)}{(n-1)^2} H \pl_b H. 
\end{align*}
Let $\tr_\sigma \mathcal{R} = \sigma^{ab} \mathcal{R}_{ab}$ be the trace of $\mathcal{R}_{ab}$ on $S$. We have $\tr_\sigma \mathcal{R} = R - R_{\nu\nu} = \frac{R}{2} + \frac{R^\sigma}{2} - \frac{n-2}{2(n-1)}H^2$ and
\begin{align*}
C_{\nu\nu b} = \na^a \mathcal{R}_{ab} + \frac{n-2}{(n-1)^2} H \pl_b H - \frac{1}{2(n-1)} \pl_b R = \na^a \lt( \mathcal{R}_{ab} - \frac{1}{n-1} \tr_\sigma \mathcal{R} \sigma_{ab} \rt).
\end{align*}
From the formula of Ricci curvature quoted in the proof of Proposition \ref{derivation of static equations}, the traceless symmetric 2-tensor is related to $u$ by
\begin{align*}
\mathcal{R}_{ab} - \frac{1}{n-1} \tr_\sigma \mathcal{R} \sigma_{ab} = u^{-1} \lt( \na_a\na_b u - \frac{\Delta_\sigma u}{n-1} \sigma_{ab} \rt).
\end{align*}
By the vanishing of Cotton tensor, we get
\begin{align*}
0 = \int_S \na^b u \na^a \lt( u^{-1} (\na_a\na_b u - \frac{\Delta_\sigma u}{n-1}\sigma_{ab}) \rt) = - \int_S u^{-1} \lt| \na_a\na_b u - \frac{\Delta_\sigma u}{n-1} \sigma_{ab} \rt|^2
\end{align*}
and hence $u|_S \in \mathcal{H}_{\ell\le 1}$ by Proposition \ref{characterizing Hle1}.   
\end{proof}
Since Schwarzschild metric is conformally flat, we apply the previous proposition to get $u = a_0(r) + \sum_{i=1}^n a_i(r) X^i$ and the zero scalar curvature equation becomes
\begin{align*}
(a_0 + a_i X^i)^2 (-n a_j X^j) - (n-1) r (a_0' + a_i' X^i) - \frac{(n-1)(n-2)}{2} \lt[ a_0 + a_i X^i - (a_0 + a_i X^i)^3 \rt]=0.
\end{align*}
Projecting to $\ell=2$ modes yields
\begin{align*}
a_0 \lt( -2n - \frac{3}{2}(n-1)(n-2) \rt) (a_i X^i a_j X^j)_{\ell=2} =0.
\end{align*}
Hence $a_i =0$ for $i=1, \cdots, n$ by Proposition \ref{udot_square_l=2} and $u$ is a function of $r$.
\end{proof}
\subsection{Proof of Theorem \ref{equality case}}
Suppose the equality holds for $\Sigma$. Let $\Sigma_t$ be the weak solution of IMCF starting from $\Sigma$. Corollary \ref{cor} implies
\begin{align*}
g = u^2 (r,\theta) dr^2 + r^2 g_{S^{n-1}}
\end{align*}
outside $\Sigma_{T_{0}}$ in the case of equality in \eqref{Minkowski_static}, where $T_0$ is the time IMCF becomes smooth in Theorem \ref{imcf_smoothing}. By Theorem \ref{main}, $M^{n} \setminus \overline{\Omega_{T_{0}}}$ is nothing but a Schwarzschild space \eqref{schwarzschild} with each $\Sigma_{t}, t \in [T_0,\infty),$ being a slice $\{ r=r(t)\}$. Since static metrics are analytic \cite{static_analytic}, we conclude the whole $(M^{n},g)$ is Schwarzschild space. 

For the rest of the proof, assume the nontrivial case $m \neq 0$. We know
\begin{enumerate}
\item the equality holds for each $\Sigma_t$
\item each $\Sigma_t$ is weakly umbilic
\item each $\Sigma_t$ is $C^{1,1}$ away from a closed singular set of dimension at most $n-8$ by Heidusch's regularity theorem, see the footnote on page 369 of \cite{HI}.
\end{enumerate}

Since Schwarzschild is conformally flat and umbilicity is invariant under conformal change, $\Sigma_t$ is a $C^{1,1}$ hypersurface in $\R^n$ with a small singular set and umbilic on its regular part, which is path-connected by Lemma \ref{path-connected}. By Lemma 3.2 of \cite{DRKS20}, the regular part of $\Sigma_t$ lies on a sphere\footnote{Without knowing the regular part is path-connected, the best we can say is that $\Sigma_t$ consists of pieces of spheres and planes joined along the singular set, although this is very unlikely given the Hausdorff dimension of the singular set is small.} and hence the whole $\Sigma_t$ is a sphere.

We note that all $\Sigma_t$ must enclose the horizon if $m > 0$ or the curvature singularity if $m<0$. Otherwise when some $\Sigma_t$ jumps into a sphere that encloses both the horizon and $\Sigma_t$, the area jumps discontinuously and violates $|\Sigma_t| = e^t |\Sigma|$.

If $m < 0$, all spheres enclosing the curvature singularity are strictly mean-convex. Hence $\Sigma_t, t \ge 0$ is a smooth solution to IMCF and all of $\Sigma_t$ are slices $\{ r=r(t) \}$ by Corollary \ref{cor} and Theorem \ref{main}.

It remains to handle $m > 0$ case. By Lemma \ref{property of spheres}, all $\Sigma_t$ are strictly outer-minimizing. From the construction of weak solution to IMCF, there is no jump. Let \[ \mathcal{T} = \{ t_1 \in [0,\infty): \Sigma_t \mbox{ is a slice $\{ r=r(t)\}$ in Schwarzschild for } [t_1,\infty) \}. \]
We know $\mathcal{T} \neq \emptyset$ and $\mathcal{T}$ is closed because there is no jump. Let $t_1 \in \mathcal{T}$. Since there is no jump, there exists a small $\delta > 0$ such that $\Sigma_t, t_1 -\delta < t < t_1$ is strictly mean-convex. Hence the weak solution on $[t_1 - \delta, t_1]$ is actually smooth. Corollary \ref{cor} and Theorem \ref{main} then imply each $\Sigma_t, t\in [t_1-\delta,\infty),$ is a slice. This concludes $\mathcal{T}$ is open and in particular $\Sigma = \Sigma_0$ is a slice $\{ r = r_0 \}$.

\section{Applications}\label{section:applications}
Our approach to the uniqueness theorems for black holes, equipotential photon surfaces, and static metric extensions involves an additional Minkowski inequality which arises from the conformal symmetries of the static vacuum equations. We originally proved this inequality in Section 4 of \cite{HW}, and because it comes from the original Minkowski inequality \eqref{Minkowski_static} it also extends to the larger class of static manifolds considered here, provided the boundary is connected.
\begin{prop} \label{conformal_theorem}
Let $(M^{n},g,V)$ be asymptotically flat  with connected boundary $\partial M = \Sigma$ satisfying $V|_{\Sigma}>0$. Suppose that $\Sigma$ is outer-minimizing in $(M^{n},g_{-})$, where $g_{-}$ the conformally-modified metric

\begin{equation} \label{g_conf}
    g_{-}= V^{\frac{4}{n-2}} g.
\end{equation}
Then

\begin{equation} \label{conformal_ineq}
    \frac{1}{(n-1)w_{n-1}} \int_{\Sigma} VH d\sigma \geq \left( \frac{1}{w_{n-1}} \int_{\Sigma} V^{2 \frac{n-1}{n-2}} d\sigma \right)^{\frac{n-2}{n-1}}.
\end{equation}
\end{prop}

\begin{proof}
The metric $g_{-}$ and associated potential function

\begin{equation*}
    V_{-} = V^{-1} \in C^{\infty}_{+}(M^{n})
\end{equation*}
solve the static equations \eqref{static} on $M^{n}$; we refer the reader to the proof of Proposition 4.1 in \cite{HW} for the calculation. Given $V=1 + o_{2}(1)$ and $g_{ij}=\delta_{ij}+o_{2}(1)$, it immediately follows that

\begin{eqnarray*}
    (g_{-})_{ij}(x) &=& \delta_{ij} + o_{2}(1), \\
    V_{-}(x) &=& 1 + o_{2}(1).
\end{eqnarray*}
Therefore, if $\Sigma$ is outer-minimizing in $(M^{n},g_{-})$, we apply \eqref{Minkowski_static} to $\Sigma$. Also in the proof of \cite[Proposition 4.1]{HW}, we compute the left-hand and right-hand sides of \eqref{Minkowski_static} to respectively be
\begin{eqnarray*}
    \frac{1}{(n-1)w_{n-1}} \int_{\Sigma} V_{-} H_{-} d\sigma_{-} + 2m_{-} &=& \frac{1}{(n-1)w_{n-1}} \int_{\Sigma} VHd\sigma, \\
    \left( \frac{|\Sigma|_{g_{-}}}{w_{n-1}} \right)^{\frac{n-2}{n-1}} &=& \left( \frac{1}{w_{n-1}} \int_{\Sigma} V^{2 \frac{n-1}{n-2}} d\sigma \right)^{\frac{n-2}{n-1}}.
\end{eqnarray*}
\end{proof}
Proposition \ref{conformal_theorem} implies a Minkowski inequality for the level sets of the potential function, provided they are outer-minimizing in $(M^{n},g_{-})$.
\begin{cor}[Minkowski Inequality for Equipotential Boundaries] \label{equi_minkowski}
Let $(M^{n},g,V)$ be asymptotically flat  with connected boundary $\partial M= \Sigma$.  Suppose that $V|_{\Sigma} = V_{0} > 0$ is constant and that $\Sigma$ is outer-minimizing with respect to the conformal metric $g_{-}$ from \eqref{g_conf}. Then we have

\begin{equation} \label{equi}
    V_{0} \leq \frac{1}{(n-1)\omega_{n-1}} \left( \frac{|\Sigma|}{w_{n-1}} \right)^{\frac{2-n}{n-1}} \int_{\Sigma} H d\sigma.
\end{equation}
\end{cor}

\subsection{Photon surface and static metric extension uniqueness}
Given inequality \eqref{equi}, both Theorems \ref{photon_surface} and \ref{static_extension} can be reduced to the following rigidity result for Schwarzschild.
\begin{theorem} \label{general_rig}
Let $(M^{n},g,V)$ be asymptotically flat with connected boundary $\partial M = \Sigma$. Suppose that the following hold on $\Sigma$:

\begin{eqnarray}
    V|_{\Sigma} &=& V_{0} > 0, \label{V_const} \\
    H_{\Sigma} &=& H_{0} > 0, \label{H_const} \\
    h_{\Sigma} &=& \frac{H_{0}}{n-1} \sigma, \label{pure_trace}
\end{eqnarray}
where $H_{0}$, $V_{0}$ are constants, and $\sigma$, $h_{\Sigma}$ are the induced metric and second fundamental form on $\Sigma^{n-1}$. Assume also that

\begin{itemize}
\item $\Sigma$ is outer-minimizing in $(M^{n},g)$ if $m \geq 0$,
\item $\Sigma$ is outer-minimzing in $(M^{n},g_{-})$ if $m<0$.
\end{itemize}
Then 
\begin{equation*}
\left( \frac{|\Sigma|}{w_{n-1}} \right)^{\frac{n-2}{n-1}} \leq \frac{1}{(n-1)w_{n-1}} \int_{\Sigma} V_{0}H_{0} d\sigma + 2m \leq \frac{E(\Sigma,\sigma)}{E(\mathbb{S}^{n-1},g_{S^{n-1}})} \left( \frac{|\Sigma|}{w_{n-1}} \right)^{\frac{n-2}{n-1}}.
\end{equation*}
where $E(\Sigma,\sigma)$ is the normalized Einstein-Hilbert energy \eqref{einstein_hilbert} of $(\Sigma^{n-1},\sigma)$. In particular, $E(\Sigma,\sigma) \geq E(\mathbb{S}^{n-1},g_{S^{n-1}})$ with equality if and only if $(M^{n},g,V) \cong (\mathbb{S}^{n-1} \times (r_{0},\infty), g_{m}, V_{m})$.
\end{theorem}
\begin{proof}
\textit{Case I, $m >0$}: The lower bound immediately follows from the outer-minimizing property. To obtain the upper bound, we claim that $\Sigma \subset (M^{n},g_{-})$ is outer-minimizing. Observe that $m>0$ implies $V_{0} < 1$. Otherwise, we could choose a point $x \in \Sigma$ with $\frac{\partial V}{\partial \nu}(x) > 0$, which gives a $y \in M^{n}$ with

\begin{equation*}
    V(y) > \max\{ V_{0}, 1 \},
\end{equation*}
contradicting the elliptic maximum principle. Also by the elliptic maximum principle, we have that $\inf_{M^{n}} V = \min \{ V_{0}, 1 \}$, and so $V \geq V_{0}$ in $M^{n}$.

Let $\widetilde{\Sigma}^{n-1} \subset M^{n}$ be a hypersurface enclosing $\Sigma^{n-1}$. Given $|\widetilde{\Sigma}|_{g} \geq |\Sigma|_{g}$, the above observation implies that $|\widetilde{\Sigma}|_{\widehat{g}} > |\Sigma|_{\widehat{g}}$, and so $\Sigma$ is outer-minimizing under the conformal change. Thus inequality \eqref{equi} holds on $\Sigma$.

\eqref{V_const}-\eqref{pure_trace} imply $\frac{\partial V}{\partial \nu} = \kappa$ for a nonzero constant $\kappa$ via the Gauss and Codazzi equations, see for example Section 4 of \cite{CG}. Putting together the Gauss and static equations, we get

\begin{equation*}
    H_{0} \kappa = -Hess (V)_{\nu \nu} = -V_{0} \text{Ric}_{\nu \nu} = \frac{V_{0}}{2} \left( R_{\sigma} - \frac{n-2}{n-1} H_{0}^{2} \right).
\end{equation*}
Using the definition of $m$, we derive the upper bound

\begin{eqnarray*}
    m &=& \frac{1}{(n-2)w_{n-1}} \int_{\Sigma} \kappa d\sigma = -\frac{1}{2(n-1)w_{n-1}} \int_{\Sigma} V_{0} H_{0} d\sigma + \frac{1}{2(n-2)w_{n-1}H_{0}} \int_{\Sigma} V_{0} R_{\sigma} d\sigma \\
    &\leq&  -\frac{1}{2(n-1)w_{n-1}} \int_{\Sigma} V_{0} H_{0} d\sigma + \frac{1}{2(n-1)(n-2)w_{n-1}} \left( \frac{|\Sigma|}{w_{n-1}} \right)^{\frac{1}{n-1}} \int_{\Sigma} R_{\sigma} d\sigma \\
    &=& -\frac{1}{2(n-1)w_{n-1}} \int_{\Sigma} V_{0} H_{0} d\sigma + \frac{1}{2} \frac{E(\Sigma,\sigma)}{E(\mathbb{S}^{n-1},g_{S^{n-1}})} \left( \frac{|\Sigma|}{w_{n-1}} \right)^{\frac{n-2}{n-1}}.
\end{eqnarray*}
Here, the inequality follows from the equipotential Minkowski inequality \eqref{equi}, which gives a lower bound on $H_{0}$ here.

\textit{Case 2, $m=0$}: We claim that $V \equiv 1$ on $(M^{n},g)$. If not, then for almost every $s \in \text{Im}(V) \subset \mathbb{R}^{+}$ we have $\int_{\{V=s\}} \frac{\partial V}{\partial \nu} d\sigma \neq 0$ by Sard's theorem. On the other hand, for any $s$ sufficiently close to $1$, we have that $\int_{\{V=s\}} \frac{\partial V}{\partial \nu} d\sigma= (n-2)w_{n-1} m =0$ by the divergence theorem. This is a contradiction, and so $V \equiv 1$ on $M^{n}$ and hence $(M^{n},g)$ is Ricci flat. The Minkowski inequality then reduces to

\begin{equation*}
\frac{1}{(n-1)w_{n-1}}\int_{\Sigma} Hd\sigma \geq \left( \frac{|\Sigma|}{w_{n-1}} \right)^{\frac{n-2}{n-1}}.
\end{equation*}
For a CMC boundary, this implies that $H_{0} \geq (n-1) \left( \frac{|\Sigma|}{w_{n-1}} \right)^{\frac{1}{n-1}}$. Furthermore, if $\Sigma$ is total umbilical the Gauss-Codazzi equations then imply

\begin{equation*}
H_{0}= \frac{n-1}{(n-2)H_{0}} R_{\sigma} \leq \frac{1}{n-2} \left( \frac{|\Sigma|}{w_{n-1}} \right)^{\frac{-1}{n-1}} R_{\sigma}.
\end{equation*}
From this, we recover the upper bound on $\int_{\Sigma} H d\sigma$.

\textit{Case 3, $m <0$}: Observe that the mass $(M^{n},g_{-})$ is $m_{-}=-m > 0$. Therefore by the same argument as in Case I, $\Sigma$ is also outer-minimizing in $(M^{n},g)$ given that it is outer-minimizing in $(M^{n},g_{-})$. Therefore, we may apply both inequalities \eqref{Minkowski_static} and \eqref{equi} to $\Sigma$ like in Case I.
\end{proof}

\begin{proof}[Proof of Theorems \ref{photon_surface} and \ref{static_extension}]
The time slices of an equipotential photon surface automatically obey \eqref{V_const}-\eqref{pure_trace}-- once again see Section 4 of \cite{CG} for the calculation-- and so we recover Theorem \ref{photon_surface} from Theorem \ref{general_rig}. For the case of Schwarzschild Bartnik boundary data, we have \eqref{H_const} and $E(\Sigma,\sigma)=E(\mathbb{S}^{n-1},g_{S^{n-1}})$ automatically. \eqref{pure_trace} follows as a consequence of Schwarzschild stability, and if $H_{0} < (n-1) r_{0}$ then one may also obtain \eqref{V_const}-- see Section 5 of \cite{HW}. If $H_{0}=(n-1)r_{0}$, then we also show in Section 5 of \cite{HW} that $m=0$, and hence $V \equiv 1$.
\end{proof}

\subsection{Black hole uniqueness}
To begin, we recall some facts about static horizons. The first two are well-known, while the third follows from \cite{BM23}, Proposition 5.1.

\begin{prop} \label{horizon_boundary}
Let $(M^{n},g,V)$ be asymptotically flat with connected boundary $\partial M= \Sigma$ satisfying $V|_{\Sigma} = 0$. Then

\begin{enumerate} [label=\roman*.]
     \item $\frac{\partial V}{\partial \nu} = \kappa$ on $\Sigma$ for a constant $\kappa > 0$.
    \item $\Sigma^{n-1} \subset (M^{n},g)$ is totally geodesic.
    \item $\Sigma^{n-1}$ is outer-minimizing in $(M^{n},g)$.
\end{enumerate}
\end{prop}
\begin{proof}
By the first static equation, we have that $\text{Hess} \left(V \right) (x) = 0$ for $x \in \Sigma$. Denoting the unit normal on $\Sigma$ by $\nu$, we have in local coordinates $(y_{1},\dots,y_{n-1})$ of $\Sigma^{n-1}= \{ V =  0 \}$ that

\begin{equation*}
    0 = \text{Hess} \left( V \right)_{a\nu} = \partial_{a} \left( \frac{\partial V}{\partial \nu} \right) - h_{a}^{b} \partial_{b} V = \partial_{a} \left( \frac{\partial V}{\partial \nu} \right).
\end{equation*}
Thus $\frac{\partial V}{\partial \nu}$ is constant on $\Sigma$. The positivity of this constant follows from the Hopf Lemma. In directions tangent to $\Sigma$, the vanishing of the Hessian yields

\begin{equation*}
    0= \text{Hess}_{\Sigma}\left(V \right)_{ab} + \kappa h_{ab} = \kappa h_{ab}.
\end{equation*}
For item (iii), we invoke Proposition 5.1 in \cite{BM23}, which states that a static horizon is outer-minimizing hypersurface $S$ for any $n$ provided that it is homologous to an outer-minimizing hypersurface. We may choose $S$ to be a large coordinate sphere in an asymptotically flat $(M^{n},g)$. 
\end{proof}

The idea for the proof of Theorem \ref{black_hole_uniqueness} is to consider the level set flow of the static potential beginning from the horizon and apply the equipotential Minkowski inequality \eqref{equi} to the near-horizon level sets. The cumbersome technical aspect of this argument is establishing that the near-horizon level sets are outer-minimizing in $(M^{n},g_{-})$. We can verify this by considering an additional conformal metric

\begin{equation} \label{g_conf2}
\widetilde{g} = \left( \frac{V+1}{2} \right)^{\frac{4}{n-2}} g.
\end{equation}
Bunting and Masood ul-Alam \cite{BMA} used this metric in their $3$-dimensional static black hole uniqueness proof, and it also been used in the study of static manifolds by Raulot \cite{Rau} and Cederbaum-Galloway \cite{photon_sphere_positive_mass}, probably among others. Unlike in those works, we do \text{not} use this metric for a gluing argument-- rather, we will argue that the level set $\{ V =s \} \subset (M^{n},\widetilde{g})$ is outer-minimizing for sufficiently small $s$-- the argument also uses barriers formed by the weak IMCF in $(M^{n},\widetilde{g})$. In turn, this guarantees the outer-minimizing property in $(M^{n},g_{-})$ as well (note that for Corollary \ref{equi_minkowski}, we do not need outer-minimizing W.R.T. the original metric $g$).
\begin{lemma}
Let $(M^{n},g,V)$ be asymptotically flat with and connected boundary $\partial M=\Sigma$ satisfying $V|_{\Sigma} \equiv 0$. Then there is an $s_{0} > 0$ such that the level sets $\Sigma_{s}= \{ V = s\}$ are outer-minimizing hypersurfaces in $(M^{n},g_{-})$ for each $s \in (0,s_{0})$, where $g_{-}$ is the metric defined in \eqref{g_conf}. As a consequence, the equipotential Minkowski inequality \eqref{equi} holds on $\Sigma_{s} \subset (M^{n},g)$ for each $s \in (0,s_{0})$.
\end{lemma}
\begin{proof}
If $V|_{\Sigma} \equiv 0$, then the mean curvature $\widetilde{H}$ of $\Sigma$ in $(M^{n},\widetilde{g})$ for $\widetilde{g}$ in \eqref{g_conf2} is easily computed to be positive-- see for example equation (4.2) in \cite{Rau}. Therefore, we first choose an $s_{1} >0$ sufficiently small so that $\Sigma_{s}=\{ V=s \}$ is regular and $\widetilde{H}_{\Sigma_{s}} >0$ for each $s \in (0,s_{1})$. 

It is also immediately clear by the choice of conformal factor that $\Sigma$ is strictly outer-minimizing in $(M^{n},\widetilde{g})$. Let $u: (M^{n},\widetilde{g}) \rightarrow \mathbb{R}$ be the proper weak IMCF in $(M^{n},\widetilde{g})$ with initial condition $\Sigma$. Since $\Sigma$ is strictly outer-minimizing, we have that $u > 0$ in $M^{n}$. It follows in particular that
\begin{equation}
\{ u \leq \epsilon \} \subset \{ V < s_{1}\} \subset M^{n}
\end{equation}
holds for some $\epsilon>0$. We choose an $s_{0} < s_{1}$ so that

\begin{equation}
\{ V < s_{0} \} \subset \{ u \leq \epsilon \}.
\end{equation}
For $s \in (0,s_{0})$, let $\Sigma'_{s}$ be the strictly minimizing hull of $\Sigma_{s} \subset (M^{n},\widetilde{g})$, and call the domains that these surfaces bound $E'$ and $E$, respectively. Since $\partial \{ u \leq \epsilon \} \subset (M^{n},\widetilde{g})$ is strictly outer-minimizing (see the appendix) and $E \subset \{ u \leq \epsilon \}$, it follows that $E' \subset \{ u \leq \epsilon \}$ and hence $\Sigma'_{s} \subset \{ V < s_{1} \}$. Now we argue that $E'=E$. For any smooth domain $F \subset \{ V < s_{1} \}$ containing $E$, the divergence theorem in $(M^{n},\widetilde{g})$ and the foliation by mean-convex $\Sigma_{s}$ imply
\begin{eqnarray*}
   0 &<& \int_{F \setminus E} \widetilde{H}_{\Sigma_{s}} d\Omega = \int_{F \setminus E} \widetilde{\text{div}} \left( \frac{\widetilde{\nabla} V}{|\widetilde{\nabla} V|} \right) d\Omega = \int_{\partial F}   \widetilde{g} \left( \frac{\widetilde{\nabla} V}{|\widetilde{\nabla} V|}, \widetilde{\nu} \right) d\sigma - \int_{\Sigma_{s}} \widetilde{g} \left( \frac{\widetilde{\nabla} V}{|\widetilde{\nabla} V|}, \frac{\widetilde{\nabla} V}{|\widetilde{\nabla} V|} \right) d\sigma \\
   & \leq & |\partial F|_{\widetilde{g}} - |\Sigma_{s}|_{\widetilde{g}}.
\end{eqnarray*}
This means that $E'=E$ and hence $\Sigma_{s}=\{ V = s\}$ is strictly outer-minimizing in $(M^{n},\widetilde{g})$ for $s \in (0,s_{0})$. Finally, the metrics $\widetilde{g}$ and $g_{-}$ are related by the conformal transformation

\begin{equation*}
g_{-} = \left( \frac{2V}{V+1} \right)^{\frac{4}{n-2}} \widetilde{g}.
\end{equation*}
We once again have that $V \geq s$ outside $\Sigma_{s}$ by the maximum principle, and hence $\frac{V}{V+1} \geq \frac{s}{s+1}$ on this exterior. Therefore,
\begin{equation*}
|\Sigma'|_{g_{-}} \geq \left( \frac{2s}{s+1} \right)^{\frac{4}{n-2}} |\Sigma'|_{\widetilde{g}} \geq \left( \frac{2s}{s+1} \right)^{\frac{4}{n-2}} |\Sigma_{s}|_{\widetilde{g}}= |\Sigma_{s}|_{g_{-}}
\end{equation*}
for any hypersurface $\Sigma'$ enclosing $\Sigma_{s}$.

\end{proof}
\begin{proof}[Proof of Theorem \ref{black_hole_uniqueness}]
We consider the level set flow $F: \Sigma^{n-1} \times (0, s_{0}) \rightarrow M$ of $V$ beginning from $\Sigma_{0} = \{ V=0 \}$ and with $s_{0}$ as in the previous lemma. Let us define the function $\gamma: (0,s_{0}) \rightarrow \mathbb{R}$ by

\begin{equation}
    \gamma(s)= - s + \frac{1}{(n-1)\omega_{n-1}} \left( \frac{|\Sigma_{s}|}{w_{n-1}} \right)^{\frac{2-n}{n-1}} \int_{\Sigma_{s} } H d \sigma,
\end{equation}
where $\Sigma_{s}= \{ V=s \}$. Applying the previous lemma, we know $\gamma(s) \geq 0$ for each $s \in (0,s_{0})$. Meanwhile, $\gamma(0)=0$ by property (ii) in Proposition \ref{horizon_boundary}. Therefore, we compute

\begin{eqnarray} \label{horizon_deriv}
    0 \leq \frac{d}{ds}|_{s =0 } \gamma(s) 
   &=& -1 + \frac{1}{(n-1)\omega_{n-1}} \left(\frac{|\Sigma|}{w_{n-1}} \right)^{\frac{2-n}{n-1}} \int_{\Sigma} \frac{\partial}{\partial s}\rvert_{s=0} H d\sigma. \nonumber
\end{eqnarray}
On the other hand, first variation of mean curvature for the variation field $\frac{\nabla V}{|\nabla V|^{2}}$ is
\begin{eqnarray*}
    \frac{\partial}{\partial s}\vert_{s =0} H &=& - \Delta_{\Sigma} |\nabla V|^{-1} - (|h|^{2} + \text{Ric}(\nu,\nu)) |\nabla V|^{-1} \\
    &=& -\text{Ric}(\nu,\nu) \kappa^{-1}. \nonumber
\end{eqnarray*}
Here we have used the properties $h=0$ and $\frac{\partial V}{\partial \nu} = \kappa >0$ from Proposition \ref{horizon_boundary} again. Given that $R =0$ and $h_{\Sigma}=0$, the Gauss equation reduces to $-2\text{Ric}(\nu,\nu)= R_{\sigma} $ on $\Sigma$, and so
\begin{equation*}
    \frac{\partial}{\partial s} \vert_{s=0} H= \frac{1}{2} R_{\sigma} \kappa^{-1} \hspace{1cm} \text{on} \hspace{0.5cm} \Sigma^{n-1}.
\end{equation*}
Substituting into \eqref{horizon_deriv} then yields
\begin{eqnarray*}
    \kappa &\leq&  \frac{1}{2(n-1)\omega_{n-1}} \left(\frac{|\Sigma|}{w_{n-1}} \right)^{\frac{2-n}{n-1}} \int_{\Sigma} R_{\sigma} d\sigma,
\end{eqnarray*}
and so

\begin{eqnarray*}
    m &=& \frac{1}{(n-2) w_{n-1}} \int_{\Sigma } \frac{\partial V}{\partial \nu} d\sigma \\
    &\leq& \frac{1}{2(n-1)(n-2)w_{n-1}} \left(\frac{|\Sigma|}{w_{n-1}} \right)^{\frac{1}{n-1}} \int_{\Sigma} R_{\sigma} d\sigma = \frac{1}{2} \frac{E(\Sigma,\sigma)}{E(\mathbb{S}^{n-1},g_{S^{n-1}})} \left( \frac{|\Sigma|}{w_{n-1}} \right)^{\frac{n-2}{n-1}}.
\end{eqnarray*}
This gives the upper bound and the associated uniqueness statement in Theorem \ref{black_hole_uniqueness}.
\end{proof}
\appendix
\section{Proof of Theorem 1.1}
 IMCF \eqref{IMCF_flow} defines a fully non-linear system of parabolic PDE, but it admits an elliptic formulation as well: consider the following degenerate-elliptic exterior Dirichlet problem on a Riemannian manifold $(M^{n},g)$ with boundary $\partial M$:

\begin{eqnarray} \label{IMCF}
   \text{div}\left(\frac{\nabla u(x)}{|\nabla u(x)|}\right) &=& |\nabla u(x)| \hspace{2cm} x \in M  \\
   u|_{\partial M} &\equiv& 0 \nonumber.
\end{eqnarray}
Direct verification shows that the level-set flow $F: \Sigma^{n-1} \times [0,T) \rightarrow (M^{n},g)$ of \eqref{IMCF}, where $\Sigma_{t}= \{ u =t \} \subset M^{n}$, is a solution to \eqref{IMCF_flow}. In \cite{HI}, Huisken and Ilmanen consider variational weak solutions of \eqref{IMCF}. A proper weak solution $u$ of \eqref{IMCF} is a locally Lipschitz function on $(M^{n},g)$ with $u|_{\partial M}=0$ and with compact level sets solving \eqref{IMCF} in a variational sense (see p. 364 of \cite{HI} for the precise functional). The ``weak flow surfaces" $\Sigma_{t}$ of $u$ are the sets 
\begin{equation}
 \Sigma_{t}= \partial \Omega_{t} \hspace{1cm} \Omega_{t}= \{ u < t \} \subset \subset M^{n}.   
\end{equation}
There exists a unique proper weak solution of \eqref{IMCF} on any $(M^{n},g)$ with $C^{1,\alpha}$ boundary $\partial M$ that is weakly asymptotically flat in the sense of \eqref{weakly_af}-- see Theorem 3.1 and remark 1 in \cite{HI}. Each $\Sigma_{t}$ is a $C^{1,\alpha}$ hypersurface away from a singular set $Z$ of Hausdorff dimension at most $n-8$ and has weak mean curvature $H \in L^{2}(\Sigma_{t})$ with $H(x)=|\nabla u(x)| >0$ for a.e. $x \in \Sigma_{t}$ and each $t > 0$. Finally, when $\Sigma_{0}=\partial M$ is outer-minimizing, $|\Sigma_{t}|= e^{t}|\Sigma_{0}|$ as with the classical flow.

Wei's proof of monotonicity in Section 4 of \cite{Wei} works if the Schwarzschild potential $V_{m}$ is replaced with an arbitrary static potential $V$ as noted by McCormick, but another important point is that this monotonicity argument \textit{is completely local}, i.e. is independent of asymptotics. We include the proof of monotonicity under either assumption of Theorem \ref{Mccormick} here for convenience-- we also remark that the additional minimal components of $\partial M$ in low dimensions must be stable, which is not stated in Theorem 1.1 of \cite{Mc}. When $\partial M \setminus \Sigma$ is non-empty, weak IMCF must be interrupted at discrete times to jump over the additional boundary components as in Section 6 of \cite{HI}, and when $n>7$ one cannot necessarily restart the weak flow from these constructed jumps due to a singular set of positive Hausdorff dimension. However, it is not necessary to construct jumps when $\partial M$ is connected, and so showing weak flow monotonicity in all dimensions suffices in that case. For the asymptotic analysis, only the assumptions of Huisken-Ilmanen's blow-down lemma (Lemma 7.1 in \cite{HI}) are needed to determine the $L^{1}$ limit of the weak mean curvature.

\subsection{Monotonicity}
In studying variational solutions of \eqref{IMCF}, it is useful to consider a solution $u^{\epsilon}$ the perturbed, non-degenerate elliptic Dirichlet problem

\begin{eqnarray} \label{perturbed_dirichlet}
    \text{div} \left( \frac{\nabla u^{\epsilon}(x)}{\sqrt{|\nabla u^{\epsilon}(x)|^{2} + \epsilon^{2}}} \right) &=& \sqrt{|\nabla u^{\epsilon}(x)|^{2} + \epsilon^{2}} \hspace{2cm} x \in \Omega_{L} \label{u_epsilon} \\
    u^{\epsilon}|_{\partial M} &=& 0, \\ \nonumber
    u^{\epsilon}|_{\partial \Omega_{L} \setminus \partial M} &=& L-2. \nonumber
\end{eqnarray}
Here, the domain $\Omega_{L}= \{ v < L \}$ is defined using a subsolution at infinity of \eqref{IMCF}-- for $(M^{n},g)$ satisfying \eqref{af}, this is given by the expanding sphere solution $v(x)= (n-1) \log(|x|)$. From section 3 of \cite{HI}, there exists a sequence $\epsilon_{i} \rightarrow 0$ such that $u_{i}=u^{\epsilon_{i}}$ solving \eqref{perturbed_dirichlet} converge locally uniformally to a weak solution $u$ of \eqref{IMCF}. Due to elliptic regularity, each $u_{i}$ is proper and smooth, and so almost every level set $\Sigma^{i}_{s}=\{ u_{i} = s \}$ is regular. This allows one to apply the following lemma to these $u_{i}$ over static $(M^{n},g)$.

\begin{lemma}
Let $(M^{n},g,V)$ be static and $u: M^{n} \rightarrow \mathbb{R}^{+}$ a proper smooth function with $u|_{\partial M} \equiv 0$. Suppose also that the set $C(u)$ of critical points of $u$ has $\mathcal{H}^{n-1}$ measure $0$. Define the domain $\Omega_{t} = \{ u \leq t \} \subset M^{n}$, and let $\Phi: (0,t) \rightarrow \mathbb{R}^{+}$ be Lipschitz and compactly supported. Then for the function
\begin{equation*}
    \phi(x) = (\Phi \circ u)(x), \hspace{2cm} x \in \Omega_{t},
\end{equation*}
 we have the integral identity

 \begin{equation} \label{integral_identity}
     - \int_{\Omega_{t}} V H \frac{\partial \phi}{\partial \nu} d\Omega = \int_{\Omega_{t}} \phi [2H \frac{\partial V}{\partial \nu} + V(H^{2} - |h|^{2})] d \Omega,
 \end{equation}
 where $\nu$, $H$, and $h$ are the unit normal, mean curvature, and second fundamental form of the regular level sets of $u$.
\end{lemma}

\begin{proof}
We will initially assume $\Phi \in C^{1}_{0}(0,t)$. By Sard's theorem, $\nabla u$ is non-vanishing over the level set $\Sigma_{s}= \{ u = s \}$ for a.e. $s \in [0,t)$, and hence those level sets are regular. For a regular level set $\Sigma_{s}$, we consider the variation field

\begin{equation*}
    \frac{\nabla u}{|\nabla u|^{2}} = |\nabla u|^{-1} \nu.
\end{equation*}
The variation formula for mean curvature reads

\begin{equation*}
    -|\nabla u|^{-1} \frac{\partial H}{\partial \nu} = \Delta_{\Sigma_{s}} |\nabla u|^{-1} + (|h|^{2} + \text{Ric}(\nu,\nu)) |\nabla u|^{-1}.
\end{equation*}
Multiplying by $V$ and integrating yields

\begin{eqnarray} \label{level_set_u}
    \int_{\Sigma_{s}} -V |\nabla u|^{-1} \frac{\partial H}{\partial \nu} d\sigma &=& \int_{\Sigma_{s}} V \Delta_{\Sigma_{s}} |\nabla u|^{-1} + (|h|^{2} + \text{Ric}(\nu,\nu)) V |\nabla u|^{-1} d\sigma \nonumber \\
    &=& \int_{\Sigma_{s}} (\Delta_{\Sigma_{s}} V + V \text{Ric}(\nu,\nu) +|h|^{2}) |\nabla u|^{-1} d\sigma \\
    &=& \int_{\Sigma_{s}} - H |\nabla u|^{-1} \frac{\partial V}{\partial \nu} + V|h|^{2} |\nabla u|^{-1} d\sigma. \nonumber 
\end{eqnarray}
In the last line of we used the static equation $\Delta_{\Sigma_{s}} V + V \text{Ric}(\nu,\nu) = - \frac{\partial V}{\partial \nu} H$. Since \eqref{level_set_u} holds on $\Sigma_{s}$ for a.e. $s \in (0,t)$, we find by the co-area formula that

\begin{eqnarray} \label{H_deriv_identity}
    \int_{\Omega_{t}} V \phi \frac{\partial H}{\partial \nu} d\Omega &=& \int_{0}^{t} \int_{\Sigma_{s}} V \phi \frac{\partial H}{\partial \nu} |\nabla u|^{-1} d\sigma ds = \int_{0}^{t} \Phi(s) \int_{\Sigma_{s}} V \frac{\partial H}{\partial \nu} |\nabla u|^{-1} d\sigma ds \nonumber \\
    &=& \int_{0}^{t} \Phi(s) \int_{\Sigma_{s}} H \frac{\partial V}{\partial \nu} |\nabla u|^{-1} - V|h|^{2} |\nabla u|^{-1} d\sigma ds \\
    &=& \int_{\Omega_{t}} \phi (H \frac{\partial V}{\partial \nu} - V|h|^{2}) d\Omega. \nonumber
\end{eqnarray}
We compute

\begin{eqnarray} \label{divergence_formula}
    \text{div}( V \phi H \nu) &=& \phi H \frac{\partial V}{\partial \nu} + V H \frac{\partial \phi}{\partial \nu} + V \phi \frac{\partial H}{\partial \nu} + V\phi H \text{div} \nu
\end{eqnarray}
Because $\Phi$ is compactly supported, $\phi$ vanishes over $\partial \Omega_{t}= \{ u = t\} \cup \partial M$. Therefore, the integral of the left-hand side of \eqref{divergence_formula} vanishes, and we are left with

\begin{eqnarray} \label{final_identity}
    -\int_{\Omega_{t}} VH \frac{\partial \phi}{\partial \nu} d\Omega &=& \int_{\Omega_{t}} \phi H \frac{\partial V}{\partial \nu} + V \phi \frac{\partial H}{\partial \nu} + V\phi H \text{div} \nu d\Omega \\
    &=& \int_{\Omega_{t}} \phi( 2H \frac{\partial V}{\partial \nu} - V|h|^{2}) d\Omega + \int_{\Omega_{t}} V\phi H \text{div} \nu d\Omega. \nonumber
\end{eqnarray}
In the last line, we applied \eqref{H_deriv_identity}. For the divergence term, we once again compute via the co-area formula
\begin{eqnarray*}
    \int_{\Omega_{t}} V\phi H \text{div} \nu d\Omega &=& \int_{0}^{t} \int_{\Sigma_{s}} V \phi H \text{div}(\nu) |\nabla u|^{-1} d\sigma ds \\
    &=& \int_{0}^{t} \int_{\Sigma_{s}} V \phi H \text{div}_{\Sigma_{s}}(\nu) |\nabla u|^{-1} d\sigma ds \\
    &=& \int_{0}^{t} \int_{\Sigma_{s}} V \phi H^{2} |\nabla u|^{-1} d\sigma ds \\
    &=& \int_{\Omega_{t}} V\phi H^{2} d\Omega.
\end{eqnarray*}
Substituting this into \eqref{final_identity} yields \eqref{integral_identity}.
\end{proof}

Wei shows the following convergence for the mean curvature of the regular level sets of $u_{i}$ solving \eqref{u_epsilon} to the mean curvature $H$ of the weak solution $u$ of \eqref{IMCF} (See section 2 and the proof of Lemma 4.2 of \cite{Wei} and section 5 of \cite{HI}):

\begin{eqnarray}
        \int_{\Omega^{i}_{t}} \phi H_{i} d\Omega &\rightarrow& \int_{\Omega_{t}} \phi H d\Omega, \hspace{2cm} \phi \in C^{0}_{c}(\Omega_{t}), \label{L1_conv} \\
     \int_{\Omega^{i}_{t}} \phi H_{i}^{2} d\Omega &\rightarrow& \int_{\Omega_{t}} \phi H^{2} d\Omega, \label{L2_conv}
\end{eqnarray}

Via the inequality $H^{2} \leq (n-1) |h|^{2}$, \eqref{integral_identity} yields an integral inequality for the $u_{i}$  in terms of the mean curvature $H_{i}$ and the normal derivatives of $u_{i}$ and $V$. Using \eqref{L1_conv}, \eqref{L2_conv}, this inequality passes to the level sets $\Sigma_{t}$ of the variational solution. An appropriate choice of $\Phi$ then yields an upper bound on the $L^{1}$ norm of $VH$.

\begin{lemma} 
Let $(M^{n},g)$ be asymptotically flat with connected, outer-minimizing boundary $\partial M=\Sigma$, and let $\Sigma_{t}$ the weak solution to IMCF with initial data $\Sigma$. Suppose $(M^{n},g)$ admits a static potential $V$. Then for all $0\leq \tilde{t} < t$ we have

\begin{equation} \label{key_inequality}
    \int_{\Sigma_{t}} VH d\sigma \leq \int_{\Sigma_{\tilde{t}}} VH d\sigma + \frac{n-2}{n-1} \left( \int_{\tilde{t}}^{t} \int_{\Sigma_{s}} VH d\sigma + 2(n-1)m w_{n-1} ds \right).
\end{equation}
\end{lemma}

\begin{proof}
For $u^{\epsilon_{i}}=u_{i}$, we have by \eqref{integral_identity} and the inequality $H^{2} \leq (n-1) |h|^{2}$, we have that

\begin{equation*}
    -\int_{\Omega^{i}_{t}} VH_{i} \frac{\partial \phi}{\partial \nu} d\Omega \leq \int_{\Omega^{i}_{t}} \phi_{i} [ 2 H_{i} \frac{\partial V}{\partial \nu_{i}} + \frac{n-2}{n-1} H_{i}^{2}] d\Omega.
\end{equation*}
Taking limits then yields

\begin{equation} \label{weak_imcf_inequality}
    -\int_{\Omega_{t}} VH \frac{\partial \phi}{\partial \nu} d\Omega \leq \int_{\Omega_{t}} \phi[2H \frac{\partial V} {\partial \nu} + \frac{n-2}{n-1} VH^{2}] d\Omega
\end{equation}
\
Applying \eqref{weak_imcf_inequality} along with the co-area formula for Lipschitz functions, c.f. \cite{S18}, we find for any $\Phi \in C^{0,1}_{0}((0,t))$ that 

\begin{eqnarray*}
    -\int_{0}^{t} \Phi'(s) \int_{\Sigma_{s}} VH d\mu_{s} ds &=& - \int_{\Omega_{t}} \Phi'(s) \frac{\partial u}{\partial \nu} VH d\Omega \\
    &\leq& \int_{\Omega_{t}} \phi \left( 2 \frac{\partial V}{\partial \nu} H + \frac{n-2}{n-1} VH^{2} \right) d\Omega \\
    &=& \int_{\Omega_{t}} \phi \left( 2 \frac{\partial V}{\partial \nu} + \frac{n-2}{n-1} VH \right) |\nabla u| d\Omega \\
    &=& \int_{0}^{t} \Phi(s) \int_{\Sigma_{s}} \left( 2 \frac{\partial V}{\partial \nu} + \frac{n-2}{n-1} VH \right) d\sigma ds.
\end{eqnarray*}

for $\Phi \in C^{0,1}_{0}((0,t))$. Fix $0 \leq \tilde{t} < t$, and for $0 < \delta < \frac{1}{2} (t-\tilde{t}) $, choose

\begin{equation*}
    \Phi(s) = \begin{cases} 0 & s \in (0,\tilde{t})  \\
    \frac{1}{\delta} (s- \tilde{t}) & s \in (\tilde{t}, \tilde{t} + \delta) \\
    1 & s \in (\tilde{t} + \delta, t - \delta) \\
    \frac{1}{\delta} (t-s) & s \in (t - \delta, t).
    
    \end{cases}
\end{equation*}
Then 

\begin{equation*}
    \frac{1}{\delta} \int_{t - \delta}^{t} \int_{\Sigma_{s}} VH d\sigma ds - \frac{1}{\delta} \int_{\tilde{t}}^{\tilde{t}+\delta} \int_{\Sigma_{s}} VH d\sigma ds \leq \int_{\tilde{t}}^{t} \int_{\Sigma_{s}} \left( 2 \frac{\partial V}{\partial \nu} + \frac{n-2}{n-1} VH \right) d\sigma ds.
\end{equation*}
The limit of the left-hand side as $\delta \rightarrow 0$ is well-defined for a.e. $0 < \tilde{t} < t$, and so for a.e. $0 < \tilde{t} < t$ we have

\begin{eqnarray}
    \int_{\Sigma_{t}} VH d\mu - \int_{\Sigma_{\tilde{t}}} VH d\sigma &\leq& \int_{\tilde{t}}^{t} \int_{\Sigma_{s}} \left( 2 \frac{\partial V}{\partial \nu} + \frac{n-2}{n-1} VH \right) d\sigma ds \nonumber \\
    &=& \int_{\tilde{t}}^{t} \left( 2(n-2) w_{n-1}m + \frac{n-2}{n-1}  \int_{\Sigma_{s}}VH d\sigma \right) ds \label{VH_upper_bound}.
\end{eqnarray}
Note that $\int_{\Sigma_{s}} \frac{\partial V}{\partial \nu} d\sigma = \int_{\Sigma_{0}} \frac{\partial V}{\partial \nu} d \sigma = (n-2)w_{n-1} m$ due to the static equation $\Delta V = 0$ and the divergence theorem applied on $\Omega_{s}$. 

To extend this inequality to all times, for any $0 < \tilde{t} < t$ we choose sequences $t_{j} \nearrow t$ and $\tilde{t}_{j} \nearrow \tilde{t}$ such that \eqref{VH_upper_bound} holds for the times $\tilde{t}_{j} < t_{j}$. By (1.10) in \cite{HI}, we have that 

\begin{equation*}
    \Sigma_{t_{j}} \rightarrow \Sigma_{t} \hspace{2cm} \text{  in  } \hspace{1cm} C^{1,\alpha}, \\
\end{equation*}
resp. for $\tilde{t}_{j},\tilde{t}$, away from a small singular set $Z$. Equation (1.13) in \cite{HI} then implies $L^{1}$ convergence of the weak mean curvature, and so altogether we find that

\begin{eqnarray*}
    \lim_{j \rightarrow \infty} \int_{\Sigma_{t_{j}}} VH d\sigma &=& \int_{\Sigma_{t}} VH d\sigma, \\
    \lim_{j \rightarrow \infty} \int_{\Sigma_{\tilde{t}_{j}}} VH d\sigma &=& \int_{\Sigma_{\tilde{t}}} VH d\sigma.
\end{eqnarray*}
Therefore, applying \eqref{VH_upper_bound} to $t_{j}, \tilde{t}_{j}$, we may conclude

\begin{eqnarray*}
    \int_{\Sigma_{t}} VH d\mu - \int_{\Sigma_{\tilde{t}}} VH d\sigma &\leq& \lim_{j \rightarrow \infty} \int_{\tilde{t}_{j}}^{t_{j}} \left( 2(n-2) w_{n-1}m + \frac{n-2}{n-1}  \int_{\Sigma_{s}}VH d\sigma \right) ds \\
    &=& \int_{\tilde{t}}^{t} \left( 2(n-2) w_{n-1}m + \frac{n-2}{n-1}  \int_{\Sigma_{s}}VH d\sigma \right) ds.
\end{eqnarray*}
\end{proof}

\begin{theorem} \label{monotone}
Under the hypotheses of the previous lemma, the functional $Q(t)$ defined by

\begin{equation}
    Q(t) = |\Sigma_{t}|^{-\frac{n-2}{n-1}} \left( \int_{\Sigma_{t}} VH d\sigma + 2(n-1) w_{n-1} m \right)
\end{equation}
is monotone non-increasing under weak inverse mean curvature flow for each $t \in (0,\infty)$.
\end{theorem}

\begin{proof}
Define $h: [0,\infty) \rightarrow \mathbb{R}$ by

\begin{equation*}
    h(t)= \int_{\Sigma_{t}} VHd\sigma + 2(n-1)w_{n-1}m.
\end{equation*}
It follows from Lemma A.2 that for $0 \leq t_{1} < t_{2}$, 
\begin{equation*}
    h(t_{2}) - h(t_{1}) \leq \frac{n-2}{n-1} \int_{t_{1}}^{t_{2}} h(s) ds.
\end{equation*}
It follows from Gronwall's Lemma that $h(t_{2}) \leq e^{\frac{n-2}{n-1} (t_{2}-t_{1})} h(t_{1})$, and hence $Q(t_{2}) \leq Q(t_{1})$.
\end{proof}
For the case of $n < 8$ and additional horizon boundary components, we use a result by \cite{HMM} that McCormick applies in the asymptotically flat case (this is where the area-minimizing assumption is necessary). We argue slightly differently here, but the conclusion is identical.
\begin{lemma} \label{extra_boundary}
Let $(M,g,V)$, $3 \leq n \leq 7$, be asymptotically flat  with boundary satisfying the second item of Theorem \ref{Mccormick}. Let $\Sigma_{i} \subset \partial M \setminus \Sigma$ be a locally area-minimizing hypersurface. Then

\begin{equation*}
    \int_{\Sigma_{i}} \frac{\partial V}{\partial \nu} d\sigma \geq 0.
\end{equation*}
As a result, if $\Omega \subset M^{n}$ is a bounded domain, we have that

\begin{equation} \label{m_upper_bound}
\int_{\partial \Omega \setminus \partial M} \frac{\partial V}{\partial \nu} d\sigma \leq (n-2) w_{n-1} m.
\end{equation}
\end{lemma}

\begin{proof}
By Lemma 4 of \cite{HMM}, we have either that $V|_{\Sigma_{i}} > 0$ or $V|_{\Sigma_{i}} \equiv 0$. By the arguments in that lemma, we know that $\frac{\partial V}{\partial \nu} = \kappa$ for a constant $k >0$ in the latter case. In the former case, Proposition 5(3) of \cite{HMM} implies that the Ricci curvature of $V$ vanishes in a collar neighborhood of $\Sigma_{i}$. For the $\Lambda =0$ static equations, this means that $V$ is constant within this neighborhood, which implies $V \equiv 1$ on $M$. Then $\int_{\Sigma_{i}} \frac{\partial V}{\partial \nu} d\sigma \geq 0$ and hence $\int_{\partial \Omega \cap \partial M} \frac{\partial V}{\partial \nu} d\sigma \leq (n-2)w_{n-1} m$. Applying the divergence theorem over $\Omega$ then yields \eqref{m_upper_bound}.
\end{proof}
For a proper weak solution $u: (M^{n},g) \rightarrow \mathbb{R}$ to IMCF with $u|_{\Sigma} \equiv 0$, \eqref{m_upper_bound} implies that \eqref{VH_upper_bound}, and hence monotonicity of $Q(t)$, holds on $\Sigma_{t}= \partial \Omega_{t} \setminus \partial M$. Arguing the same as in Section 6 of \cite{HI} and Section 4.2 of \cite{Wei}, we also obtain monotonicity of weak IMCF with constructed jumps. 

\begin{theorem} \label{constructed_jumps}
Let $(M,g,V)$, $3 \leq n \leq 7$, be asymptotically flat  with boundary satisfying the second item of Theorem \ref{Mccormick}. Then there exists a flow $\{ \Sigma_{t} \}_{0 \leq t < \infty}$ of $C^{1,\alpha}$ hypersurfaces such that $Q(t)$ is monotone non-increasing in $t$ and $\{ \Sigma_{t} \}_{T \leq t < \infty}$ are a proper weak solution to IMCF for a sufficiently large $T$.
\end{theorem}
\begin{proof}
Take fill-ins of $W_{1}, \dots, W_{k}$ of the minimal boundary components $\Sigma_{1}, \dots, \Sigma_{k}$. For the weak IMCF $\Sigma_{t}$ of $\Sigma$, there exists a time $t_{1} > 0$ and fill-in $W_{j_{1}}$ such that $\Omega_{t_{1}} \cup W_{j_{1}}$ is not outer-mininimizing in $(M^{n},g)$. Let $F_{1}$ be the minimizing hull of $\Omega_{t_{1}} \cup W_{j_{1}}$. Note that $\partial F_{1}$ is a $C^{1,\alpha}$ hypersurface for $n < 8$ by Theorem 1.3 from \cite{HI}. Furthermore, $\partial F_{1}$ is connected, since any components of $\partial F_{1}$ which do not enclose $\Sigma_{t_{1}}$ are area-minimizing hypersurfaces and hence are disjoint from $\partial_{j_{1}} M$ by the maximum principle, which contradicts Lemma 4 of \cite{HMM}. The boundary of $F_{1}$ satisfies

\begin{eqnarray}
    |\partial F_{1}| &\geq& |\Sigma_{t_{1}}|, \label{jump1} \\
    \int_{F_{1}} VH d\sigma &\leq& \int_{\Sigma_{t_{1}}} VH d\sigma. \label{jump2}
\end{eqnarray}
The first inequality follows from the fact that $\Omega_{t_{1}}$ is outer-minimizing in $(M^{n},g)$, while the second follows from the fact that $H = 0$ a.e. on $\partial F_{1} \setminus \Sigma_{t_{1}}$ (equation (1.15) in \cite{HI}). Next, given $\partial F_{1}$ is $C^{1,\alpha}$, we consider the proper weak solution of IMCF with initial condition $\partial F_{1}$. For $t < t_{1}$ and $t > t_{1}$, $Q(t)$ is monotone decreasing by Lemma \ref{extra_boundary} and \eqref{VH_upper_bound}. By \eqref{jump1}-\eqref{jump2}, we have that $\lim_{t \nearrow t_{1}} Q(t) \geq \lim_{t \searrow t_{1}} Q(t)$. Repeating this procedure for times $t_{2} < \dots< t_{k}$ and boundary components $W_{j_{2}}, \dots, W_{j_{k}}$ yields the conclusion.  
\end{proof}

\subsection{Asymptotic Convergence}
The $L^{1}$ limit of the weak mean curvature of $\Sigma_{t}$ determines the limit of $Q(t)$. The blow-down lemma of Huisken-Ilmanen allows one to evaluate this limit on any asymptotically flat $(M^{n},g)$.
\begin{prop}
Let $(M^{n},g)$ be asymptotically flat , and $U \cong \mathbb{R}^{n} \setminus B_{\frac{1}{2}}(0)$ be the exterior region of $(M^{n},g)$. Let $u: (M^{n},g) \rightarrow \mathbb{R}^{+}$ be a proper weak solution to IMCF with initial condition $\Sigma_{0}$, and consider the blow-down objects associated with $u$ from Section 3. Then the total mean curvature of $\widetilde{\Sigma}_{t} \subset (U,g_{t})$ converges to that of the unit sphere, i.e.

\begin{equation*}
    \lim_{t \rightarrow \infty} \int_{\widetilde{\Sigma}_{t}} \widetilde{H} d\sigma = (n-1) w_{n-1}.
\end{equation*}
\end{prop}
By the blowdown Lemma 7.1 in \cite{HI}, there exist constants $\lambda \rightarrow 0$, $c_{\lambda} \rightarrow \infty$ so that the dilated solution $u_{\lambda}(\textbf{x}) = u(\lambda^{-1} \mathbf{x}) - c_{\lambda}$ converges locally uniformally to the expanding sphere solution in $\mathbb{R}^{n} \setminus \{ 0 \}$. The gradient estimate of $u_{\lambda}$ is independent of $\lambda$ (equation $(7.2)$ in their proof), and so according to the Compactness Theorem 2.1 and the remark following it in \cite{HI}, we have that

\begin{equation*}
    \widetilde{\Sigma}_{t} \rightarrow \partial B_{1} (0) \subset U \hspace{1cm} \text{ in } \hspace{1cm} C^{1,\alpha},
\end{equation*}
 as $t \rightarrow \infty$. Once again by (7.2) in Lemma 7.1, we also have that $\widetilde{H}$ is essentially bounded as $t \rightarrow \infty$. In view of the discussion on pages 370-371, we have for every compactly-supported vector field that $ X \in \Gamma_{c}(TU)$ that

\begin{equation*}
    \int_{\widetilde{\Sigma_{t}}} \widetilde{H} \langle \widetilde{\nu}, X \rangle d\sigma \rightarrow \int_{\partial B_{1}(0)} H_{\partial B_{1}(0)} \langle \nu_{\partial B_{1}(0)}, X \rangle d\sigma \hspace{1cm} \text{ as  } \hspace{1cm} t \rightarrow \infty.
\end{equation*}
Choosing $X=\frac{\partial}{\partial r}$ and using $C^{1,\alpha}$ convergence of $\widetilde{\Sigma}_{t}$ yields $L^{1}$ convergence of $\widetilde{H}$.  
\begin{cor} \label{convergence}
Let $(M^{n},g,V)$ be asymptotically flat . Then for a proper weak solution to IMCF in $(M^{n},g,V)$ we have

\begin{equation*}
    \lim_{t \rightarrow \infty} Q(t) = (n-1)w_{n-1}^{\frac{1}{n-1}}.
\end{equation*}
\end{cor}

\begin{proof}
We have

\begin{eqnarray}
    Q(t) &=& |\Sigma_{t}|^{-\frac{n-2}{n-1}} \left( \int_{\Sigma_{t}} V H d\sigma + 2m \right) \\
    &=& w_{n-1}^{-\frac{n-2}{n-1}} \int_{\widetilde{\Sigma}_{t}} \widetilde{V} \widetilde{H} d\sigma + 2m |\Sigma_{t}|^{-\frac{n-2}{n-1}}, \nonumber
\end{eqnarray}
where $\widetilde{V}(x)= V(r(t) x)$ in the end $U$. Since

\begin{equation*}
    \left( \min_{\Sigma_{t}} V  \right) \int_{\widetilde{\Sigma}_{t}} H d\sigma \leq \int_{\widetilde{\Sigma}_{t}} \widetilde{V} \widetilde{H} d\sigma \leq \left( \max_{\Sigma_{t}} V \right) \int_{\widetilde{\Sigma}_{t}} H d\sigma,
\end{equation*}
the result follows from the convergence of $V$ and the previous proposition.
\end{proof}

\section{The mean curvature of spheres in Schwarzschild}
Setting $r = s (1 + \frac{m}{2} s^{2-n})^{\frac{2}{n-2}}$ reveals that the Schwarzschild metric is conformally flat
\begin{align*}
g_m = \lt( 1 + \frac{m}{2} s^{2-n}\rt)^{\frac{4}{n-2}} (ds^2 + s^2 g_{S^{n-1}}).
\end{align*}
Here $g_m$ is defined on $\R^n \setminus B_{s_0}$ with $s_0 = \lt( \frac{m}{2}\rt)^{\frac{1}{n-2}}$ for $m > 0$ and $s_0 =0$, the curvature singularity, for $m < 0$. For $m > 0$, the horizon is the sphere $\{ s = s_0\}$.

Let $S$ be a sphere in $\R^n \setminus B_{s_0}$ that encloses $B_{s_0}$ with defining equation $|x - x_0|^2 = s_1^2$. Let $\phi = 1 + \frac{m}{2} s^{2-n}$. The mean curvature of $S$ with respect to $g_m$ is given by
\begin{align}
H &= \phi^{-\frac{2}{n-2}} \lt( \frac{n-1}{s_1} + \frac{2(n-1)}{n-2} \phi^{-1} \frac{\pl\phi}{\pl\nu_1}\rt)\notag\\
&= \phi^{-\frac{2}{n-2}} \lt( \frac{n-1}{s_1} -2 (n-1) \phi^{-1} s^{1-n} \frac{m}{2} \langle \nu_1,\pl_s \rangle \rt) \label{formula of mean curvature for spheres in Sch}
\end{align}
where $\nu_1$ denotes the outward normal vector of $S$ and $\langle \nu_1, \pl_s \rangle$ denotes the inner product, both with respect to the Euclidean metric. We immediately see that when $m<0$, all spheres enclosing $\{ 0\}$ are strictly mean-convex since $\langle \nu_1,\pl_s \rangle > 0$. 

It is different for $m > 0$. For any $s_1 > s_0$  we translate the sphere $\pl B_{s_1}(0)$ upward. Initially, the sphere is strictly mean-convex. But at the instant $\pl B_{s_1}(x_1)$ with
\[ \frac{1}{s_1} - 2 (\phi(|x_1| -s_1))^{-1} (|x_1|-s_1)^{1-n} \frac{m}{2} \] 
it becomes weakly mean-convex: the mean curvature is zero\footnote{This fact was overlooked in \cite{Wei}.} at the south pole and positive elsewhere. To verify the last assertion, use high school algebra to show
\begin{align*}
\langle \nu_1,\pl_s \rangle &= \langle \frac{x-x_1}{s_1}, \frac{x}{s} \rangle = \frac{s^2 - \langle x,x_1 \rangle}{s s_1} = \frac{s^2 - |x_1|^2 + s_1^2}{2s s_0},
\end{align*} 
where the defining equation $s^2 - 2 \langle x,x_1 \rangle + |x_1|^2 = s_1^2$ is used in the last equality, and get 
\begin{align*}
\frac{d}{ds} \lt( s^{1-n} \langle \nu_1,\pl_s \rangle \rt) = \frac{s^{-n-1}}{2s_0} \lt( (2-n)s^2 + n |x_1|^2 - n s_1^2 \rt) < 0
\end{align*}
since $s_1 > |x_1|$. If we further translate $\pl B_{s_1}(x_1)$ up, the cap centered at the south pole with $H < 0$ grows larger.  Since $\phi s$ increases in $s$, we deduce from \eqref{formula of mean curvature for spheres in Sch} that for radii $s_1 < s_2$, the corresponding south pole $x_2 - s_2 \vec e_n$ is higher than $x_1 - s_1 \vec{e}_n$, see Figure \ref{spheres}. We also infer that  $\pl B_{t}(x_1), s_1 < t < \infty$ forms an strictly mean-convex foliation of $\R^n \setminus \overline{ B_{s_1}(x_1)}$. It follows that, see  \cite[Proposition 2.6]{HW} for example, $\pl B_{s_1}(x_1)$ is strictly outer-minimizing. The same argument shows that strictly mean-convex spheres are strictly outer-minimizing.     

\begin{figure}
\centering
\scalebox{0.5}{
\begin{tikzpicture}
\draw (0,0) circle (1);
\draw (0,2) circle (3.5);
\draw (0,2.4) circle (3.7);
\draw (0,0) --(1,0) node[below, midway] {$s_0$};
\draw (0,2)--(3.5,2) node[below, midway] {$s_1$};
\draw (0,2.4)--(3.7,2.4) node[above,midway] {$s_2$};
\node (A) at (0,2) [below] {$x_1$};
\node (B) at (0,0) [above] {$0$};
\node (C) at (0,2.4) [above] {$x_2$};
\end{tikzpicture}
}
\caption{Two spheres in $\mathcal{A}$ with radius $s_1 < s_2$. The sphere $\pl B_{s_0}(0)$ is the horizon.}
\label{spheres}
\end{figure}

We summarize the discussion into the following lemma.
\begin{lem}\label{property of spheres}
Consider Schwarzschild space with $m \neq 0$. When $m < 0$, all spheres that enclose the curvature singularity are strictly mean-convex. 

When $m> 0$, there is a nonempty collection of spheres $\mathcal{A}$ that enclose the horizon, has zero mean curvature at one point and positive mean curvature elsewhere. Spheres in $\mathcal{A}$, as well as spheres that enclose the horizon and are strictly mean-convex, are strictly outer-minimizing.  
\end{lem}

We include a well-known fact used in the proof of Theorem \ref{equality case}.
\begin{lem}\label{path-connected}
Let $E$ be a subset of Hausdorff dimension less than $n-1$ in $\R^n, n \ge 2$. Then $\R^n \setminus E$ is path-connected.
\end{lem}
\begin{proof}
We first treat the case $n=2$. Let $x,y \in \R^2 \setminus E$. Consider the space of arcs connecting $x$ and $y$ which can be identified with $[0,\pi]$ equipped with the Lebesgue measure, see Figure \ref{space of arcs}.  By Area theorem, almost every arc has empty intersection with $E$. This proves that $\R^2 \setminus E$ is path-connected. 

\begin{figure}
\centering
\begin{tikzpicture}
\node (A) at (0,-1) [below] {$x$};
\node (B) at (0,1) [above] {$y$};
\draw (0,-1) -- (0,1);
\draw (0,-1) to [bend right] (0,1);
\draw (0,-1) to [bend left] (0,1);
\draw (0,-1) to [bend right=75] (0,1);
\draw (0,-1) to [bend left=75] (0,1);
\draw (0,-1) arc (-90:90:1);
\draw (0,-1) arc (270:90:1);
\end{tikzpicture}
\caption{Space of arcs}
\label{space of arcs}
\end{figure}

Let $x,y \in \R^n \setminus E$ and consider the space of hyperplanes containing $x$ and $y$. By Area theorem, almost every hyperplane has $\dim (P \cap E) < n-2$. By induction hypothesis, $x$ and $y$ are connected by a path on these hyperplanes. 
\end{proof}

\printbibliography[title={References}]

\begin{center}
\textnormal{ \large Centre for Geometry and Topology \\
University of Copenhagen\\
Copenhagen, DE \\
e-mail: brdh@math.ku.dk}\\
\end{center}
\vspace{1cm}
\begin{center}
\textnormal{ \large Department of Applied Mathematics \\
National Yang-Ming Chiao Tung University \\
Hsinchu, Taiwan 30010 \\
e-mail: yekaiwang@nycu.edu.tw}\\
\end{center}

\end{document}